\documentclass[a4paper,11pt]{article}

\usepackage[utf8]{inputenc} 

\usepackage[english]{babel}

\usepackage[T1]{fontenc}   

\usepackage{mathpazo}

\usepackage{graphicx}

\usepackage{comment}

\usepackage{amsthm,amsmath,amsfonts,stmaryrd,mathrsfs,amssymb}

\usepackage{mathtools}

\usepackage{enumerate}

\usepackage{csquotes}

\usepackage{chemarrow}

\usepackage{tikz}

\usepackage[top=2.5cm, bottom=2.5cm, left=2cm, right=2cm]{geometry}

\usepackage{subfig}

\usepackage{multirow}

\usepackage{array}

\usepackage{bigints}

\usepackage{etoolbox}

\usepackage{mdframed}

\usepackage{color}
\usepackage{tabularx}

\usepackage[unicode,colorlinks=true,linkcolor=blue,citecolor=red,pdfencoding=auto,psdextra]{hyperref}

\makeatletter

\newcommand{\mylabel}[2]{#2\def\@currentlabel{#2}\label{#1}}

\makeatother

\DeclareMathOperator{\haut}{ht}
\DeclareMathOperator{\diam}{diam}

\DeclareMathOperator{\GEM}{GEM}

\DeclareMathOperator{\MLMC}{MLMC}

\DeclareMathOperator{\BGW}{BGW}

\newcommand{\ensemblenombre}[1]{\mathbb{#1}}
\newcommand{\N}{\ensemblenombre{N}}
\newcommand{\Z}{\ensemblenombre{Z}}

\newcommand{\R}{\ensemblenombre{R}}
\newcommand{\intervalle}[4]{\mathopen{#1}#2
	\mathclose{}\mathpunct{},#3
	\mathclose{#4}}
\newcommand{\intervalleff}[2]{\intervalle{[}{#1}{#2}{]}}

\newcommand{\intervallefo}[2]{\intervalle{[}{#1}{#2}{)}}
\newcommand{\intervalleoo}[2]{\intervalle{(}{#1}{#2}{)}}
\newcommand{\intervalleentier}[2]{\intervalle\llbracket{#1}{#2}
	\rrbracket}
\newcommand{\enstq}[2]{\left\lbrace#1\mathrel{}\middle|\mathrel{}#2\right\rbrace}
\newcommand{\abs}[1]{\left\lvert#1\right\rvert}
\newcommand{\Gam}[1]{\Gamma \left(#1\right)}
\newcommand{\Cut}[2]{\mathrm{Cut}\left(#1,#2\right)}

\newcommand{\tend}[2]{\underset{#1}{\overset{#2}{\longrightarrow}}}

\renewcommand{\P}{\mathbb{P}}

\newcommand{\bU}{\mathbb{U}}
\newcommand{\bT}{\mathbb{T}}
\newcommand{\bA}{\mathbb{A}}
\newcommand{\bt}{\mathbf{t}}

\newcommand{\sG}{\mathscr{G}}

\newcommand{\Ec}[1]{\mathbb{E} \left[#1\right]}

\newcommand{\Pp}[1]{\mathbb{P} \left(#1\right)}

\newcommand{\Ecsq}[2]{\mathbb{E} \left[#1\mathrel{}\middle|\mathrel{}#2\right]}

\newcommand{\Ppsq}[2]{\mathbb{P} \left(#1\mathrel{}\middle|\mathrel{}#2\right)}

\newcommand{\ndN}{\mathbb{N}}

\renewcommand{\Pr}[1]{\mathbb{P}(#1)}

\newcommand{\Ex}[1]{\mathbb{E}[#1]}

\newcommand{\Va}[1]{\mathbb{V}[#1]}

\newcommand{\cL}{\mathcal{L}}
\newcommand{\cM}{\mathcal{M}}

\newcommand{\cB}{\mathcal{B}}

\newcommand{\cS}{\mathcal{S}}

\newcommand{\cT}{\mathcal{T}}

\newcommand{\cA}{\mathcal{A}}
\newcommand{\cH}{\mathcal{H}}
\newcommand{\cD}{\mathcal{D}}

\newcommand{\mM}{\mathsf{M}}

\newcommand{\KT}{\mathbf{T}}

\newcommand{\UHT}{\mathbb{U}}

\newcommand{\eqdist}{\,{\buildrel d \over =}\,}

\newcommand{\convdis}{\,{\buildrel d \over \longrightarrow}\,}
\newcommand{\convp}{\,{\buildrel p \over \longrightarrow}\,}

\newtheorem{theorem}{Theorem}[section]

\newtheorem{corollary}[theorem]{Corollary}
\newtheorem{proposition}[theorem]{Proposition}
\newtheorem{lemma}[theorem]{Lemma}

\newtheorem{remark}[theorem]{Remark}

\newtheorem{question}[theorem]{Question}

\numberwithin{equation}{section}

\newcommand{\dres}{\mathbf{d}}

\newcommand{\ben}[1]{{\textcolor{blue}{B: #1}}}

\newcommand{\Bb}{\mathbb{B}}
\newcommand{\ind}[1]{\mathbf{1}_{\left\lbrace #1 \right\rbrace}}  

\newcommand{\out}[1]{d_{#1}^+}
\newcommand{\dec}{\text{dec}}

\title{\textbf{Decorated stable trees}}
\date{}

\author{Delphin S\'{e}nizergues\thanks{University of British Columbia, E-mail: senizergues@math.ubc.ca} \and Sigurdur \"Orn Stef\'ansson\thanks{University of Iceland, E-mail: sigurdur@hi.is} \and
Benedikt Stufler\thanks{Vienna University of Technology, E-mail: benedikt.stufler@tuwien.ac.at}}

\begin{document}
	
	\maketitle
	
	\let\thefootnote\relax\footnotetext{ \\\emph{MSC2010 subject classifications}. 60F17, 60C05  }
	
	\vspace {-0.5cm}

\begin{abstract}
We define \emph{decorated $\alpha$-stable trees} which are informally obtained from an $\alpha$-stable tree by blowing up its branchpoints into random metric spaces. This generalizes the $\alpha$-stable looptrees of Curien and Kortchemski, where those metric spaces are just deterministic circles.
We provide different constructions for these objects, which allows us to understand some of their geometric properties, including compactness, Hausdorff dimension and self-similarity in distribution. 
We prove an invariance principle which states that under some conditions, analogous discrete objects, random decorated discrete trees, converge in the scaling limit to decorated $\alpha$-stable trees.
We mention a few examples where those objects appear in the context of random trees and planar maps, and we expect them to naturally arise in many more cases.  
\end{abstract}

\begin{figure}[h]
	\centering
	\begin{minipage}{\textwidth}
		\centering
		\includegraphics[width=1.0\linewidth]{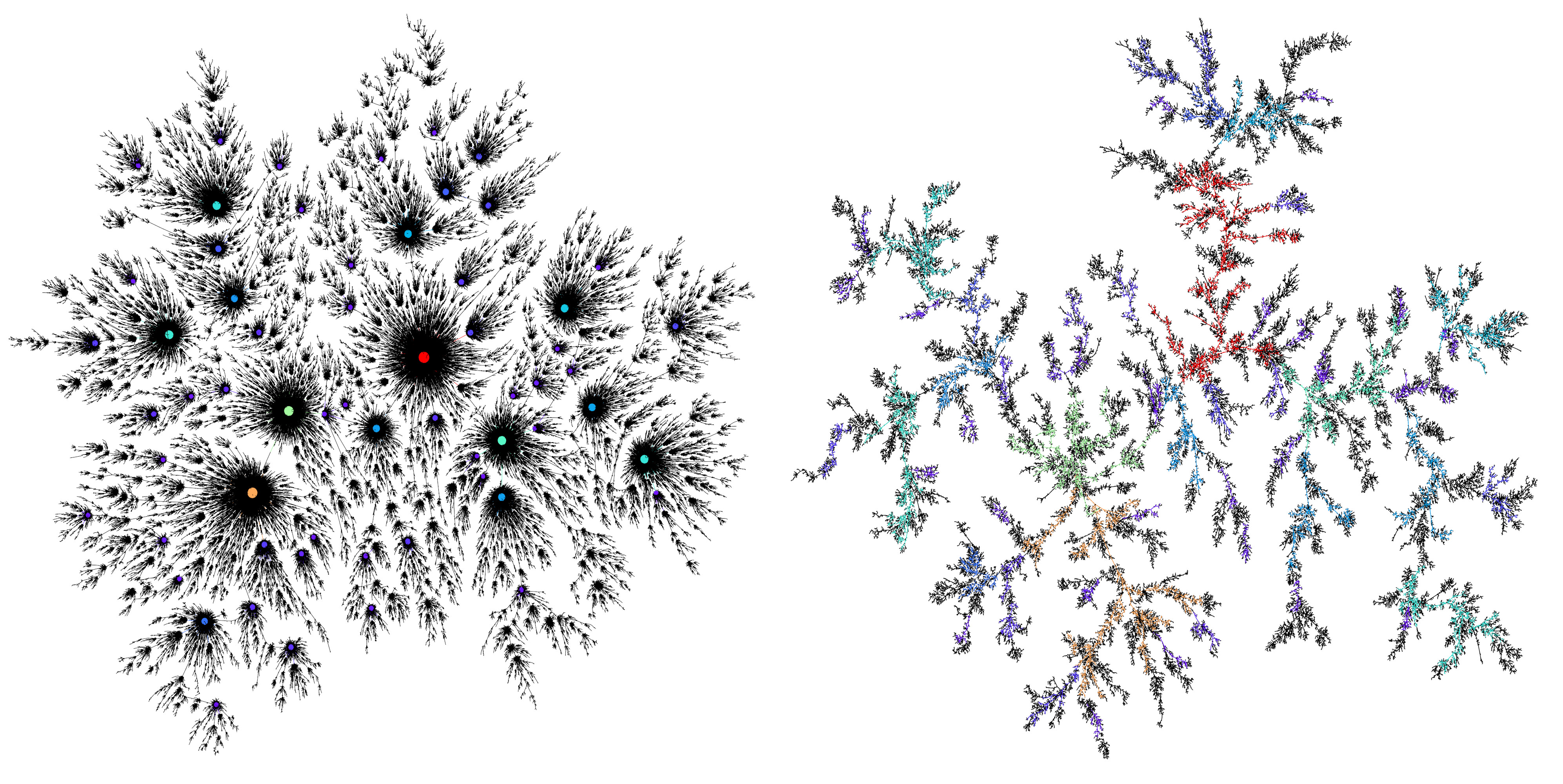}
		\caption{
		The right side illustrates a decorated $\alpha$-stable tree for $\alpha=5/4$. The decorations are Brownian trees (2-stable trees) rescaled according to the ``width'' of branchpoints. The left side shows the corresponding $\alpha$-stable tree before blowing up its vertices. Colours mark vertices with high width and their corresponding decorations.}
		\label{fi:exiterated}
\end{minipage}
\end{figure}

\section{Introduction}

In this paper we develop a general method for constructing random metric spaces by gluing together metric spaces referred to as \emph{decorations} along a tree structure which is referred to as \emph{the underlying tree}. 
The resulting object, which we define informally below and properly in Section~\ref{s:disc_construction} and Section~\ref{s:cts_construction}, will be called a \emph{decorated tree}. 
We will consider two frameworks, which we refer to as the \emph{discrete case} and the \emph{continuous case}. 

In the discrete case the underlying tree is a finite, rooted plane tree $T$. 
To each vertex $v$ in $T$ with outdegree $d_T^+(v)$ there is associated a decoration $B(v)$ which is a metric space with one distinguished point called the inner root and $d_T^+(v)$ number of distinguished points called the outer roots. 
In applications, the decorations $B(v)$ may e.g.~be finite graphs and they will only depend on the outdegree of $v$. 
We will thus sometimes write $B(v) = B_{d^+_T(v)}$. 
The decorated tree is constructed by identifying, for each vertex $v$, each of the outer roots of $B(v)$ with the inner root of $B(vi)$, $i=1,\ldots,d^+_T(v)$, where $vi$ are the children of $v$. 
This identification will either be done at random or according to some prescribed labelling. 
Our aim is to describe the asymptotic behaviour of \emph{random} spaces that are constructed in this way. 
To this end, let $(T_n)_{n \ge 1}$ be a sequence of random rooted plane trees. 
The sequence is chosen in such a way that the exploration process of the random trees, properly rescaled,  converges in distribution towards an excursion of an $\alpha$-stable Lévy process with $\alpha \in (1,2)$. 
This particular choice is motivated by applications as will be explained below. 
We will further assume that there is a parameter $0 < \gamma \leq 1$, which we call the \emph{distance exponent}, such that the (possibly random) decorations $B_k$, with the metric rescaled by $k^{-\gamma}$, converge in distribution towards a compact metric space $\mathcal{B}$ as $k\to \infty$, in the Gromov--Hausdorff sense.

In the continuous case, we aim to construct a decorated tree which is the scaling limit of the sequence of discrete decorated trees described above. 
The underlying tree is chosen as the $\alpha$-stable tree $\mathcal{T}_\alpha$, with $\alpha \in (1,2)$, which is a random compact real tree. 
One important feature of the $\alpha$-stable tree is that it has countably many vertices of infinite degree. 
The size of such a vertex is measured by a number called its \emph{width} which plays the same role as the degree in the discrete setting. 
The continuous decorated tree is constructed in a way similar as in the discrete case, where now each vertex $v$ of infinite degree in $\mathcal{T}_\alpha$ is independently replaced by a decoration $\mathcal{B}_v$ distributed as $\mathcal{B}$, in which distances are rescaled by width of $v$ to the distance exponent $\gamma$.  
The gluing points are specified by singling out an inner root and sampling outer roots according to a measure $\mu$ and the decorations are then glued together along the tree structure.
See Figure~\ref{fi:exiterated} for an example of a decorated tree.
The precise construction is somewhat technical and is described in more detail in Section~\ref{ss:realtrees}. 
In order for the construction to yield a compact metric space we require some control on the size of the decorations $\mathcal{B}_v$. 
More precisely, we will assume that $\gamma > \alpha-1$  and that there exists $p>\frac{\alpha}{\gamma}$ such that $\Ec{\diam(\mathcal{B})^p}<\infty$. 
Under some further suitable conditions on the decorations and the underlying tree which are stated in Section~\ref{ss:conditions} we arrive at the following main results:	

\begin {itemize}
\item \emph{Construction of the decorated $\alpha$-stable tree:} The decorated $\alpha$-stable tree defined above is a well defined random compact metric space, see Section~\ref{s:cts_construction} and Theorem~\ref{thm:compact_hausdorff}.

\item \emph{Invariance principle:} The sequence of discrete random decorated trees, with a properly rescaled metric, converges in distribution towards a decorated $\alpha$-stable tree in the Gromov--Hausdorff--Prokhorov sense. 
See Theorem~\ref{th:invariance} for the precise statement.

\item \emph{Hausdorff dimension:} Apart from degenerate cases, the Hausdorff dimension of \emph {the leaves} of the decorated $\alpha$-stable tree equals $\alpha/\gamma$, see Theorem~\ref{thm:compact_hausdorff}.  
It follows that if the decorations a.s.~have a Hausdorff dimension $d_\mathcal{B}$ then the Hausdorff dimension of the decorated $\alpha$-stable tree equals $\max\{d_\mathcal{B},\alpha/\gamma\}$.

\end{itemize}
\begin{remark}
	A heuristic argument for the condition $\alpha -1 < \gamma$ is the following. If $(T_n)$ is in the domain of attraction of $\cT_\alpha$ then the maximum degree of a vertex in $T_n$ is of the order $n^{1/\alpha}$ and the height of $T_n$ is of the order $n^{\frac{\alpha-1}{\alpha}}$ (up to slowly varying functions). A decoration on a large vertex will then have a diameter of the order $n^{\gamma/\alpha}$. If $\alpha-1 < \gamma$, the diameter of the decorations will thus dominate the height of $T_n$ and the decorated tree will have height of the order $n^{\gamma/\alpha}$, coming only from the decorations. Furthermore, the decorations will "survive" in the limit since the distances of the decorated tree will be rescaled according to the inverse of its height which is of the order of the diameter of the decorations.
	 
	 In the case $\alpha -1 > \gamma$, the decorations will have a diameter which is of smaller order than the height of $T_n$. This means that the proper rescaling of the distances in the decorated tree should be the inverse of the height of $T_n$, i.e.~$n^{-\frac{\alpha-1}{\alpha}}$ which would contract the decorations to points.
	  We therefore expect that in this case, the scaling limit of the decorated tree is equal to a constant multiple of $\cT_\alpha$ which has a Hausdorff dimension $\alpha/(\alpha-1)$.  We we do not go through the details in the current work.
	  
	  We do not know what happens exactly at the value $\alpha-1 = \gamma$ at which there is a phase transition in the geometry and the Hausdorff dimension of the scaling limit. We leave this as an interesting open case.
\end{remark}

\subsection{Relations to other work}

The only current explicit example of decorated $\alpha$-stable trees are the $\alpha$-stable looptrees which were introduced by Curien and Kortchemski \cite{MR3286462}. The looptrees are constructed by decorating the $\alpha$-stable trees by deterministic circles. The authors further showed, in a companion paper \cite{MR3405619}, that the looptrees appear as boundaries of percolation clusters on random planar maps and there are indications that they play a more general role in statistical mechanical models on random planar maps \cite{Richier:2017}. Looptrees have also been shown to arise as limits of various models of discrete decorated trees such as discrete looptrees, dissections of polygons \cite{MR3286462} and outerplanar maps \cite{zbMATH07138334}. In these cases the discrete decorations $B_k$ are graph cycles, chord restricted dissections of a polygon and dissections of a polygon respectively.

The idea of decorating $\alpha$-stable trees by more general decorations was originally proposed by Kortchemski and Marzouk in \cite{kortchemski2016}:

\begin{displayquote} "Note that if one starts with a stable tree and explodes each branchpoint by simply gluing inside
	a “loop”, one gets the so-called stable looptrees .... More
	generally, one can imagine exploding branchpoints in stable trees and glue inside any compact
	metric space equipped with a homeomorphism with [0, 1]."
\end{displayquote}
They mention explicitly an example where an $\alpha_1$-stable tree would be decorated by an $\alpha_2$-stable tree and so on and ask the question what the Hausdorff dimension of that object would be. We give a partial answer to their question in Section~\ref{s:markedtrees}. See Figure \ref{fi:exiterated} for a simulation where $\alpha_1 = 5/4$ and $\alpha_2 = 2$. 

Archer studied a simple random walk on decorated Bienaymé--Galton--Watson trees in the domain of attraction of the $\alpha$-stable tree \cite{arXiv:2011.07266}.
Although her results do not explicitly involve constructing a Gromov--Hausdorff limit of the decorated trees, they do require an understanding of their global metric properties. 
We refer further to her work for examples of decorated trees, some of which we do not discuss in the current work.

 S\'{e}nizergues has developed a general method for gluing together sequences of random metric spaces \cite{senizergues_gluing_2019} and introduced a framework where metric spaces are glued together along a tree structure \cite{senizergues_growing_2022}.  Section~\ref{sec:self-similarity} relies heavily on the results of the latter.  

Let us also mention that Stufler has studied local limits and Gromov--Hausdorff limits of decorated trees which are in the domain of attraction of the Brownian continuum random tree (equivalently the 2-stable tree) \cite{StEJC2018, zbMATH07235577}. 
The Brownian tree has no vertices of infinite degree and thus the rescaled decorations contract to points in the limit which only changes the metric of the decorated tree by a multiplicative factor. 
Its Gromov--Hausdorff limit then generically turns out to be a multiple of the Brownian tree itself (see \cite[Theorem~6.60]{zbMATH07235577}).

\subsection{Outline} The paper is organized as follows.
In Section~\ref{s:preliminaries} we recall some important definitions. 
In Section~\ref{s:disc_construction} the construction of discrete decorated trees is provided and the required conditions on the underlying tree and decorations are stated. 
In Section~\ref{s:cts_construction} we give the definition of the continuous decorated trees and in Section~\ref{s:invariance} we state and prove the invariance principle, namely that the properly rescaled discrete decorated trees converge towards the continuous decorated trees in the Gromov--Hausdorff--Prokhorov sense. 
An alternative construction of the continuous decorated tree is given in Section~\ref{sec:self-similarity} which allows us to prove that it is a compact metric space and gives a way to calculate its Hausdorff dimension. Section~\ref{s:applications} provides several applications of our results. 
Finally, the Appendix is devoted to proving that the conditions on the discrete trees stated in Section~\ref{s:disc_construction} are satisfied for natural models of random trees, in particular Bienaymé--Galton--Watson trees which are in the domain of attraction of $\alpha$-stable trees, with $\alpha \in (1,2)$.

\subsection*{Notation}
We  let $\mathbb{N} = \{1, 2, \ldots\}$ denote the set of positive integers. 
We usually assume that all considered random variables are defined on a common probability space whose measure we denote by $\mathbb{P}$. 
All unspecified limits are taken as $n$ becomes large, possibly taking only values in a subset of the natural numbers.  
We write $\convdis$ and $\convp$ for convergence in distribution and probability, and $\eqdist$ for equality in distribution. An event holds with high probability, if its probability tends to $1$ as $n \to \infty$.
We let $O_p(1)$ denote an unspecified random variable $X_n$ of a stochastically bounded sequence $(X_n)_n$, and write $o_p(1)$ for a random variable $X_n$ with $X_n \convp 0$.  
The following list summarizes frequently used terminology.

\begin{center}

\begin{tabularx}{\linewidth}{lll}
\qquad \qquad &$\out{T}(v)$ & outdegree of vertex $v$ in a tree $T$.\\
&$|T|$ & the number of vertices of a tree $T$.\\
&$B$ & decoration in the discrete setting.  \\
&$B_k$ & decoration of size $k$ in the discrete setting. \\
&$d_k$ & metric on $B_k$. \\
&$\dres_k = k^{-\gamma} d_k$ & rescaled metric on $B_k$. \\
&$\cB$ & decoration in the continuous case.\\
&$\cB_t$ & decoration in the continuous case indexed by $t\in[0,1]$.\\
&$\diam(\cB)$ & the diameter of the metric space $\cB$.\\
&$T$ & discrete tree.\\
&$T_n$ & discrete tree of size $n$.\\
&$\mathcal{T}$ & real tree.\\
&$\mathcal{T}_\alpha$ & the $\alpha$-stable tree.\\
&$\Bb$ & set of branchpoints of $\mathcal{T}_\alpha$. \\
&$X^{(\alpha)}$ & a spectrally positive, $\alpha$-stable Lévy process.\\
&$X^{\text{exc},(\alpha)}$ & an excursion of a spectrally positive, $\alpha$-stable Lévy process.
\end{tabularx}

\end{center}

\paragraph*{Acknowledgements.}
The authors are grateful to Nicolas Curien for suggesting this collaboration, and to Eleanor Archer for discussing some earlier version of her project \cite{arXiv:2011.07266}. 
The second author acknowledges support from the Icelandic Research Fund, Grant Number: 185233-051, and is
grateful for the hospitality at Université Paris-Sud Orsay.

\section{Preliminaries} \label{s:preliminaries}
In this section we collect some definitions and refer to the key background results required in the current work. We start by introducing plane trees, then real trees and the $\alpha$-stable tree and finally we recall the definition of the Gromov--Hausdorff--Prokhorov topology which is the topology we use to describe the convergence of compact mesasured metric spaces.

\subsection{Plane trees} \label{subsec:plane trees}
A rooted plane tree is a subgraph of the infinite Ulam--Harris tree, which is defined as follows. 
Its vertex set, $\UHT$,  is the set of all finite sequences 
\begin{align*}
\UHT = \bigcup _{n\geq 0} \mathbb{N}^n
\end{align*}
where $\mathbb{N}= \{1, 2,\ldots\}$ and  $\mathbb{N}^0 = \{\varnothing\}$ contains the empty sequence  denoted by $\varnothing$. If $v$ and $w$ belong to $\UHT$, their concatenation is denoted by $vw$.  The edge set consists of all pairs $(v,vi)$, $v\in\UHT$, $i\in\mathbb{N}$. The vertex $vi$ is said to be the $i$-th child of $v$ and $v$ is called its parent. In general, a vertex $vw$ is said to be a descendant of $v$ if $w \neq \varnothing$, and in that case $v$ is called its ancestor. This ancestral relation is denoted by $v \prec w$. We write $v \preceq w$ if either $v=w$ or $v \prec w$.   We denote the most recent comment ancestor of $v$ and $w$ by $v\wedge w$, with the convention that $v\wedge w = v$ if $v \preceq w$ (and $v \wedge w = w$ if $w \preceq v$).

A rooted plane tree $T$ is a finite subtree of the Ulam--Harris tree such that

\begin{enumerate}
	\item $\varnothing$ is in $T$, and is called its root. 
	\item If $v$ is in $T$ then all its ancestors are in $T$. 	
	\item If $v$ is in $T$ there is a number $\out{T}(v)\geq 0$, called the outdegree of $v$, such that $vi$ is in $T$ if and only if $1 \leq i \leq \out{T}(v)$. 
\end{enumerate}
In the following we will only consider rooted plane trees. For simplicity, we will refer to them as trees. The number of vertices in a tree $T$ will be denoted by $|T|$ and we denote them, listed in lexicographical order, by $v_1(T), v_1(T),\ldots, v_{|T|}(T)$. When it is clear from the context, we may leave out the argument $T$ for notational ease. 

To the tree $T$ we associate its \L{}ukasiewicz path $(W_{k}(T))_{0\leq k \leq |T|}$ by $W_0(T) = 0$ and 
\begin{align}
W_k(T) = \sum_{j=1}^{k} (\out{T}(v_j)-1)
\end{align}
for $1 \leq k \leq |T|$. 

Throughout the paper, we will consider sequences of \emph{random trees} $(T_n)_{n\geq 1}$. The index $n$ will usually refer to the size of the tree in some sense, e.g.~the number of vertices or the number of leaves. In Section~\ref{ss:conditions} we state the general conditions on these random trees. However, the prototypical example to keep in mind is when $T_n$ is a Bienaymé--Galton--Watson tree, conditioned to have $n$ vertices. In this case the \L{}ukasiewicz path $(W_{k}(T_n))_{0\leq k \leq n}$ is a random walk conditioned on staying non-negative until step $n$ when it hits $-1$.

\subsection{Real trees and the $\alpha$-stable trees} \label{ss:realtrees}

Real trees may be viewed as continuous counterparts to discrete trees. Informally, they are compact metric spaces obtained by gluing together line segments without creating loops. A formal definition  may e.g.~be found in \cite{MR2147221} but for our purposes we give an equivalent definition in terms of excursions on $[0,1]$, which are continuous functions $g:[0,1]\to \mathbb{R}$ with the properties that $g\geq 0$ and $g(0) = g(1) = 1$. For $s,t\in[0,1]$ define $\displaystyle m_g(s,t) = \min_{s\wedge t \leq r \leq s \vee t} g(r)$ end define the pseudo distance 
\begin{align*}
d_g(s,t) = g(s)+g(t) - 2 m_g(s,t).
\end{align*}
The real tree $\mathcal{T}^g$ encoded by the excursion $g$ is defined as the quotient
\begin{align*}
	\cT^g = [0,1]/\{d_g=0\}
\end{align*}
endowed with the metric inherited from $d_g$ which will still be denoted $d_g$ by abuse of notation.

We will consider a family of random real trees $\cT_\alpha$, $\alpha \in (1,2)$ called $\alpha$-stable trees, which serve as the underlying trees in the continuous random decorated trees. A brief definition of $\cT_\alpha$ is given below, but we refer to definitions and more details in \cite{MR1964956}. Let $(X^{(\alpha)}_t; t\in[0,1])$ be a spectrally positive $\alpha$-stable Lévy process, with $\alpha \in (1,2)$, normalized such that
\begin{align} \label{eq:levymeasure}
	\mathbb{E}(\exp(-\lambda X_t)) = \exp(t\lambda^\alpha)
\end{align}
for every $\lambda > 0$. Let $(X^{\text{exc},(\alpha)}_t; t\in[0,1])$ be an excursion of $X_t^{(\alpha)}$, i.e.~the process conditioned to stay positive for $t \in (0,1)$ and to be 0 for $t\in\{0,1\}$. This involves conditioning on an event of probability zero which may be done via approximations, see e.g.~\cite[Chapter~VIII.3]{MR1406564}. Define the associated \emph{continuous height process}
\begin{align*}
	H_t^{\text{exc},(\alpha)} = \lim_{\epsilon\to 0}	 \frac{1}{\epsilon}\int_{0}^t  \mathbf{1}_{\{X_s^{\text{exc}}<I_s^t + \epsilon\}}ds
\end{align*}
where 
\begin{align*}
	I_s^t = \inf_{[s,t]} X^{\text{exc}}.
\end{align*}
Continuity of $H_t^{\text{exc},(\alpha)}$ is not guaranteed by the above definition but there exists a continuous modification which we consider from now on \cite{10.1214/aop/1022855417}.
The $\alpha$-stable tree is the random real tree $\cT_\alpha := \cT^{H^{\text{exc},(\alpha)}}$ encoded by $H^{\text{exc},(\alpha)}$. We refer to \cite{MR1964956,MR1954248,MR2147221} for more details on the construction and properties of $\cT_\alpha$.

Let $T^{\BGW}_n$ be a random BGW tree with a critical offspring distribution in the domain of attraction of an $\alpha$-stable law with $\alpha\in (1,2)$. 
Recall that this example is the prototype for the underlying random trees we consider in the current work. Since the \L ukasiewicz path $W_n(T^{\BGW}_n)$ is in this case a conditioned random walk, it follows that by properly rescaling $(W_{\lfloor n t \rfloor}(T^{\BGW}_n); t\in[0,1])$, it converges in distribution towards the excursion $(X^{\text{exc},(\alpha)}_t; t\in[0,1])$ in the Skorohod space $D([0,1],\mathbb{R})$, \cite[Chapter~VIII]{MR1406564}. 
This observation is the main idea in proving that $T_n^{\BGW}$, with a properly rescaled graph metric, converges towards $\cT_\alpha$ in the Gromov--Hausdorff sense \cite{MR1964956,MR2147221}. The pointed GHP-distance may and will be generalized in a straightforward manner by adding more points and more measures.

\subsection{The Gromov-Hausdorff-Prokhorov topology}

Let $\mathbb{M}$ be the set of all compact measured metric spaces modulo isometries.
An element $(X,d_X,\nu_X) \in \mathbb{M}$ consists of the space $X$, a metric $d_X$ and a Borel-measure $\nu_X$.  
In particular, $\mathbb{M}$ contains finite graphs (e.g.~with a metric which is a multiple of the graph metric)  and real trees. 
The Gromov--Hausdorff--Prokhorov (GHP) distance between two elements $(X,d_X,\nu_X),(Y,d_Y,\nu_Y) \in \mathbb{M}$ is defined by
\begin{align} \label{eq:dghp}
	d_{GHP}(X,Y) = \inf\left\{d_H^Z(\phi_1(X),\phi_2(Y)) + d_P^Z(\phi_1^\ast \nu_X,\phi_2^\ast \nu_Y)\right\}
\end{align}
where the infimum is taken over all measure preserving isometries $\phi_1: X \to Z$ and $\phi_2: Y \to Z$ and metric spaces $Z$. 
Here $d_H^Z$ is the Hausdorff distance between compact metric spaces in $Z$, $d_P^Z$ is the Prokhorov distance between measures on $Z$ and $\phi^\ast \nu$ denotes the pushforward of the measure $\nu$ by $\phi$.  
We refer to \cite[{Chapter~27}]{MR2459454} for more details. 
When we are only interested in the convergence of the metric spaces without the measures, we leave out the term in \eqref{eq:dghp} with the Prokhorov metric and refer to the metric simply as the Gromov--Hausdorff (GH) metric, denoted $d_{GH}$.

The following formulation of GHP-convergence may be useful in calculations and we will use it repeatedly in the following sections.
If $(X,d_X)$ and $(Y,d_Y)$ are metric spaces and $\epsilon>0$, a function $\phi_\epsilon: X \to Y$ is called an $\epsilon$-isometry if the following two conditions hold:
\begin{enumerate}
	\item (Almost isometry) $$\sup_{x_1,x_2\in X}|d_Y(\phi(x_1),\phi(x_2)) - d_X(x_1,x_2)| < \epsilon.$$
	\item (Almost surjective) For every $y \in Y$ there exists $x \in X$ such that $$d_Y(\phi(x),y) < \epsilon.$$
\end{enumerate}
If $(X_n,d_n,\nu_n)$ is a sequence of compact measured metric spaces which converge towards $(X,d,\nu)$ in the GHP topology then it is equivalent to the existence of  $\epsilon_n$-isometries $\phi_{\epsilon_n}: X_n \to X$, where $\epsilon_n \to 0$, such that
\begin{align} \label{eq:epsiloniso}
	d_H^X(\phi_{\epsilon_n}(X_n),X) \to 0 \quad \text{and} \quad  d_P^X(\phi_{\epsilon_n}^\ast \nu_n, \nu) \to 0
\end{align}
as $n\to \infty$. When the GH topology is considered, only the first convergence of these two is assumed.

In most of the applications we will consider \emph{pointed} metric spaces $(X,d_X,\rho_X)$ where $\rho_X$ is a distinguished element called  \emph{a point} or \emph{a root}. The GHP-distance is adapted to this by requiring that the isometries in \eqref{eq:dghp} and the $\epsilon_n$-isometries in \eqref{eq:epsiloniso}, send root to root. 
In this case we speak of the \emph{pointed} GHP-distance.
In the same vein, we will sometimes deal with metric spaces that carry some extra finite measure. In this case we just add a term in \eqref{eq:dghp} that corresponds to the Prokhorov distance computed between the extra measures and modify \eqref{eq:epsiloniso} accordingly.

\section{Construction in the discrete} \label{s:disc_construction}

In this section we define the discrete decorated trees and state the sufficient conditions on the decorations and the underlying trees for our main results to hold. 

\emph{A decoration} of size $k$ is a compact metric space $B$ equipped with a metric $d$. 
We distinguish a vertex $\rho$ and call it the \emph{internal root} of $B$ and furthermore, we label $k$ vertices (not necessarily different)  and call them the external roots of $B$. 
This labeling is described by a function $\ell: \{1,\ldots,k\} \to B$. 
Note that we explicitly allow external roots to coincide with the internal root.
 The space $B$ is equipped with a finite Borel measure $\nu$. 
 We will use the notation $(B,d,\rho,\ell,\nu)$ for a decoration. 
 In practice, $B$ will often be a finite graph, $d$ the graph metric, and $\nu$ the counting measure on the vertex set. 
 The labeled vertices will often be the boundary of $B$ in some sense, e.g.~the leaves of a tree or the vertices on the boundary of a planar map. 
\begin{figure}[t]
	\centering
	\begin{minipage}{\textwidth}
		\centering
		\includegraphics[width=0.55\linewidth]{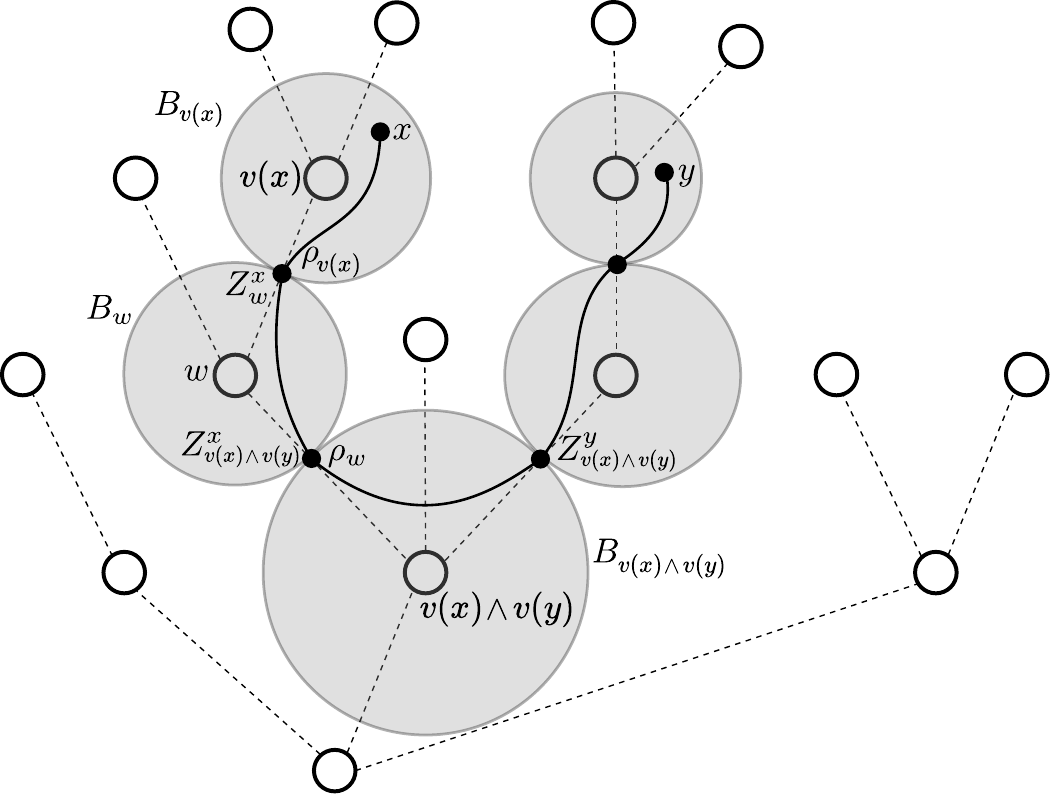}
		\caption{Definition of the metric $d_n^{\dec}$.}
		\label{f:metric}
	\end{minipage}
\end{figure}
Let $(T_n)_{n\geq 1}$ be a sequence of random trees and $(\tilde B_k,\tilde \rho_k,\tilde d_k,\tilde \ell_k,\tilde \nu_k)_{k\geq 0}$ a sequence of random decorations indexed by their size. 
Conditionally on $T_n$, for each $v \in T_n$, sample independently  a decoration $(B_{v},\rho_{v},d_{v},\ell_{v},\nu_{v})$ distributed as $(\tilde B_{\out{}(v)},\tilde \rho_{\out{}(v)},\tilde d_{\out{}(v)},\tilde \ell_{\out{}(v)},\tilde \nu_{\out{}(v)})$.  
We aim to glue the decorations to each other along the structure of $T_n$ which is informally explained as follows. 
Consider the disjoint union 
\begin{align*}
T_n^{\ast,\dec} = \bigsqcup_{v\in T_n} B_v
\end{align*}
of decorations on which is defined an equivalence relation $\sim$, such that for each $v$ in $T_n$, and each $1\leq i \leq \out{}(v)$ we have $\ell_v(i) \sim \rho_{vi}$. In other words, the $i$-th external root of $B_v$ is glued to the internal root of $B_{vi}$. The precise construction is given below.

We first define a pseudo-metric $d_n^{\ast,\dec}$ on $T_n^{\ast,\dec}$. 
For each $x \in T_n^{\ast,\dec}$ we let $v(x)$ denote the unique vertex in $T_n$ such that $x\in B_{v(x)}$. 
If $w$ is a vertex in $T_n$ such that $w \preceq v(x)$, let $Z_w^x \in B_w$ be defined by
\begin{align*}
Z_w^x = \begin{cases}
\ell_w(i) & \text{if $wi \preceq v(x)$,}\\
x & \text{if $w = v(x)$}.
\end{cases}
\end{align*}
In order to clarify the first case on the right hand side, note that if $w \prec v(x)$ then there is a unique child $wi$, $1 \le i \le \out{T_n}(w)$, of $w$ in $T_n$ such that $wi \preceq v(x)$, and we set $Z_w^x = \ell_w(i)$.

Now, let $x,y \in T_n^{\ast,\dec}$ and first assume that $v(x) \prec v(y)$ and set
\begin{align*}
d_n^{\ast,\dec}(x,y) &= d_n^{\ast,\dec}(y,x) = d_{v(x)}\left(x,Z_{v(x)}^{y}\right)+\sum_{v(x) \prec w \preceq y} d_w\left(\rho_w,Z_w^{y}\right).
\end{align*}
Otherwise, if $v(x)$ and $v(y)$ are not in an ancestral relation, set
\begin{align}\label{eq:definition distance dec n}
d_n^{\ast,\dec}(x,y) = d_{v(x)\wedge v(y)}\left(Z_{v(x)\wedge v(y)}^{x},Z_{v(x)\wedge v(y)}^{y}\right) + d_n^{\ast,\dec}\left(Z_{v(x)\wedge v(y)}^{x},x\right) + d_n^{\ast,\dec}\left(Z_{v(x)\wedge v(y)}^{y},y\right). 
\end{align}
These definitions are clarified in Fig.~\ref{f:metric}. Finally, the decorated tree is the quotient
\begin{align}
T_n^{\dec} = T_n^{\ast,\dec}\Big/\sim
\end{align}
where $x\sim y$ if and only if $d_n^{\ast,\dec}(x,y) = 0$, and the metric induced by $d_n^{\ast,\dec}$ on the quotient is denoted by $d_n^{\dec}$. 

Define a measure $\nu^\dec_n$, on $T_n^\dec$ by
\begin{align*}
\nu_n^\dec = \sum_{v\in T_n} \nu_v. 
\end{align*}
Here the mass of a gluing point is the sum of the masses of all vertices of its equivalence class.
Finally, root $T_n^\dec$ at $\rho_\varnothing$. We view the the decorated tree as a random, compact, rooted and measured metric space
\begin{align*}
(T_n^\dec,d_n^\dec,\rho_\varnothing, \nu_n^\dec).
\end{align*}

\subsection{Conditions on the trees and decorations} \label{ss:conditions}
Assume in this subsection that $\alpha \in (1,2)$ is fixed and that $(T_n)_n$ is a sequence of random trees with decorations  $(B_v,d_v,\rho_v,\ell_v,\nu_v)_{v\in T_n}$ sampled independently according to $(\tilde B_k,\tilde \rho_{k},\tilde d_{k},\tilde \ell_{k},\tilde \nu_{k})_{k\geq 0}$.  Let $\mu_k$ be a measure on $\tilde B_k$ defined by
\begin{align*}
\mu_k = \frac{1}{k}\sum_{j=1}^{k}\delta_{\tilde\ell_k(j)}.
\end{align*}
We let $v_1,\dots v_{|T_n|}$ be the vertices of $T_n$ listed in depth-first-search order and for all $1\leq k \leq |T_n|$, we let
\begin{align*}
	M_k=\sum_{i=1}^{k}\nu_{v_i}(B_{v_i})
\end{align*}
be the sum of all the total mass of the $k$ first decorations in depth-first order. We impose the following conditions on the decorations (D), the trees (T) and on both trees and decorations (B).
\subsubsection*{Conditions on the decorations} \label{sss:decorations}

\begin{enumerate}[\quad D1)]
	\item \label{c:GHPlimit} (\emph{Pointed GHP limit}). There are  $\beta,\gamma > 0$, a compact metric space $(\mathcal{B},d)$ with a point $\rho$,  a Borel probability measure $\mu$ and a finite Borel measure $\nu$ such that
	\begin{align*}
	\left(\tilde B_k,\tilde\rho_k,k^{-\gamma} \tilde d_k, \mu_k,k^{-\beta}\tilde \nu_k\right) \to (\mathcal{B},\rho,d,\mu,\nu)
	\end{align*}
	as $k\to \infty$, weakly in the pointed Gromov--Hausdorff--Prokhorov sense. 
%
%
	\item \label{c:moment_diam_limit} (\emph{Moments of the diameter of the limit}) The limiting block $(\mathcal{B},\rho,d,\mu,\nu)$ satisfies $\Ec{\diam(\mathcal{B})^p}<\infty$, for some $p>\frac{\alpha}{\gamma}$.
	
	\item \label{c:mass_limit} (\emph{Limit of the expected mass}) It holds that  $$k^{-\beta}\cdot \Ec{\tilde \nu_k(\tilde B_k)} \underset{k\rightarrow \infty }{\rightarrow} \Ec{\nu(\mathcal{B})} <\infty.$$

\end{enumerate}

\subsubsection*{Conditions on the trees}\label{sss:tree}

\begin{enumerate}[\quad T1)]
	\item \label{c:permute} (\emph{Symmetries of subtrees}) The distribution of $T_n$ is invariant under the permutation of any set of siblings.
	\item \label{c:invariance} (\emph{Invariance principle for the \L ukasiewicz path}) There is a sequence $b_n \ge 0$ such that  
	\[
		(b_n^{-1} W_{\lfloor|T_n|t\rfloor}(T_n))_{0 \le t \le 1} \convdis  X^{\text{exc}, (\alpha)}
	\]
	where $(W_k(T_n))_{0\leq k \leq |T_n|}$ is the \L ukasiewicz path of $T_n$ and $X^{\text{exc},(\alpha)}$ is an excursion of the spectrally positive $\alpha$-stable Lévy-process defined by \eqref{eq:levymeasure}.
\end{enumerate}

\subsubsection*{Conditions on both}\label{sss:both}
Let $\gamma$ and $b_n$ be as in conditions T and D in the previous sections.
\begin{enumerate}[\quad B1)]
	\item \label{cond:small decorations dont contribute to distances} (\emph{Small decorations do not contribute}). For any $\epsilon>0$, 
	\begin{align*}
	\Pp{\sup_{v\in T_n} \left\lbrace b_n^{-\gamma}\sum_{w\preceq v} \diam(B_w) \ind{\out{T_n}(w)\leq \delta b_n}\right\rbrace>\epsilon } \underset{\delta\rightarrow 0}{\rightarrow}0,
	\end{align*}
	uniformly in $n\geq 1$ for which $T_n$ is well-defined.
	\item \label{cond:measure is spread out} (\emph{Measure is spread out, case $\beta\leq \alpha$}). We assume that we have the following uniform convergence in probability, for some sequence $(a_n)$,
	\begin{align}\label{eq:convergence depth-first mass}
	\left(\frac{M_{\lfloor |T_n|t\rfloor}}{a_n}\right)_{t \in \intervalleff{0}{1}} \underset{n\rightarrow\infty}{\longrightarrow} (t)_{0  \leq t \leq 1}.
	\end{align} 
	\item \label{cond:small decorations dont contribute to mass} (\emph{Measure is on the blocks, case $\beta>\alpha$}). For any $\epsilon>0$ we have 
	\begin{align}
	\limsup_{n\rightarrow \infty} \Pp{b_n^{-\beta}\sum_{i=1}^{|T_n|}\nu_{v_i}(B_{v_i}) \ind{\out{T_n}(v_i)\leq \delta b_n} > \epsilon } \underset{\delta\rightarrow 0}{\longrightarrow} 0.
	\end{align} 
 In this case we write $a_n=b_n^\beta$. 
\end{enumerate}
\subsection{Reduction to the case of shuffled external roots} \label{ss:shuffle}
Given our Hypothesis T\ref{c:permute} on the tree $T_n$, we can always shuffle the external roots of the blocks that we use without changing the law of the object $T_n^\dec$ that we construct. 

More precisely, given the law of $(\tilde B_k,\tilde d_k,\tilde\rho_k,\tilde\ell_k,\tilde \nu_k)$ for any $k\geq 0$ we can define the block with shuffled external roots $(\hat B_k,\hat d_k,\hat\rho_k,\hat\ell_k,\hat\nu_k)$ as 
\begin{align*}
	(\hat B_k,\hat d_k,\hat\rho_k,\hat\nu_k)=(\tilde B_k,\tilde d_k,\tilde\rho_k,\tilde \nu_k) \quad \text{and} \quad \hat\ell_k(i)= \tilde\ell_k(\sigma(i)),
\end{align*}
where $\sigma$ is a uniform permutation of length $k$, independent of everything else. 
Then, assuming that Hypothesis~T\ref{c:permute} holds, the law of the object $\hat T_n^\dec$ constructed with underlying tree $T_n$ and decorations $(\hat B_k,\hat \rho_k,\hat d_k,\hat \ell_k,\hat \nu_k)_{k\geq 1}$ is the same as the that of the object $T_n^\dec$ constructed with underlying tree $T_n$ and decorations $(\tilde B_k,\tilde d_k,\tilde\rho_k,\tilde\ell_k,\tilde \nu_k)_{k\geq 0}$. 

We will require the following lemma which informally states that the sequence of shuffled external roots in a decoration $B_k$ converges weakly to an i.i.d. sequence of external roots in the GHP-limit $\cB$ of $B_k$

\begin{lemma} \label{l:shuffle}
	Let $(\tilde B_k,\tilde d_k,\tilde\rho_k,\tilde\ell_k,\tilde \nu_k)_{k\geq 1}$ be a (deterministic) sequence of decorations which converges to a decoration $(\cB,\rho,d,\mu,\nu)$ in the GHP-topology. Let $Y_1,Y_2,\ldots$ be a sequence of independent points in $\cB$ sampled according to $\mu$. Then there exists a sequence $\phi_k: \tilde B_k \to \cB$ of $\epsilon_k$-isometries, with $\epsilon_k \to 0$, such that
	\begin{align*}
		(\phi_k(\hat\ell_k(1)),\phi_k(\hat\ell_k(2))\ldots,\phi_k(\hat\ell_k(k)),\rho,\rho,\ldots) \convdis  (Y_1,Y_2,\ldots)
	\end{align*}
as $k\to \infty$.
\end{lemma}
\begin{proof}
First note that 
	\begin{align*}
		\left\lVert \mu_k^{\otimes k},\mathcal{L}\big(\hat{\ell}_k\big)\right\rVert_{TV} \to 0,
	\end{align*}
	as $k\to \infty$, where $\mathcal{L}(\hat{\ell}_k)$ denotes the law of the uniformly shuffled labels $\hat{\ell}_k$. Thus, asymptotically, we may work with the measure $\mu_n^{\otimes k}$ instead of the shuffled labels. The GHP-convergence guarantees existence of $\epsilon_k$-isometries $\phi_k:\tilde B_k \to \cB$ with $\epsilon_k\to 0$, such that the pushforward $\phi_k^\ast \mu_k$ converges to $\mu$ in distribution. It follows that $\phi_k^\ast \mu_k^{\otimes \mathbb{N}} \to \mu^{\otimes \mathbb{N}}$ in distribution which concludes the proof.
	
\end{proof}

In the rest of the paper, we are always going to assume that the external roots of the decorations that we use have been shuffled as above. 
Note that there can be situations where there is a natural ordering of the external roots e.g.~when the decorations are oriented circles, maps with an oriented boundary, trees encoded by contour functions and more.

\subsection{The case of Bienaymé--Galton--Watson trees} \label{ss:BGW}
One important case of application is when the tree $T_n$ is taken to be a Bienaymé--Galton--Watson (BGW) tree in the domain of attraction of an $\alpha$-stable tree, conditioned on having $n$ vertices.
If we denote by $\mu$ a reproduction law on $\mathbb Z_+$, this corresponds to having $\sum_{i=0}^{\infty}i \mu(i)= 1$ and the function
\begin{align*}
	x\mapsto\frac{1}{\mu([x,\infty))}
\end{align*} 
to be $\alpha$-regularly varying, see the Appendix for a discussion on regularly varying functions.
Note that then $T_n$ is possibly not defined for all $n\geq 1$; in that case we only consider values of $n$ for which it is. 

Condition T1 is satisfied for any BGW tree and T2 holds, thanks to results of Duquesne and Le Gall \cite{MR1964956,MR2147221}, for a sequence $(b_n)_{n\geq 1}$ that depends on the tail behaviour of $\mu([n,\infty))$. 
In the rest of this section, we assume that $\alpha\in(1,2)$ and $\gamma> \alpha -1$ are fixed and that we are working with BGW trees with reproduction distribution $\mu$ as above and corresponding sequence $(b_n)_{n\geq 1}$.

The results below concern properties of the diameter and the measure on the decorations and they guarantee that the remaining conditions in Section~\ref{ss:conditions} are satisfied. Their proofs are given in the Appendix.
We first state a result that ensures that small decorations don't contribute to the macroscopic distances.

\begin{proposition}\label{prop:small blobs don't contribute}
	Under the following moment condition on the diameter of the decorations
	\begin{align*}
	 \sup_{k\geq 0} \Ec{\left(\frac{\diam(\tilde B_k)}{k^\gamma}\right)^m}<\infty,
	\end{align*}
	for some $m>\frac{4\alpha}{2\gamma +1 -\alpha}$, condition B1 is satisfied.
\end{proposition}
\begin{remark}
	In \cite[Remark 1.5]{arXiv:2011.07266}, it is stated that in this context $b_n^{-\gamma}\cdot \diam(T_n^\dec)$ is tight 
	under an assumption that corresponds to having $m>\frac{\alpha}{\alpha -1}$ in the above statement. 
	Note that neither of those assumptions are optimal since $\frac{\alpha}{\alpha -1}$ and $\frac{4\alpha}{2\gamma -(\alpha-1)}$ are not always ordered in the same way for different values of $\alpha$ and $\gamma$. 
	In particular, our bound does not blow up for a fixed $\gamma$ and $\alpha\rightarrow 1$.
\end{remark}

Now we state another result that ensures that either condition B\ref{cond:measure is spread out} or B\ref{cond:small decorations dont contribute to mass} is satisfied in this setting. 
We let $D$ denote a random variable that has law given by the reproduction distribution $\mu$ of the $\BGW$ tree.

\begin{lemma}\label{lem:moment measure discrete block implies B2 or B3}
Assuming that D1 holds for some value $\beta>0$ and that
	\begin{align*}
	\sup_{k\geq 0} \Ec{\left(\frac{\nu_{k}(\tilde B_k)}{k^\beta}\right)}<\infty,
	\end{align*}
	then condition B\ref{cond:measure is spread out} or B\ref{cond:small decorations dont contribute to mass} is satisfied, depending on the value of $\beta$:
	\begin{itemize}
		\item If $\beta < \alpha$ then B\ref{cond:measure is spread out} holds with $a_n:=n \cdot \Ec{\nu_D(B_D)}$. 
		\item If $\beta=\alpha$ then B\ref{cond:measure is spread out} holds with
		\begin{equation}
			a_n:=\left\{
			\begin{aligned}
				&n \cdot \Ec{\nu_D(B_D)} \quad &\text{if}\quad  \Ec{D^\alpha} <\infty ,\\
		      &n \cdot \Ec{\nu(\cB)}\cdot \Ec{D^\alpha \ind{D\leq b_n}} \quad  &\text{if} \quad  \Ec{D^\alpha} =\infty,
			\end{aligned}
			\right.
		\end{equation}
		where in the second case, we further assume that $\sup_{k\geq 0} \Ec{\left(\frac{\nu_{k}(\tilde B_k)}{k^\beta}\right)^{1+\eta}}<\infty$ for some $\eta>0$.
		\item If $\beta > \alpha$ then B\ref{cond:small decorations dont contribute to mass} holds with $a_n = b_n^\beta$.
	\end{itemize}
\end{lemma}

\begin{remark}
	\label{re:remarkleaves}
	Both those results are stated for $\BGW$ trees conditioned on having exactly $n$ vertices. 
	In fact, it is possible to use arguments from \cite[Section~4, Section~5]{MR2946438} to deduce that the same results hold for $\BGW$ trees conditioned on having $\lfloor \mu(\{0\}) n\rfloor$ leaves. 
	This is actually the setting that we need for most of our applications.
\end{remark}

\section{Construction in the continuous} \label{s:cts_construction}

Let $(X^{\text{exc},(\alpha)}_t; t\in[0,1])$ be an excursion of an $\alpha$-stable spectrally positive L\'{e}vy process, with $\alpha\in(1,2)$ and denote the corresponding $\alpha$-stable tree by $\cT_\alpha$. 
Denote the projection from $\intervalleff01$ to $\cT_\alpha$ by $\mathbf{p}$. 
Recall the definition
\begin{align*}
I_s^t = \inf_{[s,t]} X^{\text{exc},(\alpha)}.
\end{align*}
We define a partial order $\preceq$ on $[0,1]$ as follows: for every $s,t\in[0,1]$ 
\begin{equation}\label{eq:def genealogical order continuous tree}
s\preceq t \quad \text{if} \quad s\leq t \quad \text{and} \quad X^{\text{exc},(\alpha)}_{s^{-}} \leq I_s^t.
\end{equation}
For $s,t\in[0,1]$ with $s\preceq t$, define
\begin{equation}\label{eq:definition xts}
x_s^t = I_s^t - X_{s^-}^{\text{exc}}
\end{equation}
and let
\begin{equation}\label{eq:definition uts}
u_s^t := \frac{x_s^t}{\Delta_s}
\end{equation}
where $\Delta_s = X^{\text{exc},(\alpha)}_s-X^{\text{exc},(\alpha)}_{s^-}$ is the jump of $X_s^{\text{exc},(\alpha)}$ at $s$. If $s$ is not a jump time of the process, we simply let $u_s^t:=0$. 

Let us denote $\Bb=\enstq{v\in \intervalleff{0}{1}}{\Delta_v>0}$, which is in one-to-one correspondence with the set of branch points of $\cT_\alpha$. Almost surely, these points in the tree all have an infinite degree, meaning that for any $v\in \Bb$, the space $\cT\setminus\{\mathbf{p}(v)\}$ has a countably infinite number of connected components. For any $v\in\Bb$, we let 
\begin{align*}
\cA_v:= \enstq{u_v^t}{t\in\intervalleff{0}{1} \quad \text{and} \quad \exists s\in \intervalleff{0}{1}, v\prec s \prec t }.
\end{align*} 
This set $\cA_v$ is in one-to-one correspondence with the connected components of $\cT\setminus\{\mathbf{p}(v)\}$ above $\mathbf{p}(v)$, meaning not the one containing the root. 

Conditionally on this, let $\left((\mathcal{B}_v,\rho_v,d_v,\mu_v,\nu_v)\right)_{v\in \Bb}$ be random measured metric spaces, sampled independently with the distribution of $(\mathcal{B},\rho,d,\mu,\nu)$ and for every $v\in \Bb$, let $(Y_{v,a})_{a\in\cA_v}$ be i.i.d. with law $\mu_v$. For $s \in [0,1]\setminus \Bb$, let $\mathcal{B}_s = \{s\}$. Define for every $v\in\Bb$ the rescaled distance 
\begin{equation*}
\delta_v=\Delta_v^\gamma \cdot d_v.
\end{equation*}
and for any $t\succeq v$: 
\begin{align}\label{eq:definition Zvt}
Z_v^t = \begin{cases} 
Y_{v,u_v^t} & \text{if} \quad u^t_v\in\cA_v,\\
 \rho_v \qquad &  \text{otherwise.}
\end{cases}
\end{align}
Now we are ready to define a pseudo-distance $d$ on the set $\cT^{\ast,\dec}_\alpha=\bigsqcup_{s \in [0,1]}\mathcal{B}_s$. Before doing that, let us define, for any $v,w\in\Bb$ such that $v\preceq w$, and any $x\in \mathcal{B}_w,$
\begin{align*}
Z_v^x =  \begin{cases} Z^w_v & \text{if } v\prec w\\
 x \qquad & \text{if } v=w.
\end{cases}
\end{align*}
We also extend the genealogical order on the whole space $\cT^{\ast,\dec}_\alpha$ by declaring that for any $v\in \Bb, x\in B_v$, we have $v\preceq x \preceq v$, (thus making it a preorder, instead of an order relation). We also extend the definition of $\wedge$ by saying that $a\wedge b$ is always chosen among $\intervalleff01$. 

Then we define $d_0(a,b)$ for $a\preceq b$:
\begin{align}\label{eq:def dist 0}
d_0(a,b)=d_0(b,a) = \underset{v\in\Bb}{\sum_{a\prec v\preceq b}}\delta_v(\rho_v,Z_v^b),
\end{align}
and for any $a,b\in\cT^{\ast,\dec}_\alpha$,
\begin{align*}
d(a,b)=d_0(a\wedge b,a)+d_0(a\wedge b, b) + \delta_{a\wedge b }(Z_{a\wedge b}^a,Z_{a\wedge b}^b),
\end{align*}
We say that two elements $a$ and $b$ are equivalent $a\sim b$ if $d(a,b)=0$. We quotient the space by the equivalence relation, $\cT^\dec_\alpha=\cT^{\ast,\dec}_\alpha/\sim$ and denote $\mathbf{p}^\dec$ the corresponding quotient map.

We can endow $\cT_\alpha^\dec$ with two types of measures:
\begin{itemize}
	\item If $\beta \leq \alpha$, then we define the measure $\upsilon_\beta^\dec$ on $\cT_\alpha^\dec$ as the push-forward of the Lebesgue measure on $\intervalleff{0}{1}\setminus \Bb$ through the projection map $\mathbf{p}^\dec:\cT^{\ast,\dec}_\alpha \rightarrow \cT_\alpha^\dec$ onto the quotient.
	\item If $\beta > \alpha$, we define the measure $\upsilon_\beta^\dec$ as (the push-forward through $\mathbf{p}^\dec$ of)
	\begin{align*}
		\sum_{t \in \Bb } \Delta_t^\beta \nu_t
	\end{align*}
	 which is almost surely a finite measure, provided that $\nu(\cB)$ has finite expectation, say. 
\end{itemize}

Note that so far, it is not clear that the construction above yields a compact metric space (or even a metric space), as the metric $d$ could in principle take arbitrarily large or even infinite values. 
This is handled by the work of Section~\ref{sec:self-similarity}.
In particular, a result that is important in our approach is the following, which ensures that under some moment assumption on the diameter of the decorations the distances in the whole object are dominated by the contributions of the decorations corresponding to large jumps. This will in particular be useful in the next section.
\begin{lemma} \label{l:smalldec}
Assume that $\gamma > \alpha-1$ and that condition D\ref{c:moment_diam_limit} is satisfied. 
Then, 
	\begin{align}
	\sup_{s\in \intervalleff{0}{1}} \left\lbrace\sum_{t\prec s} \Delta_t^\gamma \diam(\mathcal{B}_t) \ind{\Delta_t\leq \delta}\right\rbrace \rightarrow 0
	\end{align}
	 as $\delta\rightarrow 0$ in probability.
\end{lemma}

\section{Invariance principle} \label{s:invariance}

In this section we state and prove one of the main results of the paper, namely that under the conditions stated in Section~\ref{ss:conditions}, the discrete decorated trees with a properly rescaled graph distance, converge towards $\alpha$-stable decorated trees. 
\begin{theorem} \label{th:invariance}
	Let $(T_n)_{n\geq 1}$ be a sequence of random trees and $(B_v,\rho_{v},d_{v},\ell_{v},\nu_{v})_{v\in T_n}$ be random decorations on $T_n$ sampled independently according to $(\tilde B_k,\tilde \rho_{k},\tilde d_{k},\tilde \ell_{k},\tilde \nu_{k})_{k\geq 0}$. Assume that conditions D,T and B in Section~\ref{ss:conditions} are satisfied for some exponents $\alpha, \beta,\gamma$ and denote the weak limit of $(\tilde B_k,\tilde \rho_{k},\tilde d_{k},\tilde \ell_{k},\tilde \nu_{k})_{k\geq 0}$  in the sense of D\ref{c:GHPlimit} by $(\mathcal{B},\rho,d,\mu,\nu)$.  Suppose that $\alpha-1<\gamma$ and let $\cT_\alpha^\dec$ be the $\alpha$-stable tree $\cT_\alpha$, decorated according to $(\mathcal{B},\rho,d,\mu,\nu)$ with distance exponent $\gamma$. Then
	\begin{align*}
	(T_n^\dec,b_n^{-\gamma}d_n^\dec, a_n^{-1}\cdot  \nu_n^\dec) \to (\mathcal{T}_\alpha^\dec, d, \upsilon_\beta^\dec)
	\end{align*}
	in distribution in the Gromov-Hausdorff-Prokhorov topology.
\end{theorem}

\begin{proof} At first we will describe in detail the GH-convergence of the sequence of decorated trees, leaving out the measures $\nu_n^{\text{dec}}$, and then we give a less detailed outline of the GHP-convergence. 
We start by constructing, in several steps, a coupling that allows us to construct an $\epsilon$-isometry between $(T_n^\dec,b_n^{-\gamma} d_n^\dec)$ and $(\cT_\alpha^\dec,d)$ for arbitrary small $\epsilon$.

\emph{Step 1, coupling of trees:} By Condition T\ref{c:invariance}, we may use Skorokhod's representation theorem, and construct on the same probability space $X^{\text{exc},(\alpha)}$ and a  sequence of random trees $T_n$ such that 
	\[
(b_n^{-1} W_{\lfloor|T_n|t\rfloor}(T_n))_{0 \le t \le 1} {\,{\longrightarrow}\,} X^{\text{exc},(\alpha)}
\]
almost surely. 
We will keep the same notation for all random elements on this new space. 
Order the vertices $(v_{n,1},v_{n,2},\ldots)$ of $T_n$ in non-increasing order of their degree (in case of ties use some deterministic rule) and denote their positions in the lexicographical order on the vertices of $T_n$ by $(t_{n,1},t_{n,2},\ldots)$, in such a way that for any $k\geq 1$ we have $v_{t_{n,k}}=v_{n,k}$. 
Similarly, order the set $\Bb$ of jump times of $X^{\text{exc},\alpha}$ in decreasing order $(t_1,t_2,\ldots)$ of the sizes of the jumps (there is no tie almost surely).
Then, by properties of the Skorokhod topology, it holds that for any $k\geq 1$, 
\begin{align}\label{eq:convergence size and position of jumps}
\frac{t_{n,k}}{|T_n|} \to t_k \qquad \text{and} \qquad \frac{\out{}(v_{n,k})}{b_n} \rightarrow \Delta_{t_k}
\end{align}
almost surely, as $n\to \infty$.  
Because of the way we can retrieve the genealogical order from the coding function, the genealogical order on $(v_{n,1},v_{n,2},\dots, v_{n,N})$ for any fixed $N$ stabilizes for $n$ large enough to that of  $(t_{1},t_{2},\dots, t_{N})$ for the order $\preceq$ defined in \eqref{eq:def genealogical order continuous tree}. 

\emph{Step 2, coupling of decorations:} Let us shuffle the external roots of the decorations $(\tilde B_k,\tilde \rho_{k},\tilde d_{k},\tilde \ell_{k},\tilde \nu_{k})_{k\geq 0}$ with independent uniform permutations as described in Section~\ref{ss:shuffle} and denote the shuffled decorations by $(\hat B_k,\hat\rho_{k},\hat d_{k},\hat \ell_{k},\hat \nu_{k})_{k\geq 0}$.
From now on, sample $(B_v,\rho_{v},d_{v},\ell_{v},\nu_{v})_{v\in T_n}$ from the latter which, as explained before, does not affect the distribution of the decorated tree. 
Since for any $k\geq 1$, $\out{T_n}(v_{n,k}) \to \infty$ as $n\to \infty$, it holds, by Condition D\ref{c:GHPlimit}, that 
\begin{align*}
(B_{v_{n,k}},\rho_{v_{n,k}},(\out{}(v_{n,k}))^{-\gamma}d_{v_{n,k}},\mu_{v_{n,k}},(\out{}(v_{n,k}))^{-\beta}\nu_{v_{n,k}}) \to (\mathcal{B},\rho,d,\mu,\nu)
\end{align*} weakly as $n\to\infty$.  
We again use Skorokhod's theorem and modify the probability space such that for each $k\geq 1$, this convergence holds almost surely and the a.s.~limit will be denoted by $(\mathcal{B}_{t_k},\rho_{t_k},d_{t_k},\mu_{t_k}, \nu_{t_k})$. 
We still keep the same notation for random elements. 

For the rest of the proof, we will write $\mathbf d_v$ in place of  $b_n^{-1}\cdot d_{v}$ to lighten notation. 
From the convergence \eqref{eq:convergence size and position of jumps}, and the definition of $\delta_v=\Delta_v^\gamma \cdot d_v$, we can re-express the above convergence as the a.s.\ convergence 
\begin{align}\label{eq:convergence block unk to block vk}
(B_{v_{n,k}},\rho_{v_{n,k}}, \mathbf d_{v_{n,k}},\mu_{v_{n,k}},(b_n)^{-\beta}\nu_{v_{n,k}}) \to (\mathcal{B}_{t_k},\rho_{t_k},\delta_{t_k},\mu_{t_k},\Delta_{t_k}^\beta\nu_{t_k})
\end{align}

\emph{Step 3, coupling of gluing points:} 
Fix a $\delta > 0$ and let $N(\delta)$ be the (finite) number of vertices $v \in \Bb$ such that $\Delta_v > \delta$. 
Thanks to Step 1, we can almost surely consider $n$ large enough so that the genealogical order on $(v_{n,1},v_{n,2},\dots, v_{n,N(\delta)})$ induced by $T_n$ corresponds to that of  $(t_{1},t_{2},\dots, t_{N(\delta)})$. 
Note that $X^{\text{exc},(\alpha)}$ almost surely does not have a jump of size exactly $\delta$, and so for $n$ large enough the vertices $(v_{n,1},v_{n,2},\dots, v_{n,N(\delta)})$ are exactly the vertices with degree larger than $\delta b_n$. 
Now recall the notation \eqref{eq:definition uts} and for any $k\in \{1,2,\dots N(\delta)\}$, consider the finite set 
\begin{align*}
	\enstq{u_{t_k}^{t_i}}{i \in \{1,2,\dots N(\delta)\},\ t_i\succ t_k} \subset \mathcal A_{t_k}
\end{align*}
and enumerate its elements as $m_k(1), m_k(2), \dots , m_k(a_k)$ in increasing order.
Similarly consider the set
\begin{align*}
	\enstq{\ell \in \mathbb N}{v_{n,k}\ell \preceq v_{n,i} \text{ for some } i \in  \{1,2,\dots N(\delta)\}}\subset \{1,\dots, \out{}(v_{n,k})\}.
\end{align*}
and denote its elements $j_1,j_2,\dots, j_{a_k}$. Note that these two sets have the same cardinality $a_k$ when $n$ is large enough because of our assumptions.

For each $k\geq 1$, given $(B_{t_k},\rho_{t_k},d_{t_k},\mu_{t_k}, \nu_{t_k})$ and $X^{\text{exc},(\alpha)}$ the points $(Y_{t_k,a})_{a\in \mathcal{A}_{t_k}}$ are sampled independently according to $\mu_{t_k}$.  
%
By Lemma~\ref{l:shuffle} and \eqref{eq:convergence block unk to block vk} we may find a sequence $\phi_{v_{n,k}}: B_{v_{n,k}} \to \mathcal{B}_{t_k}$ of $\epsilon_{v_{n,k}}$-isometries, with $\epsilon_{v_{n,k}}\to 0$ as $n\to \infty$, such that 
\begin{align} \label{eq:coupling_glue}
(\phi_{v_{n,k}}(\ell_{v_{n,k}}(j_1)),\phi_{v_{n,k}}(\ell_{v_{n,k}}(j_2))\ldots,\phi_{v_{n,k}}(\ell_{v_{n,k}}(j_{a_k}))) {\,{\buildrel w \over \longrightarrow}\,} (Y_{t_k,m_k(1)},Y_{t_k,m_k(2)},\ldots,Y_{t_k,m_k(a_k)} ).
\end{align}
We modify the probability space once again, using Skorokhod's representation theorem, so that this convergence holds almost surely.
When $n$ is large enough, for $k,k'$ such that $v_{n,k} \prec v_{n,k'}$ we know that $v_{n,k} j_p \preceq v_{n,k'}$ and $u_{t_k}^{t_k'}=m_k(p)$ for some $p\in \{1,2,\dots, a_k\}$ so we have $Z_{v_{n,k}}^{v_{n,k'}}= \ell_{v_{n,k}}(j_p)$ as well as $Z_{t_k}^{t_{k'}}=Y_{t_k,m_k(a_p)}$.
Plugging this back in \eqref{eq:coupling_glue} that we have the convergence
\begin{align*}
	\phi_{v_{n,k}}(Z_{v_{n,k}}^{v_{n,k'}}) \rightarrow Z_{t_k}^{t_{k'}}.
\end{align*}

\emph{Step 4: convergence of the decorated trees}

Now, construct the sequence of decorated trees $T_n^\dec$ from $T_n$ and the decorations \\ $(B_{v_{n,k}},\rho_{v_{n,k}},d_{v_{n,k}},\ell_{v_{n,k}},\nu_{v_{n,k}})_{k\geq 1}$. Similarly, construct the tree  $\mathcal{T}_\alpha^\dec$ from $\mathcal{T}_\alpha$ and the family of decorations $(\mathcal{B}_{t_k},\rho_{t_k},d_{t_k},\mu_{t_k},\nu_{t_k})_{k\geq 1}$. Denote the projections to the quotients by $\mathbf{p}_n^\dec$ and $\mathbf{p}^\dec$ respectively. 
We claim that this coupling of trees, decorations and gluing points, guarantees that the convergence
	\begin{align*}
	(T_n^\dec,b_n^{-\gamma}d_n^\dec) \to (\mathcal{T}_\alpha^\dec, d)
\end{align*}
holds in probability for the Gromov--Hausdorff topology.

Introduce
\begin{align}
	r_n^\delta :=  b_n^{-\gamma} \cdot\sup_{v\in T_n} \left\lbrace\sum_{w\preceq v} \diam(B_w) \ind{\out{T_n}(w)\leq \delta b_n}\right\rbrace,
\end{align}
and also 
\begin{align}\label{def:r delta}
r^\delta := \sup_{s\in \intervalleff{0}{1}} \left\lbrace\sum_{t\prec s} \Delta_t^\gamma \diam(B_t) \ind{\Delta_t\leq \delta}\right\rbrace.
\end{align}
It holds, by condition B1 and Lemma~\ref{l:smalldec} that
\begin{align}
	\lim_{\delta \to 0} \limsup_{n\rightarrow\infty} r_n^\delta = 0  \quad \text{and} \quad \lim_{\delta\rightarrow 0} r^\delta =0
\end{align}
in probability. 

%
Let us consider the subset $\cT_\alpha^{\dec,\delta}\subset \cT_\alpha^\dec$ which consists only of the blocks corresponding to jumps larger than $\delta$, i.e.  
\begin{align}
\cT_\alpha^{\dec,\delta} = \bigsqcup_{1\leq k \leq N(\delta)}\mathbf{p}^\dec \left(\mathcal{B}_{t_k}\right),
\end{align}
 endowed with the induced metric $d$. 
 From the construction of $d$ it is clear that we can bound their Hausdorff distance
 \begin{align}\label{eq:T alpha well approximated by large jumps}
 	\mathrm{d_H}(\cT_\alpha^{\dec,\delta},\cT_\alpha^\dec) \leq 2 r^\delta.
 \end{align} 

 Let us do a similar thing with the discrete decorated trees.  
 We introduce $T_n^{\dec,\delta} \subset T_n^\dec$ endowed with the induced metric from $b_n^{-\gamma} \cdot d_n^\dec$, where
 \begin{align}
 	T_n^{\dec,\delta}=\bigsqcup_{1\leq k \leq N(\delta)} \mathbf{p}_n^\dec\left(B_{v_{n,k}}\right).
 \end{align}
We have, for the distance $b_n^{-\gamma} \cdot d_n^\dec$,
 	 \begin{align}
 	\mathrm{d_H}(T_n^{\dec,\delta},T_n^\dec) \leq 2 r_n^\delta .
 	\end{align}

We introduce the following auxiliary object
\begin{align*}
	\widehat{T}_n^{\dec,\delta}=\bigsqcup_{1\leq k \leq N(\delta)} B_{v_{n,k}},
\end{align*}
which we endow with the distance $\hat{d}$ defined as 
\begin{align*}
	\hat{d}(x,y)= b_n^{-\gamma}\cdot d_n^\dec(\mathbf{p}^\dec_n(x),\mathbf{p}^\dec_n(y))+\delta\cdot \ind{x,y \text{ don't belong to the same block}}. 
\end{align*}
It is easy to check that $\hat{d}$ is a distance on $\widehat{T}_n^{\dec,\delta}$ and that the projection $\mathbf{p}^\dec_n$ is a $\delta$-isometry, so that 
\begin{align*}
	\mathrm{d_{GH}}\left((\widehat{T}_n^{\dec,\delta}, \hat{d}),(T_n^{\dec,\delta},b_n^{-\gamma} \cdot d_n^\dec)\right) \leq \delta.
\end{align*}
The role of this auxilliary object is just technical: it allows to have a metric space very close to $T_n^\dec$ but where no two vertices of two different blocks are identified together. This allows to define function from $\widehat{T}_n^{\dec,\delta}$ to $\cT_\alpha^{\dec,\delta}$ by just patching together the almost isometries that we have on the blocks.

\emph{Step 5, constructing an almost isometry:}
Now, we construct a function from $\widehat{T}_n^{\dec,\delta}$ that we will show is an $\epsilon$-isometry for some small $\epsilon$ when $n$ is large enough. 
Recall the $\epsilon_{v_{n,k}}$-isometries $\phi_{v_{n,k}}$ from \eqref{eq:coupling_glue}.
We define the function $\phi$ from $\widehat{T}_n^{\dec,\delta}$ to $\cT_\alpha^{\dec}$ as the function such that for any $1\leq k \leq N(\delta)$ and any $x\in B_{v_{n,k}}$ we have
\begin{align*}
	\phi(x):= \mathbf{p}^\dec(\phi_{v_{n,k}} (x))
\end{align*}
Now let us show that for $n$ large enough $\phi$ is an $\epsilon$ isometry for 
\begin{align*}
	\epsilon=3r^\delta+3 r^{\delta}_{n}+ 3\delta.
\end{align*} 

First, from \eqref{eq:T alpha well approximated by large jumps} and the definition of $\phi$ from the $\phi_{v_{n,k}}$, it is clear that the $\phi$ satisfies the almost surjectivity condition for that value of $\epsilon$. 
Then we have to check the almost isometry condition. 
For that, let us take $n$ large enough so that on the realization that we consider, the genealogical order on $(v_{n,1},v_{n,2},\dots, v_{n,N(\delta)})$ in $T_n$ coincides with that on $(t_{1},t_{2},\dots, t_{N}(\delta))$ defined by \eqref{eq:def genealogical order continuous tree}.

Now for $k,k'\leq N(\delta)$ for which $t_k \prec t_{k'}$ (and equivalently $v_{n,k} \prec v_{n,k'}$) we have
\begin{align*}
\underset{1\leq i\leq N(\delta)}{\sum_{v_{n,k}\prec v_{n,i}\prec v_{n,k'}}} \mathbf d_{v_{n,i}}\left(\rho_{v_{n,i}},Z_{v_{n,i}}^{v_{n,k'}}\right)
&\leq \sum_{w\in T_n: v_{n,k} \prec w \prec v_{n,k'}} \mathbf d_w\left(\rho_w,Z_w^{v_{n,k'}}\right)
 \leq  \underset{1\leq i\leq N(\delta)}{\sum_{v_{n,k}\prec v_{n,i}\prec v_{n,k'}}} \mathbf d_{v_{n,i}}\left(\rho_{v_{n,i}},Z_{v_{n,i}}^{v_{n,k'}}\right) + r_n^\delta
\end{align*}
and similarly
\begin{align*}
 \underset{1\leq i\leq N(\delta)}{\sum_{t_{k}\prec t_{i}\prec t_{k'}}} \delta_{t_{i}}\left(\rho_{t_{i}},Z_{t_{i}}^{t_{k'}}\right)
&\leq \sum_{w\in \mathbb{B}: t_k \prec w \prec t_{k'}} \delta_w\left(\rho_w,Z_w^{t_{k'}}\right)
\leq  \underset{1\leq i\leq N(\delta)}{\sum_{t_{k}\prec t_{i}\prec t_{k'}}} \delta_{t_{i}}\left(\rho_{t_{i}},Z_{t_{i}}^{t_{k'}}\right) + r^\delta
\end{align*}
and so we have 
\begin{align*}
&|\sum_{w\in T_n: v_{n,k} \prec w \prec v_{n,k'}} \mathbf d_w\left(\rho_w,Z_w^{v_{n,k'}}\right)- \sum_{w\in \mathbb{B}: t_k \prec w \prec t_{k'}} \delta_w\left(\rho_w,Z_w^{t_{k'}}\right) | \\
&\leq \sum_{i: v_{n,k} \prec v_{n,i} \prec v_{n,k'}} |\mathbf d_{v_{n,i}}\left(\rho_{v_{n,i}},Z_{v_{n,i}}^{v_{n,k'}}\right)-  \delta_{t_i}\left(\rho_{t_i},Z_{t_i}^{t_{k'}}\right) | + r^\delta + r_n^\delta\\
&\leq \underset{\epsilon_n}{\underbrace{\sum_{k,k',i=1}^{N(\delta)} |\mathbf d_{v_{n,i}}\left(\rho_{v_{n,i}},Z_{v_{n,i}}^{v_{n,k'}}\right)-  \delta_{t_i}\left(\rho_{t_i},Z_{t_i}^{t_{k'}}\right) |}} + r^\delta + r_n^\delta
\end{align*}
and by \eqref{eq:convergence block unk to block vk}, the term $\epsilon_n$ in the above display tends to $0$ for a fixed $\delta$ as $n$ tends to infinity. In particular, it is eventually smaller than $\delta$. 
 
Now let us take two $x,x'\in \widehat{T}_n^{\dec,\delta}$, where $x\in B_{v_{n,k}}$ and $x'\in B_{v_{n,k'}}$. Let us study the different cases

$\bullet$ Case 1: assume that $k=k'$, then
\begin{align}
d(\phi(x),\phi(x')) = \delta_{t_k}(\phi_{v_{n,k}}(x),\phi_{v_{n,k}}(x')) \quad \text{and} \quad \hat d(x,x') =\mathbf d_{v_{n,k}}(x,x'). 
\end{align}
and then we have $|d(\phi(x),\phi(x'))  - \hat d(x,x')|\leq \epsilon_{v_{n,k}}$ by definition of $\phi_{v_{n,k}}$ being an $\epsilon_{v_{n,k}}$-isometry. For a fixed $\delta$, this quantity becomes smaller than $\delta$ for $n$ large enough. Since  $k\leq N(\delta) < \infty$ a.s., this holds uniformly in $k$.

$\bullet$ Case 2: assume that $t_k\prec t_{k'}$, then
\begin{align*}
d(\phi(x),\phi(x')) &= \delta_{t_{k}}(\phi(x),Z_{t_k}^{t_{k'}})+ \sum_{w\in \mathbb{B}: t_k \prec w \prec t_{k'}}\delta_w\left(\rho_w,Z_w^{t_{k'}}\right) +
 \delta_{t_{k'}}(\rho_{t_{k'}},\phi(x'))\\
 &= \delta_{t_{k}}(\phi_{u_{n,k}}(x),Z_{t_k}^{t_{k'}})+ \sum_{w\in \mathbb{B}: t_k \prec w \prec t_{k'}}\delta_w\left(\rho_w,Z_w^{t_{k'}}\right) +
 \delta_{t_{k'}}(\rho_{t_{k'}},\phi_{t_{k'}}(x'))
\end{align*}
and 
\begin{align*}
\hat d(x,x') = \mathbf d_{v_{n,k'}}(x,Z_{v_{n,k}}^{v_{n,k'}})+ \sum_{w\in T_n: v_{n,k} \prec w \prec v_{n,k'}}\mathbf d_w\left(\rho_w,Z_w^{v_{n,k'}}\right) +
\mathbf d_{v_{n,k'}}(\rho_{v_{n,k'}},x') + \delta.
\end{align*}
Hence we have 
\begin{align*}
	|d(\phi(x),\phi(x')) - \hat d(x,x')| \leq |\delta_{t_{k}}(\phi(x),Z_{t_k}^{t_{k'}}) -  \mathbf d_{v_{n,k'}}(x,Z_{v_{n,k}}^{v_{n,k'}}) |\\ + |\sum_{w\in \mathbb{B}: t_k \prec w \prec t_{k'}}\delta_w\left(\rho_w,Z_w^{t_{k'}}\right) - \sum_{w\in T_n: v_{n,k} \prec w \prec v_{n,k'}}\mathbf d_w\left(\rho_w,Z_w^{v_{n,k'}}\right)|\\ + |\delta_{t_{k'}}(\rho_{t_{k'}},\phi_{t_{k'}}(x'))-\mathbf d_{v_{n,k'}}(\rho_{v_{n,k'}},x')|
\end{align*}
and so, upper-bounding each term, the last display is eventually smaller than $\epsilon$ as $n\rightarrow \infty$. \\
$\bullet$ Case 3: assume that $t_k\npreceq t_{k'}$ and $ t_{k'}\npreceq t_k$. Then, we can write
\begin{align}
d(\phi(x),\phi(x'))= \delta_{t_k \wedge t_{k'}} (Z_{t_k \wedge t_{k'}}^{t_k},Z_{t_k \wedge t_{k'}}^{t_{k'}}) + \sum_{t_k \wedge t_{k'} \prec w \preceq t_{k'}}\delta_{w}(\rho_{w},Z_w^{t_{k'}})+\delta_{t_{k'}}(\rho_{t_{k'}},\phi(x))\\
+\sum_{t_k \wedge t_{k'} \prec w\preceq t_{k}}\delta_{w}(\rho_{w},Z_w^{t_{k}})+\delta_{t_k}(\rho_{t_{k}}, \phi(x')),
\end{align}
in the same fashion as before, and a similar expression can be written down for the discrete object.

Now, we can consider two cases: either $t_k \wedge t_{k'}=t_j$ for some $j\in \{1,\dots,N(\delta)\}$, in which case we can write 
\begin{align*}
	|\delta_{t_k \wedge t_{k'}} (Z_{t_k \wedge t_{k'}}^{t_k},Z_{t_k \wedge t_{k'}}^{t_{k'}}) - \mathbf d_{v_{n,k} \wedge v_{n,k'}} (Z_{v_{n,k} \wedge v_{n,k'}}^{v_{n,k}},Z_{v_{n,k} \wedge v_{n,k'}}^{v_{n,k'}})|\leq 2\epsilon_n, 
\end{align*}
which tends to $0$ as $n\rightarrow \infty$. Otherwise, it means that $\Delta_{t_k \wedge t_{k'}}\leq \delta$ and so we can upper-bound the term in the last display by $r^\delta + r_n^\delta$.

With the other terms, we can reason similarly as in the previous case and get 
\begin{align}
|\hat d(x,x') - d(\phi(x),\phi(x'))|\leq  3r^\delta+4\epsilon_n+3r^{\delta}_{n}+2\delta+ \epsilon_{v_{n,k}} + \epsilon_{v_{n,k'}},
\end{align}
which is smaller than $\epsilon$ for $n$ large enough.

\paragraph{Adding the measures.}
We sketch here a modification of the proof above to improve the convergence from Gromov--Hausdorff to Gromov--Hausdorff--Prokhorov topology.

\medskip

\emph{Step 1, the total mass converges:} First, remark that under the assumptions of the theorem with $\beta\leq \alpha$ the total mass $\nu_n^\dec(T_n^\dec)$ of the measure carried on $T_n^\dec$ satisfies 
\begin{align*}
	a_n^{-1} \cdot \nu_n^\dec(T_n^\dec) \underset{n\rightarrow \infty}{\rightarrow} 1.
\end{align*}
On the other hand, if $\beta>\alpha$, then using the coupling defined in the proof above, we claim that we have 
\begin{align}\label{eq:convergence total mass nu n}
	b_n^{-\beta}\cdot \nu_n^\dec(T_n^\dec)=b_n^{-\beta}\cdot\sum_{i=1}^{|T_n|} \nu_{v_i}(B_{v_i}) \underset{n\rightarrow \infty}{\rightarrow} \sum_{0\leq s\leq 1}\Delta_s^\beta \cdot \nu_s(\mathcal{B}_s).
\end{align}
Let us prove the above display.
First, using \eqref{eq:convergence size and position of jumps} and \eqref{eq:convergence block unk to block vk}, we can write for all $k\geq 1$,
\begin{align*}
	b_n^{-\beta}\cdot \nu_{v_{n,k}}(B_{v_{n,k}}) \underset{n\rightarrow \infty}{\rightarrow} \Delta_{t_k}^\beta\cdot \nu_{t_k}(\cB_{t_k}).
\end{align*}
This ensures that for any given $\delta>0$ we have 
\begin{align*}
b_n^{-\beta}\cdot \sum_{k=0}^\infty \nu_{v_{n,k}}(B_{v_{n,k}}) \ind{d^+(v_{n,k})> \delta b_n}\rightarrow\sum_{0\leq s\leq 1}\Delta_s^\beta \cdot \nu_s(\mathcal{B}_s)\ind{\Delta_s > \delta},
\end{align*}
and the quantity on the right-hand-side converges to $\sum_{0\leq s\leq 1}\Delta_s^\beta \cdot \nu_s(\mathcal{B}_s)$ as $\delta\rightarrow 0$.
From there, we can construct a sequence $\delta_n\rightarrow 0$ such that the following convergence holds on an event of probability $1$
\begin{align*}
b_n^{-\beta}\cdot \sum_{k=0}^\infty \nu_{v_{n,k}}(B_{v_{n,k}}) \ind{d^+(v_{n,k})\geq \delta_n b_n}\rightarrow\sum_{0\leq s\leq 1}\Delta_s^\beta \cdot \nu_s(\mathcal{B}_s).
\end{align*}
Now, using Condition~B\ref{cond:small decorations dont contribute to mass} we get that 
\begin{align*}
R(n,\delta)= b_n^{-\beta}\cdot \sum_{i=1}^{|T_n|}\nu_{v_i}(B_{v_i}) \ind{\out{}(v_i)\leq \delta b_n}
\end{align*}
tends to $0$ in probability as $\delta \rightarrow 0$, uniformly in $n$ large enough. 
This ensures that for our sequence $\delta_n$ tending to $0$ we have the convergence $R(n,\delta_n)\underset{n \rightarrow \infty}{\rightarrow} 0$ in probability. 
Putting things together we get that \eqref{eq:convergence total mass nu n} holds in probability. 


\medskip

\emph{Step 2, Convergence of the sampled points:}
Now that we know that the total mass of $\nu_n^\dec$ appropriately normalized converges to that of $\upsilon_\beta^\dec$, we consider a point $Y_n$ sampled under normalized version of $\nu_n^\dec$ and $\Upsilon_\beta$ sampled under a normalized version of $\upsilon_\beta^\dec$. 
In order to prove that $(T_n^\dec,d,\nu_n^\dec)$ converges to $(\cT_\alpha^\dec,d,\upsilon_\beta^\dec)$ in distribution, we will prove that the pointed
spaces $(T_n^\dec,d,Y_n)$ converge to $(\cT_\alpha^\dec,d,\Upsilon_\beta)$ in distribution for the pointed Gromov--Hausdorff topology.

\medskip

\emph{Step 3, Sampling a uniform point:} 
To construct $Y_n$ on $T_n^\dec$, we first sample a point $X_n$ on $\intervalleff{0}{1}$.
Then we introduce
\begin{align}
	K_n:=\inf \enstq{k\geq 1}{M_k\geq X_n \cdot M_{|T_n|}}
\end{align}
and then sample a point on $B_{v_{K_n}}$ using a normalized version of the measure $\nu_{v_{K_n}}$. 

For $\beta\leq \alpha$, in the continuous setting, we construct a random point $\Upsilon_\beta$ on $\cT_\alpha^\dec$ distributed as $\upsilon_\beta^\dec$ by defining it as $\mathbf{p}^\dec(X)$ where $X\sim \mathrm{Unif}\intervalleff{0}{1}$.

For $\beta>\alpha$ we can define the following random point $\Upsilon_\beta$ on $\cT_\alpha^\dec$ in three steps: sample $X\sim \mathrm{Unif}\intervalleff{0}{1}$ and then set
\begin{align*}
Z:=\inf \enstq{t\in \intervalleff{0}{1}}{\sum_{0\leq s\leq t}\Delta_s^\beta \nu_s(\mathcal{B}_s)\geq X\cdot \sum_{0\leq s\leq 1}\Delta_s^\beta\nu_s(\mathcal{B}_s)}
\end{align*}
and then sample a point on $\mathcal{B}_{Z}$ using a normalized version of the measure $\nu_{Z}$. (Note that in this case, $Z$ is a jump time of $X^{\mathrm{exc},(\alpha)}$ almost surely.)
This construction ensures that conditionally on $X^{\mathrm{exc},(\alpha)}$, and $(\mathcal{B}_s)_{s\in \Bb}$, for all $t\in\intervalleff{0}{1}$ we have
\begin{align*}
	\Pp{Z=t}=\frac{\Delta_t^\beta \cdot \nu_t(\mathcal{B}_t)}{\sum_{0\leq s\leq 1}\Delta_s^\beta \cdot \nu_s(\mathcal{B}_s)}.
\end{align*}

\emph{Step 4, Coupling $Y_n$ and $\Upsilon_\beta$:}
In both cases $\alpha \leq \beta$ and $\alpha >\beta$ respectively, we couple the construction of $Y_n$ and $\Upsilon$ by using the same uniform random variable $X_n=X$.
Now let us consider the two cases separately.

\emph{Case $\beta \leq \alpha$.}  Consider the convergence \eqref{eq:convergence depth-first mass}. 
Since this convergence is in probability to a constant, it happens jointly with other convergence results used above using Slutsky's lemma. 
Hence, using the Skorokhod embedding again, we can suppose that this convergence is almost sure, jointly with the other ones. 
Then, we can essentially go through the same proof as above and consider an extra decoration corresponding to $B_{v_{K_n}}$, which contains the random point $Y_n$.  

\emph{Case $\beta > \alpha$.} In that case, thanks to the reasoning of Step 1, we have the following convergence for the Skorokhod topology
\begin{align*}
b_n^{-\beta}\cdot\left(\sum_{i=1}^{\lfloor t |T_n|\rfloor} \nu_{v_i}(B_{v_i}) \right)_{0\leq t \leq 1} \underset{n\rightarrow \infty}{\rightarrow} \left(\sum_{0\leq s\leq t}\Delta_s^\beta \cdot \nu_s(\mathcal{B}_s)\right)_{0\leq t \leq 1}.
\end{align*}
With the definition of respectively $Z$ and $K_n$ using the coupling $X_n=X$, the above convergence ensures that for $n$ large enough we have $K_n=v_{n,k}$ and $Z=t_k$ for some (random) $k\geq 1$. 
Then again, we can essentially go through the same proof as above and consider an extra random point $Y_n$ sampled on $B_{v_{n,k}}$ and a point $\Upsilon_\beta$ sampled on $\cB_{t_{k}}$ in such a way that those points are close together.  

\end{proof}

\section{Self-similarity properties and alternative constructions}\label{sec:self-similarity}
The goal of this section is to decompose the decorated stable tree along a discrete tree, using a framework similar to that developed in Section~2.1, with the difference that the gluing will occur along the whole Ulam tree instead of a finite tree. 
In particular this decomposition, which follows directly from a similar description of the $\alpha$-stable tree itself, will allow us to describe the decorated stable tree as obtained by a sort of line-breaking construction, and also identify its law as the fixed point of some transformation. 
We rely on a similar construction that holds jointly for the $\alpha$-stable tree and its associated looptree described in \cite{senizergues_growing_2022}.

\textbf{In this section, we will always assume that $\beta>\alpha$ and that $\Ec{\nu(\cB)}$ and $\Ec{\diam(\cB)}$ are finite.}

\subsection{Introducing the spine and decorated spine distributions}
\begin{figure}\label{fig:spine and decorated spine}
	\centering
	\begin{tabular}{ccc}
		\includegraphics[height=6cm,page=1]{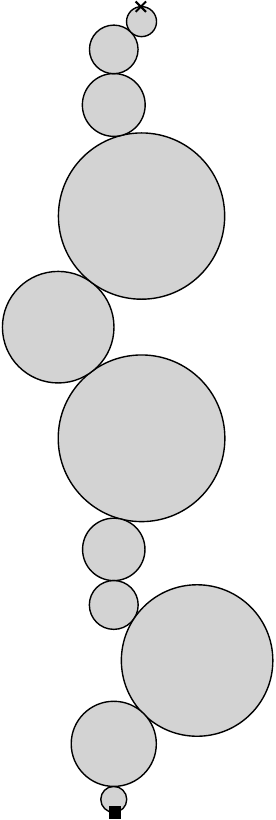} & & \includegraphics[height=6cm,page=2]{spineanddecoratedspine}
	\end{tabular}
	\caption{On the left, the space $\cS^{\dec}$. On the right, the corresponding space $\cS$ which is just a segment with atoms along it.}
\end{figure}
We construct two related random metric spaces, $\cS$ and $\cS^{\dec}$, which will be the building blocks used to construct $\cT_\alpha$ and $\cT^{\dec}_\alpha$ respectively. This construction will depend on parameters $\alpha\in(1,2)$ and $\gamma\in(\alpha-1,1]$ and $\beta>\alpha$.
First, let   
\begin{itemize}
	\item $(Y_k)_{k\geq 1}$ be a sequence of i.i.d. uniform random variables in $\intervalleff{0}{1}$,
	\item $(P_k)\sim \GEM(\alpha-1,\alpha-1)$, and $L$ its $(\alpha-1)$-diversity.
	\item $(\cB_k,d_k,\rho_k,\mu_k,\nu_k)$ i.i.d. random metric spaces with the same law as $(\cB,d,\rho,\mu,\nu)$,
	\item for every $k\geq 1$, a random point $Z_k$ taken on $\cB_k$ under $\mu_k$. 
\end{itemize} 
We refer for example to \cite[Appendix A.2]{senizergues_growing_2022} for the definitions of the distributions in the second bulletpoint. 
We first define the random spine \[\cS=(S,d_S,\rho_S,\mu_S)\] as the segment $S=\intervalleff{0}{L}$, rooted at $\rho_S=0$, 
endowed with the probability measure $\mu_S=\sum_{k\geq1}P_k \delta_{L\cdot Y_k}$ 
In order to construct $\cS^{\dec}$ we are going to informally replace every atom of $\mu_S$ with a metric space scaled by an amount that corresponds to the mass of that atom.
To do so, we consider the disjoint union
\begin{equation}
	\bigsqcup_{i=1}^\infty \cB_i,
\end{equation}
which we endow with the distance $d_{S^{\dec}}$ defined as
\begin{align*}
	d_{S^{\dec}}(x,y) &= \begin{cases}
		P_i^\gamma d_i(x,y) & \text{if} \quad x,y\in \cB_i,\\
		P_i^\gamma d_i(x,Z_i) + \sum_{k: Y_i<Y_k<Y_j}P_i^\gamma d_k(\rho_k,Z_k)+P_j^\gamma d_j(\rho_j,y) & \text{if} \quad x\in \cB_i,\ y\in \cB_j,\ Y_i<Y_j.
	\end{cases}
\end{align*}
Then $\cS^{\dec}$ is defined as the completion of $\bigsqcup_{i=1}^\infty \cB_i$ equipped with this distance. Its root $\rho$ can be obtained as a limit $\rho=\lim_{i\rightarrow\infty}\rho_{\sigma_i}$ for any sequence $(\sigma_i)_{i\geq 1}$ for which $Y_{\sigma_i}\rightarrow 0$. 
Essentially, $\cS^{\dec}$ looks like a skewer of decorations that are arranged along a line, in uniform random order. 
We can furthermore endow $\cS^{\dec}$ with a probability measure $\mu_{S^{\dec}}=\sum_{k=1}^{\infty}P_k \mu_k$.
If $\beta>\alpha$ and $\Ec{\nu(\cB)}$ is finite, we also define the measure $\nu_{S^{\dec}}=\sum_{k=1}^{\infty}P_k^\beta \nu_k$.
In the end, we have defined
\begin{align*}
	\cS^{\dec}=(S^{\dec},d_{S^{\dec}},\rho_{S^{\dec}},\mu_{S^{\dec}},\nu_{S^{\dec}}).
\end{align*}
It is important to note that conditionally on $\cS$, constructing $\cS^{\dec}$ consists in replacing every atom of $\mu_S$ by an independent copy of a random metric space appropriately rescaled. 
This exactly corresponds to what we are trying to do with the entire tree $\cT_\alpha$.

Up to now, it is not clear whether the last display is well defined though, because it is not clear whether the function $d_{S^{\dec}}$ only takes finite values.
The next lemma ensures that it is the case as long as $\diam(\cB)$ has a finite first moment.
Introduce the quantity $R=\sum_{i\geq 1} P_i^\gamma \diam(\cB_i)$. 
\begin{lemma}\label{lem:moment diam block implis moment diam spine}
	If $\Ec{\diam(\cB)^p}<\infty$ for some $p\geq 1$ then $\Ec{R^p}$ is finite so $\cS^{\dec}$ is almost surely a compact metric space with $\Ec{\diam(\cS^{\dec})^p}<\infty$. 
\end{lemma}
In order to get Lemma~\ref{lem:moment diam block implis moment diam spine} we will apply the following lemma, with $\xi=\theta=\alpha-1$ and $Z_i=\diam(\cB_i)$. 
\begin{lemma}
	Let $(Z_i)_{i\geq 1}$ be i.i.d. positive random variables with moment of order $p\geq 1$, and $(P_i)_{i\geq 1}$ an independent random variable with $\GEM(\xi,\theta)$ distribution, and $\gamma>\xi$. Then the random variable
	\begin{align*}
		\sum_{i=1}^{\infty}P_i^\gamma Z_i,
	\end{align*} 
	admits moments of order $p$.
\end{lemma}
Note that this lemma also guarantees that the measure $\nu_{S^{\dec}}=\sum_{k=1}^{\infty}P_k^\beta \nu_k$ is almost surely a finite measure, as soon as $\Ec{\nu(B)}$ is finite.

	\begin{proof}
		Suppose that $\xi<\gamma\leq 1$ and write 
		\begin{align*}
			\left(\sum_{i=1}^{\infty}P_i^\gamma Z_i\right)^p
			&=\left(\sum_{i=1}^{\infty}P_i^\gamma\right)^p \left(\sum_{i=1}^{\infty}\left(\frac{P_i^\gamma}{\sum_{i=1}^{\infty}P_i^\gamma}\right) Z_i\right)^p\\
			&\underset{\text{convexity}}{\leq} \left(\sum_{i=1}^{\infty}P_i^\gamma\right)^p \left(\sum_{i=1}^{\infty}\left(\frac{P_i^\gamma}{\sum_{i=1}^{\infty}P_i^\gamma}\right) Z_i^p\right)\\
			&\leq \left(\sum_{i=1}^{\infty}P_i^\gamma\right)^{p-1}\left(\sum_{i=1}^{\infty}P_i^\gamma Z_i^p\right)
		\end{align*}
		In the case $\gamma>1$, since all the $P_i$ are smaller than $1$, we can write
		\begin{align*}
			\left(\sum_{i=1}^{\infty}P_i^\gamma Z_i\right)^{p}&\leq \left(\sum_{i=1}^{\infty}P_i Z_i\right)^{p}\leq \left(\sum_{i=1}^{\infty}P_i Z_i^{p}\right),
		\end{align*}
		where the second inequality comes from convexity.
		Now, for any $\xi <\gamma\leq 1$, for $n=\lceil p-1 \rceil$ we have
		\begin{align*}
			\Ec{\left(\sum_{i=1}^{\infty}P_i^\gamma\right)^{p-1}\left(\sum_{i=1}^{\infty}P_i^\gamma Z_i^p\right)}
			&\leq \Ec{\left(\sum_{i=1}^{\infty}P_i^\gamma Z_i^p\right) \cdot \left(\sum_{i=1}^{\infty}P_i^\gamma\right)^{n}}\\
			&=\Ec{\sum_{(i_0,i_1,i_2,\dots,i_n)\in(\mathbb N^*)^{n+1}}P_{i_0}^\gamma\cdot  Z_{i_0}^p \cdot  P_{i_1}^\gamma\dots P_{i_n}^\gamma}\\
			&\leq	\sum_{(i_0,i_1,i_2,\dots,i_n)\in(\mathbb N^*)^{n+1}}\Ec{P_{i_1}^\gamma\dots P_{i_n}^\gamma} \cdot  \Ec{Z_{i_0}^p}\\
			&<\infty
		\end{align*}
		The last term is finite provided that $\gamma>\xi$, using Lemma~5.4 in \cite{goldschmidt_stable_2018}.
\end{proof}

\subsection{Description of an object as a decorated Ulam tree}\label{subsec:description of the objet as a gluing of spines}
Now that we have defined the random object $\cS^\dec$, we define some decorations of the Ulam tree that are constructed in such a way that up to some scaling factor, all the decorations are i.i.d.\@ with the same law as $\cS^\dec$. Gluing all those decorations along the structure of the tree, as defined in the introduction, will give us a random object $\tilde{\cT}^\dec_\alpha$ that has the same law as $\cT^\dec_\alpha$ (the proof of that fact will come later in the section).

In fact, in what follows, we provide two equivalent descriptions of the same object $\tilde{\cT}^\dec_\alpha$:
\begin{itemize}
	\item One of them is a description as a random self-similar metric space: this construction will ensure that the object that we construct is compact under the weakest possible moment assumption on the diameter of the decoration and also give us the Hausdorff dimension of the object.
	\item The other one is through an iterative gluing construction. It is that one that we use to identify the law of $\tilde{\cT}^\dec_\alpha$ with that of $\cT^\dec_\alpha$.
\end{itemize}
The fact that the two constructions yield the same object is a result from \cite{senizergues_growing_2022}. 
\subsubsection{Decorations on the Ulam tree}
We extend the framework introduced in the first section for finite trees to the entire Ulam tree, following \cite{senizergues_growing_2022}.
We work with families of decorations on the Ulam tree $\mathcal D=(\mathcal{D}(v), v\in \mathbb U)$ of measured metric spaces indexed by the vertices of $\mathbb U$.
Each of the decorations $\mathcal{D}(v)$ indexed by some vertex $v\in \mathbb U$ is rooted at some $\rho_v$ and carries a sequence of gluing points $(\ell_v(i))_{i\geq 1}$. The way to construct the associated decorated Ulam tree is to consider as before
\begin{align*}
	\bigsqcup_{v\in\mathbb U}\mathcal{D}(v)
\end{align*} 
and consider the metric gluing of those blocks obtained by the relations $\ell_v(i)\sim \rho_{vi}$ for all $v\in\mathbb U$ and $i\ge 1$. 
For topological considerations, we actually consider the metric completion of the obtained object. The final result is denoted $\sG(\cD)$. 

\paragraph{The completed Ulam tree.}
Recall the definition of the Ulam tree $\bU=\bigcup_{n\geq 0} \mathbb N^n$ with $\mathbb N=\{1,2,\dots\}$. 
Introduce the set $\partial\bU=\mathbb{N}^\mathbb{N}$, to which we refer as the \emph{leaves} of the Ulam tree, which we see as the infinite rays joining the root to infinity and set $\overline{\bU}:=\bU\cup \partial\bU$. 
On this set, we have a natural genealogical order $\preceq$ defined in such a way that $v\preceq w$ if and only if $v$ is a prefix of $w$. 
From this order relation we can define for any $v\in\bU$ the \emph{subtree descending from $v$} as the set $T(v):=\enstq{w\in \overline{\bU}}{v\preceq w}$. The collection of sets $\{T(v),\ v\in\bU\} \sqcup \{\{v\},\ v\in\bU\}$ generates a topology over $\overline{\bU}$, which can also be generated using an appropriate ultrametric distance.
Endowed with this distance, the set $\overline{\bU}$ is then a separable and complete metric space. 

\paragraph{Identification of the leaves.}
Suppose that $\cD$ is such that $\sG(\cD)$ is compact. 
We can define a map 
\begin{equation}\label{growing:eq:def de iota}
	\iota_\cD:\partial \bU\rightarrow \sG(\cD),
\end{equation}
that maps every leaf of the Ulam-Harris tree to a point of $\sG(\cD)$.
For any $\mathbf{i}=i_1i_2\dots \in\partial \bU$, the map is defined as 
\begin{align*}
	\iota_{\cD}(\mathbf{i}):= \lim_{n\rightarrow\infty} \rho_{i_1i_2\dots i_n} \in \sG(\cD),
\end{align*}
The fact that this map is well-defined and continuous is proved in \cite{senizergues_growing_2022}. 

\paragraph{Scaling of decorations.} In the rest of this section, for $\mathcal{X}=(X,d,\rho,\mu,\nu)$ a pointed metric space endowed with finite measures (as well as some extra structure like a sequence of points for example) we will use the notation
\begin{align*}
	\mathrm{Scale}(a,b,c;\mathcal{X}) = (X,a\cdot d,\rho,b \cdot \mu, c\cdot \nu),
\end{align*}
where the resulting object still carries any extra structure that $\mathcal{X}$ originally carried.
\subsubsection{A self-similar decoration}
We introduce the following random variables:
\begin{itemize}
	\item For every $v\in \mathbb{U}$, sample independently $(Q_{vi})_{i\geq 1}\sim \GEM(\frac{1}{\alpha},1-\frac{1}{\alpha})$ and denote $D_v$ the $\frac{1}{\alpha}$-diversity of $(Q_{vi})_{i\geq 1}$.
	\item This defines for every $v\in \mathbb{U}$,
	\begin{align*}
		\overline{Q_v}=\prod_{w\preceq v}Q_v \qquad \text{and} \qquad w_v=(\overline{Q_v})^\frac{1}{\alpha}\cdot D_v 
	\end{align*}
	This in fact also defines a probability measure $\eta$ on $\partial \mathbb U$ the frontier of the tree which is characterized by $\eta(T(v))=\overline{Q_v}$, for all $v\in \mathbb U$. 
	\item For every $v\in \mathbb{U}$ we sample independently of the rest $\cS^{\dec}_v$ that has the same law as $\cS^{\dec}=(S^{\dec},d_{S^{\dec}},\rho_{S^{\dec}},\mu_{S^{\dec}},\nu_{S^{\dec}})$ and sample a sequence of points from its measure $\mu_{S^{\dec}}$. 
	We call those points $(X_{vi})_{i\geq 1}$ and define $\ell_v:i\mapsto X_{vi}\in\cS^{\dec}_v $.  
	Then we can consider the following decorations on the Ulam tree as, for $v\in\mathbb{U}$,
	\begin{align}\label{eq:definition Tdec from GEM}
		\mathcal{D}(v):= \mathrm{Scale}\left(w_v^\gamma,w_v,w_v^\beta; \cS^{\dec}_v\right).
	\end{align}
\end{itemize}
From these decorations on the Ulam tree, we can define the corresponding decorated tree (and consider its metric completion) that we write as
\begin{align*}
	\tilde{\cT}^\dec_\alpha:= \sG(\cD).
\end{align*}
Note that the object defined above is not necessarily compact. 
Whenever the underlying block (and hence also the decorated spine thanks to Lemma~\ref{lem:moment diam block implis moment diam spine}) has a moment of order $p>\frac{\alpha}{\gamma}$, a result of \cite[Section~4.2]{senizergues_growing_2022} ensures that the obtained metric space $\tilde{\cT}^\dec_\alpha$ is almost surely compact. 

\paragraph{The uniform measure.}
Assume that the underlying block has a moment of order $p>\frac{\alpha}{\gamma}$ so that $\tilde{\cT}^\dec$ is almost surely compact.
Then the maps $\iota_\cD:\partial \bU \rightarrow \sG(\cD)$ is almost surely well-defined and continuous so we can consider the measure $(\iota_\cD)_*\eta$ on $\sG(\cD)$, the push-forward of the measure $\eta$. 
\paragraph{The block measure.}
If $\beta>\alpha$ and $\Ec{\nu(\cB)}<\infty$ then we can check that the total mass of the $\nu$ measures has finite expectation so it's almost surely finite. Hence we can endow $\tilde{\cT}^\dec_\alpha$ with the measure $\sum_{u}\nu_u$.

\subsubsection{The iterative gluing construction for $\tilde{\cT}^\dec_\alpha$.}
Before diving into the construction of $\tilde{\cT}^\dec_\alpha$, we define a family of time-inhomogeneous Markov chains called Mittag-Leffler Markov chains, first introduced by Goldschmidt and Haas \cite{goldschmidt_line_2015}, see also \cite{senizergues_growing_2022}. 
\paragraph{Mittag-Leffler Markov chains.}
Let $0<\eta<1$ and $\theta>-\eta$. The generalized Mittag-Leffler $\mathrm{ML}(\eta, \theta)$ distribution has $p$th moment
\begin{align}\label{growing:eq:moments mittag-leffler}
\frac{\Gamma(\theta) \Gamma(\theta/\eta + p)}{\Gamma(\theta/\eta) \Gamma(\theta + p \eta)}=\frac{\Gamma(\theta+1) \Gamma(\theta/\eta + p+1)}{\Gamma(\theta/\eta+1) \Gamma(\theta + p \eta+1)}
\end{align}
and the collection of $p$-th moments for $p \in \N$ uniquely characterizes this distribution. 

A Markov chain $(\mathsf M_n)_{n\geq 1}$ has the distribution $\MLMC(\eta,\theta)$ if for all $n\geq 1$,
\begin{equation*}
\mathsf M_n\sim \mathrm{ML}\left(\eta,\theta+n-1\right),
\end{equation*}
and its transition probabilities are characterized by the following equality in law
\begin{equation*}
\left(\mathsf M_n,\mathsf M_{n+1}\right)=\left(\beta_n\cdot \mathsf M_{n+1},\mathsf M_{n+1}\right),
\end{equation*}
where $\beta_n\sim \mathrm{Beta}\left(\frac{\theta+k-1}{\eta}+1,\frac{1}{\eta}-1\right)$ is independent of $\mathsf M_{n+1}$.

\paragraph{Construction of $\tilde{\cT}^\dec_\alpha$.}
We can now express our second construction of $\tilde{\cT}^\dec_\alpha$. 
We start with a sequence $(\mathsf{M}_k)_{k\geq 1}\sim \mathrm{MLMC}\left(\frac{1}{\alpha},1-\frac{1}{\alpha}\right)$.
Then we defined the sequence $(\mathsf{m}_k)_{k\geq 1}=(\mathsf{M}_k-\mathsf{M}_{k-1})_{k\geq 1}$ of increments of that chain, where we assume by convention that $\mathsf{M}_0=0$. 
From there, conditionally on $(\mathsf{M}_k)_{k\geq 1}$ we define an independent sequence $\mathcal{Y}_k$ of metric spaces such that
\begin{align*}
	\mathcal{Y}_k \overset{(d)}{=} \mathrm{Scale}\left(\mathsf{m}_k^\gamma, \mathsf{m}_k,\mathsf{m}_k^\beta; \cS^\dec\right)
\end{align*}
Then, we define a sequence of increasing subtrees of the Ulam tree as follows: we let $\mathtt{T}_1$ be the tree containing only one vertex $v_1=\emptyset$. 
Then if $\mathtt{T}_n$ is constructed, we sample a random number $K_{n+1}$ in $\{1,\dots,n\}$ such that conditionally on all the rest $\Pp{K_n=k}\propto\mathsf{m}_k$. Then, we add the next vertex $v_{n+1}$ to the tree as the right-most child of $v_{K_{n+1}}$.  
The sequence $(\mathtt{T}_n)_{n\geq 1}$ is said to have the distribution of a weighted recursive tree with sequence of weights $(\mathsf{m}_k)_{k\geq 1}$. 
A property of this sequence of trees is the fact that the $\enstq{v_k}{k\geq1}=\mathbb U$. Hence for any $v\in \mathbb U$ we denote $k_v$ the unique $k$ such that $v_k=v$. 
Then we consider the decorations on the Ulam tree defined by
\begin{align}\label{eq:definition Tdec from WRT}
	\widehat{\mathcal{D}}(v)=\mathcal{Y}_{k_v}.
\end{align}
In this setting, we can again define a measure $\widetilde{\eta}$ on the leaves $\partial\mathbb U$ of the Ulam tree by taking the weak limit of the uniform measure on $\{v_1,v_2,\dots v_n\}$ as $n\rightarrow\infty$, the almost sure existence of the limit being guaranteed by \cite[Theorem~1.7, Proposition~2.4]{senizergues_geometry_2021}.

Then, from  \cite[Proposition~3.2]{senizergues_growing_2022} we have 
\begin{align*}
	\left(\left(\mathsf{m}_{k_v}\right)_{v\in \mathbb U}, \widetilde{\eta}\right) \overset{\mathrm{(d)}}{=} \left(\left(w_{v}\right)_{v\in \mathbb U}, \eta\right).
\end{align*}
From that equality in distribution and the respective definitions \eqref{eq:definition Tdec from GEM} of $\mathcal{D}$ and \eqref{eq:definition Tdec from WRT} of $\widetilde{\mathcal{D}}$, it is clear that those two families of decorations have the same distribution.

\subsubsection{Strategy}
The rest of this section is about proving that the random decorated tree $\tilde{\cT}^\dec_\alpha$ that we constructed above has the same distribution as $\cT^\dec_\alpha$. 
For that we are going to characterize their ``finite-dimensional marginals'' and show that they are the same for the two processes:
Using the second description of $\tilde{\cT}^\dec_\alpha$, we can consider for any $k\geq 1$ the subset $\tilde{\cT}^\dec_k$ that corresponds to keeping only the blocks corresponding to $v_1,\dots, v_k$.
We compare this to $\cT^\dec_k$ which is the subset of $\cT^\dec_\alpha$ spanned by $k$ uniform leaves.

\subsection{Finite-dimensional marginals of $\cT_\alpha$ and $\cT^\dec_\alpha$}

\subsubsection{Approximating the decorated tree by finite dimensional marginals}
For any tuple of points $x_1,x_2,\dots x_k$ in $\intervalleff{0}{1}$, we can consider
\begin{align*}
	\mathrm{Span}(X^{\mathrm{exc},(\alpha)};x_1,x_2,\dots,x_k)=\enstq{x\in \intervalleff{0}{1}}{x\preceq x_i, \text{ for some } i\in\{1,2,\dots k\}},
\end{align*}
using the definition of $\preceq$ derived from $X^{\mathrm{exc},(\alpha)}$. 
We further define
\begin{align*}
	\mathrm{Span}(X^{\mathrm{exc},(\alpha)},(\cB_t)_{t\in \intervalleff{0}{1}};x_1,x_2,\dots,x_k)=\bigsqcup_{t\in \mathrm{Span}(X^{\mathrm{exc},(\alpha)};x_1,x_2,\dots,x_k)} \cB_t.
\end{align*}
Using independent uniform random points $(U_i)_{i\geq 1}$ on $[0,1]$, we define
\begin{align*}
	\cT_k=\mathbf{p} \left(\mathrm{Span}(X^{\mathrm{exc},(\alpha)};U_1,U_2,\dots,U_k)\right)
\end{align*}
 and 
\begin{align*}
	 \cT^\dec_k=\mathbf{p}^\dec\left( \mathrm{Span}(X^{\mathrm{exc},(\alpha)},(\cB_t)_{t\in \intervalleff{0}{1}};U_1,U_2,\dots,U_k) \right),
\end{align*}
where $\mathbf{p}:\intervalleff{0}{1} \rightarrow \cT_\alpha$ and $\mathbf{p}^\dec:\sqcup_{t\in \intervalleff{0}{1}} \cB_t \rightarrow \cT_\alpha^\dec$ are the respective quotient maps. 
We say that $\cT_k$ (resp. $\cT^\dec_k$) is the discrete random finite-dimensional marginal of $\cT_\alpha$ (resp. $\cT^\dec_\alpha$), following the standard definition from trees, introduced by Aldous.
\begin{remark}
	Note that for $\cT^\dec_k$ to be almost surely well-defined, we don't need the whole decorated tree $\cT^\dec_\alpha$ to be well-defined as a compact measured metric space. 
	In fact, it is easy to check that if the tail of $\diam(\cB)$ is such that $\Pp{\diam(\cB)\geq x}\sim x^{-p}$ with $1 < p <\frac{\alpha }{\gamma}$, then $\sup_{v \in \mathbb B} \Delta_v^\gamma \cdot \diam(\cB_v)=\infty$, even though the distance of a random point to the root in $\cT^\dec_\alpha$ is almost surely finite. 
\end{remark}

\begin{lemma}
	Identifying $\cB_v$ for any $v\in \mathbb B$ as a subset of $\cT^\dec_\alpha$ we almost surely have 
	\begin{align*}
		\bigcup_{v\in \mathbb B}\cB_v \subset \overline{\bigcup_{n\geq 0}\cT^\dec_n}.
	\end{align*}
\end{lemma}
\begin{proof}
	By properties of the stable excursion, for any $t\in \mathbb B$ the set $\enstq{s\succeq t}{s\in \intervalleff{0}{1}}$ has a positive Lebesgue measure. Hence almost surely there is some $k$ such that $t\preceq U_k$.
\end{proof}
\begin{corollary}\label{cor:Tdec is compact iff union Tndec is}
	The space $\overline{\bigcup_{n\geq 0}\cT^\dec_n}$ is compact if and only if $\cT^\dec_\alpha$ is well-defined as a compact metric space. 
\end{corollary}

\subsubsection{Description of $\cT_k$ and $\cT_k^\dec$}
First, for any $k\geq 1$, we let $L_k$ be the distance $\mathrm{d}(U_k,\cT_{k-1})$ computed in the tree $\cT_\alpha$. 
Then we consider the set of branch-points $\mathbb B \cap (\cT_k\setminus\cT_{k-1})$ and define the decreasing sequence $(N_k(\ell))_{\ell \geq 1}$ which is the decreasing reordering of the sequence $\left(\Delta_t\right)_{t\in \mathbb B \cap (\cT_k\setminus\cT_{k-1})}$. Denote $N_k= \sum_{\ell \geq 1} N_k(\ell)$.
Note that the jumps of the excursion process are all distinct almost surely. 
We denote $(t_k(\ell))_{\ell\ge 1}$ the corresponding sequence of jump times.
We also consider the sequence $(L_k(\ell))_{\ell\ge 1}$ defined as $\mathrm{d}(t_k(\ell), \cT_{k-1})$. 
Let us also write $N_k^{r}(\ell)$ for the quantity $x_{t_k(\ell)}^{U_k}$, defined in \eqref{eq:definition xts}.

Let us check that these random variables are enough to reconstruct $\cT_k^\dec\setminus \cT_{k-1}^\dec$. 
We have 
\begin{align*}
	\overline{\cT_k^\dec\setminus \cT_{k-1}^\dec}= \overline{\bigcup_{t\in \cT_k\setminus \cT_{k-1}} \cB_t}
\end{align*}
seen as a subset of $\cT^\dec_\alpha$.
Now, from the form of the distance on $\cT^\dec_\alpha$, the induced distance between two point in $\cT_k^\dec\setminus \cT_{k-1}^\dec$ only depends on:
\begin{itemize}
	\item the ordering of the jump times $(t_k(\ell))_{\ell \geq 1}$ by the relation $\preceq$ (they are all comparable because by definition we have $t_k(\ell)\preceq U_k$ for all $\ell \geq 1$),
	\item the sizes of the jumps $(\Delta_{t_k(\ell)})$ and the corresponding blocks $\cB_{t_k(\ell)}$,
	\item the position of the gluing points $Z_{t_k(\ell)}^{U_k}$ on the corresponding block  $\cB_{t_k(\ell)}$. 
\end{itemize}
We can also check that the topological closure $\partial(\cT_k^\dec\setminus \cT_{k-1}^\dec)$ contains a single point, we call this point $Z_k$. When considering the compact object $\overline{\cT_k^\dec\setminus \cT_{k-1}^\dec}$, we see it as rooted at $Z_k$.

Now, if we want to reconstruct the entire $\cT_k^\dec$ from $(\cT_k^\dec\setminus \cT_{k-1}^\dec)$ and $\cT_{k-1}^\dec$ we also need some additional information. 
For that we consider the finite sequence $U_1 \wedge U_k, U_2\wedge U_k, \dots , U_{k-1}\wedge U_k$ which are all $\preceq U_k$ by definition. Because they are all comparable, we can consider $V_k$ the maximal element of this sequence. 
We let $I_k$ be the unique $i\leq k-1$ such that $V_k\in (\cT_i^\dec\setminus \cT_{i-1}^\dec)$. 
Additionally, we let $W_k=u^{U_k}_{V_k}$. 
We can check that the corresponding point $Z_k=Y_{V_k,W_k}$ is such that, conditionally on $\cT_{k-1}^\dec$ and $V_k$, distributed as a random point under the measure $\nu_{V_k}$ carried by the block $\cB_{V_k}$, by definition. 

Now, we use the results of \cite[Proposition~2.2]{goldschmidt_stable_2018} that identify the joint law of these quantities as scaling limit of analogous quantities defined in a discrete setting on trees constructed using the so-called Marchal algorithm.
Those quantities have been studied with a slightly different approach in \cite{senizergues_growing_2022}, which provides an explicit description of those random variables, jointly in $k$ and $\ell$. 
The following can be read from \cite[Lemma~5.5, Proposition~5.7]{senizergues_growing_2022}:
\begin{itemize}
	\item the sequence $(N_k)_{k\geq 1}$ has the distribution of the sequence $(\mathsf{m}_k)_{k\geq 1}=(\mathsf{M}_k-\mathsf{M}_{k-1})_{k\geq 1}$ where $(\mathsf{M}_k)_{k\geq 1}\sim \mathrm{MLMC}\left(\frac{1}{\alpha},1-\frac{1}{\alpha}\right)$.
	\item the sequences $\left(\frac{N_k(\ell)}{N_k}\right)_{\ell \geq 1}$ are i.i.d. with law $\mathrm{PD}(\alpha-1,\alpha -1)$ and $L_k= \alpha \cdot N_k^{\alpha -1} \cdot S_k$ where $S_k$ is the $(\alpha-1)$ diversity of $\left(\frac{N_k(\ell)}{N_k}\right)_{\ell \geq 1}$. 
	\item the random variables $\frac{L_k(\ell)}{L_k}$ are i.i.d. uniform on $\intervalleff{0}{1}$,
	\item the random variables $\frac{N_k^{r}(\ell)}{N_k(\ell)}$ are i.i.d. uniform on $\intervalleff{0}{1}$.
\end{itemize} 
In particular, from the above description, we can check that conditionally on the sequence $(N_k)_{k\geq 1}$, the random variables $(\overline{\cT_k^\dec\setminus \cT_{k-1}^\dec })_{k\geq 1}$ are independent with distribution given by 
\begin{align*}
	\overline{\cT_k^\dec\setminus \cT_{k-1}^\dec} = \mathrm{Scale}(N_k^\gamma, N_k, N_k^\beta; \cS^\dec).
\end{align*} 
We have identified the laws of the two sequences $(\cT_k^\dec\setminus \cT_{k-1}^\dec)_{k\geq 1}$ and $(\tilde{\cT}_k^\dec\setminus \tilde{\cT}_{k-1}^\dec)_{k\geq 1}$. 
In order to identify the law of the sequences $(\cT_k^\dec)_{k\geq 1}$ and  $(\tilde{\cT_k}^\dec)_{k\geq 1}$, we just have to check that the attachment procedure is the same.

Recall the definition of the random variables $(I_k)_{k\geq 1}$, $(V_k)_{k\geq 1}$ and $(W_k)_{k\geq 1}$.
Still from \cite[Lemma~5.5, Proposition~5.7]{senizergues_growing_2022}, conditionally on all the quantities whose distributions were identified above, the $I_k$'s are independent and $\Pp{I_k=i}=\frac{N_i}{N_1+\dots + N_{k-1}}$ for all $i\leq k-1$. 
From those random variables we can construct an increasing sequence of trees $(\mathtt{S}_n)_{n\geq 1}$ in such a way that the parent of the vertex with label $k$ is the vertex with label $I_k$. 
From the observation above, the law of $(\mathtt{S}_n)_{n\geq 1}$ conditionally on everything else mentioned before only depends on $(N_k)_{k\geq 1}$.
This law is the same as that of $(\mathtt{T}_n)_{n\geq 1}$ conditionally on $(\mathsf m_k)_{k\geq 1}$ (and everything else). 

Now we just need to check the law of the gluing points: recall how the point $Z_k$ is given as $Z_k=Z_{V_k}^{U_k}=Y_{V_k,W_k}$. 
We just need to argue that this point, conditionally on $I_k=i$ the point $Z_k$ is taken under a normalized version of the measure $\sum_{\ell=1}^{\infty} \Delta_{t_i(\ell)} \mu_{t_i(\ell)}$.
In fact, stil from \cite{senizergues_growing_2022}, conditionally on $I_k=i$ and everything else mentioned before we have  $\Pp{V_k=t_i(\ell)}=\frac{N_i(\ell)}{N_i}=\frac{ \Delta_{t_i(\ell)}}{\sum_{\ell=1}^{\infty} \Delta_{t_i(\ell)}}$ for $\ell\geq 1$ and $W_k$ is independent uniform on $\intervalleff{0}{1}$. 
Since by definition, $W_k\in \cA_{V_k}$, then $Y_{V_k,W_k}$ is distributed as $\mu_{V_k}$ by construction.
Now, since the uniform distribution has no atom, it is almost surely the case that $W_k \notin \enstq{u_{V_k}^{U_r}}{r\leq k-1, V_k\preceq U_r}$ so that the point $Y_{V_k,W_k}$ has no been used in the construction up to time $k-1$, so the sampling of $Z_k$ is indeed independent of the rest. 

\subsection{Properties of the construction.}
We introduce the set of leaves $\cL= \cT^\dec_\alpha\setminus \cup_{n\geq 1}\cT^\dec_n$. 
Then still from \cite{senizergues_growing_2022} we have the following
\begin{theorem} \label{thm:compact_hausdorff}
	If $\Ec{\diam(\cB)^p}<\infty$ for some $p>\frac{\alpha}{\gamma}$, then $\cT^\dec_\alpha$ is almost surely a compact metric space and $\Ec{\diam(\cT^\dec_\alpha)^p}<\infty$. 
	Under the assumption that the measure on $\cB$ is not concentrated on its root almost surely, the Hausdorff dimension of the set of leaves $\mathcal L$ is given by
	\begin{align*}
		\dim_H(\cL)=\frac{\alpha}{\gamma}
	\end{align*}
	almost surely. 
\end{theorem}
We can now provide a proof of Lemma~\ref{l:smalldec} which can just be seen as a corollary of the above theorem. 
	\begin{proof}[Proof of Lemma~\ref{l:smalldec}]
		Introduce the random block $\widehat{\cB}$ defined on the same probability space as $\cB$ as the interval $\intervalleff{0}{\diam \cB}$ rooted at $0$ and whose sampling measure is a Dirac point mass at $\diam \cB$. 
		We also consider the corresponding object $\widehat{\cT}^\dec$. 
		Because $\widehat{\cB}$ also satisfies the assumptions of the theorem above, $\widehat{\cT}^\dec$ is almost surely compact and by Corollary~\ref{cor:Tdec is compact iff union Tndec is} we have $\mathrm{d}_H(\widehat{\cT}^\dec, \widehat{\cT}^\dec_n)\rightarrow 0$ almost surely as $n\rightarrow\infty$.
		Then we can check that, for uniform random variables $U_1,U_2,\dots,U_n$ that were used to define $\widehat{\cT}^\dec_n$, we have
		\begin{align*}
			\sup_{s\in \intervalleff{0}{1}} \sum_{t\prec s} \diam(\cB_t)\ind{\Delta_t<\delta} \leq\sup_{s\in \{U_1, U_2,\dots, U_n\}} \sum_{t\prec s} \diam(\cB_t)\ind{\Delta_t<\delta}  + \mathrm{d}_H(\widehat{\cT}^\dec, \widehat{\cT}^\dec_n)
		\end{align*}
		Thanks to the above theorem, $\widehat{\cT}^\dec$ is compact and the second term tends to $0$ as $n\rightarrow \infty$. 
		Then, for a large fixed $n$, the first term tends to $0$ as $\delta\rightarrow0$. 
	\end{proof}

\section{Applications} \label{s:applications}

We present several applications of the invariance principle in Theorem~\ref{th:invariance} to block-weighted models of random discrete structures. The limits of these objects are stable trees decorated by other stable trees, stable looptrees, or even Brownian discs.

\subsection{Marked trees and iterated stable trees} \label{s:markedtrees}

Define a class of combinatorial objects $\cM$ consisting of all marked rooted plane trees where the root and each leaf receives a mark 
and internal vertices may or may not receive a mark. The size of an object $M$ in $\cM$ is its number of leaves and is denoted by $|M|$. For a given $n$ there are infinitely many trees with size $n$ so for simplicity we assume that there are no vertices of out-degree 1, rendering the number of trees of size $n$ finite.
\begin{figure} [!h]
	\centerline{\scalebox{0.7}{\includegraphics{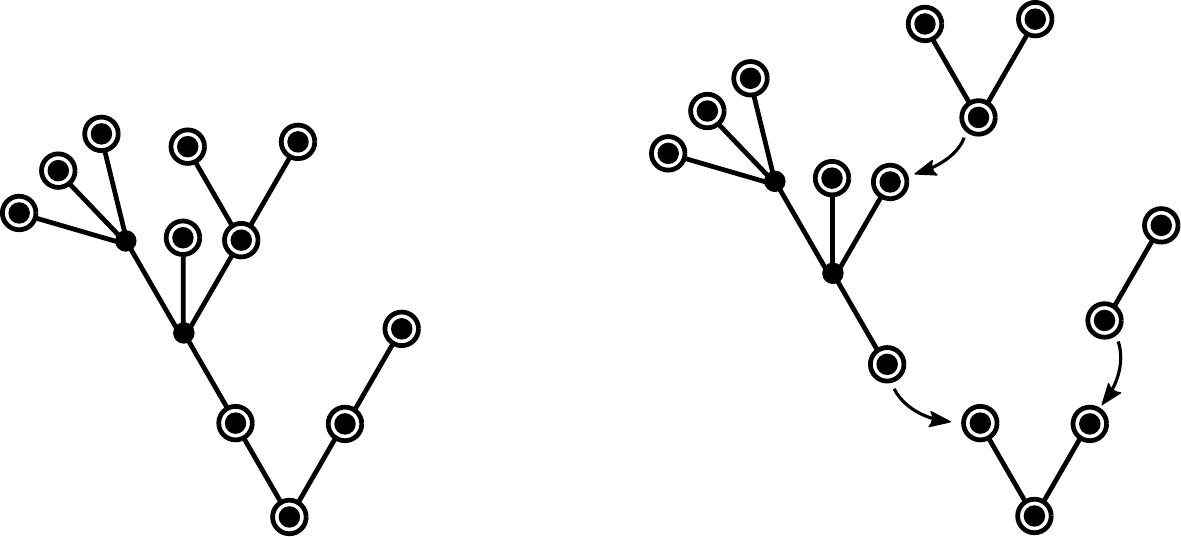}}} 
	\caption{On the left is a tree from $\cM$ and on the right it is shown how it is decomposed into its tree blocks. The root is at the bottom and marked vertices are denoted by black and white circles.} \label{f:treeblocks}
\end{figure} 
A subclass of $\cM$, which we denote by $\cL$, consists of trees where no internal vertex, except the root, receives a mark, and so an $n$-sized object from $\cL$ is a rooted plane tree with $n$ leaves.  The ordinary generating series of the classes $\cM$ and $\cL$ satisfy the equation
\begin {align}\label{eq:iso1}
\cM(z) = z + \cL(\cM(z))-\cM(z).
\end {align}
The interpretation is that either an object from $\cM$ is a single vertex (the root of the tree) or a tree in $\cL$, different from a single root, in which each marked vertex is identified with the root of an object from $\cM$.

We will call a subtree of an element $M$ from $\cM$ a \emph{tree block} if it has more than one vertex, all of its leaves are marked, its root is marked and no other vertices are marked. We further require that it is a maximal subtree with this property. The tree blocks may also be understood as the subtrees obtained by cutting the marked tree at each marked internal vertex, see Fig.~\ref{f:treeblocks}.

\begin {remark}
\label{re:iteriter}
One may introduce marks with $k$ different colors and define 'color blocks' and assign weights to them. This model is a candidate for a discrete version of further iterated trees $\mathcal{T}_{\alpha_1,\alpha_2,\ldots,\alpha_{k+1}}$. Denote the set of such structures with marks of $k$ colours by $\cM_k$, with $\cL = \cM_0$ and $\cM = \cM_1$. One then has an equation of generating series
\begin {equation}\label{eq:iso2}
\cM_k(z) = z + \cM_{k-1}(\cM_k(z))-\cM_k(z).
\end {equation}
\end {remark}

\begin {remark}
It is a standard result (and easy to check) that
\begin {equation*}
\cL(z) = \frac{1+z-\sqrt{z^2-6z+1}}{4} = z + z^2 + 3z^3 + 11 z^4 + 45 z^5 + 197 z^6 + \ldots
\end {equation*}
and from this and Equation~\eqref{eq:iso1} one may deduce that
\begin {equation*}
\cM(z) = \frac{1+7z-\sqrt{z^2-10z+1}}{12} = z + z^2 + 5z^3+31z^4 + 215z^5 +1597 z^6 + \ldots
\end {equation*}
This sequence of coefficients is not in the OEIS. Going further one finds that 
\begin {equation*}
\cM_2(z) = \frac{1+17z-\sqrt{z^2-14z+1}}{24} = z + z^2 + 7z^3+61z^4 + 595z^5 +6217 z^6 + \ldots
\end {equation*}
In general, by  induction
\begin {align*}
\cM_k(z) &=  \frac{1+(2k^2+4k+1)z-\sqrt{z^2-(4k+6)z+1}}{2 (k+1)(k+2)} \\
&= z + z^2 + (2(k+1)+1)z^3 + (5(k+1)^2+5(k+1)+1)z^4 \\
& + (14(k+1)^3+21(k+1)^2+9(k+1)+1)z^5 \ldots
\end {align*}
Note that $[z^n (k+1)^m]\cM_k(z)$ is the number of rooted plane trees with no vertex of outdegree 1 which have $n$ leaves and $m$ internal vertices (not counting the root). In particular it holds that
\begin {equation*}
\sum_{m=0}^{n-2} [z^n (k+1)^m]\cM_k(z) = [z^n] \cL(z).
\end {equation*}

\end {remark}

To each element $L$ of $\cL$ we assign a weight $\gamma(L)$ and denote the class of such weighted structures by $\cL^\gamma$. We assume the weight of the marked tree consisting of a single vertex is equal to $1$.
Then, we assign a weighting $\omega$ to elements $M$ from $\cM$ by 
\begin {equation*}
\omega(M) = \prod_{L} \gamma(L)
\end {equation*}
where the index $L$ ranges over the tree blocks in $M$. Denote the corresponding class of weighted structures by $\cM^\omega$. The weighed ordinary generating series satisfy a similar equation as before
\begin {align*}
\cM^\omega(z) = z + \cL^\gamma(\cM^\omega(z)) - \cM^\omega(z).
\end {align*} 
Define a random element $\mM_n^\omega$ from the set of $n$-sized elements of $\cM^\omega$ which is selected with a probability proportional to its weight. Let $(\iota_n)_{n\geq 0}$ and $(\zeta_n)_{n\geq 0}$ be sequences of non-negative weights, with $\iota_0=\zeta_1=1$ and $\iota_1 = \zeta_1 = 0$. We will be interested in weights $\gamma$ of the form
\begin {equation*}
\gamma(L) = \zeta_{|L|} W(L)
\end {equation*}
with
\begin{figure} [b!]
	\centerline{\scalebox{0.45}{\includegraphics{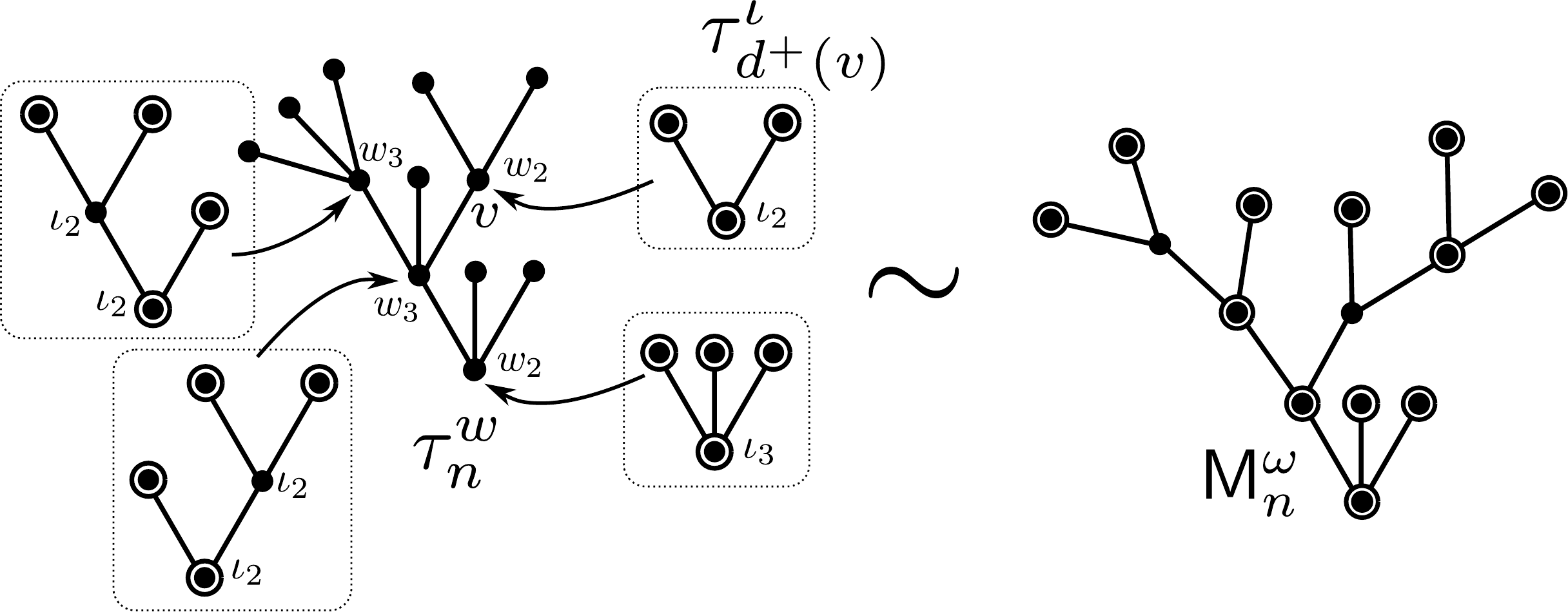}}} 
	\caption{The coupling of $\mM_n^\omega$ with simply generated trees with leaves as atoms.} \label{f:iso}
\end{figure} 

\begin {equation*}
W(L) = \prod_{v \text{ internal vertex in } L} \iota_{d^+(v)}.
\end {equation*}

Let $(w_n)_{n\geq 0}$ be a sequence of non-negative numbers such that $w_0=1$ and $w_1=0$. For $n\geq 2$ we will write the weights $\zeta_n$ as follows
\begin {align*}
\zeta_n = w_n Z_n^{-1}
\end {align*}
where
\begin {equation*}
Z_n = \sum_{\substack{L\in\cL, |L| = n}} W(L)
\end {equation*}
is the \emph{partition function} of elements from $\cL^\gamma$ of size $n$. In particular, when we choose $\zeta_n = 1$ (i.e.~$w_n = Z_n$) for all $n\geq 2$, we say that $\mM_n^\omega$ is an $n$-leaf \emph{simply generated marked tree} with branching weights $(\iota_k)_k$.

\begin {proposition}
\label{pro:samplingmnw}
The random element $\mM_n^\omega$ may be sampled as follows:
\begin {enumerate}
\item Sample an $n$-leaf simply generated tree $\tau_n = \tau_n^\omega$ with branching weights $[z^{k}]\cL^{\gamma}(z) = w_k$ assigned to each internal vertex of out-degree $k>1$ and branching weight $w_0=1$ assigned to each leaf.

\item For each vertex $v$ of $\tau_n$ sample independently a $d^+(v)$-leaf simply generated tree $\delta(v)$ with branching weights $\iota_k$ assigned to each vertex of outdegree $k\geq 0$.  Glue together according to the tree structure of $\tau_n$ (see Fig.~\ref{f:iso}).
\end {enumerate}
\end {proposition}

We refer to the surveys~\cite{MR2908619,zbMATH07235577} for details on simply generated models of trees with fixed numbers of vertices or leaves. The proof of Proposition~\ref{pro:samplingmnw} is by a straight-forward calculation, or alternatively by applying a general result on sampling trees with decorations~\cite[Lem. 6.7]{zbMATH07235577}.
We will only be interested in the case where  $\cM^\omega(z)$ has positive radius of convergence. In this case one may rescale the weights $w_n$ and $\iota_n$ such that they are probabilities without affecting the distribution of $\mM_n^\omega$, and we will assume in the following that this has been done. We denote by $\xi$ a random variable with distribution $w_n$ and $\chi$ a random variable with distribution $\iota_n$. The tree $\tau_n$ may then be viewed as a Bienaymé--Galton--Watson tree with offspring distribution $\xi$ conditioned on having $n$ leaves and $\delta(v)$ may be viewed as a Bienaymé--Galton--Watson tree with offspring distribution $\chi$ conditioned on having $d^+(v)$ leaves.

We let $d_{\mM_n^\omega}$ denote the graph-distance on $\mM_n^\omega$ and $\nu_{\mM_n}$ the counting measure on its set of non-root vertices. Proposition~\ref{pro:samplingmnw} ensures that $(\mM_n^\omega, d_{\mM_n^\omega}, \nu_{\mM_n})$ falls into the framework of discrete decorated trees described in Section~\ref{s:disc_construction}, with the random tree given by $\tau_n$ and decorations according to a sequence $(\tilde B_k,\tilde \rho_k,\tilde d_k,\tilde \ell_k,\tilde \nu_k)_{k\geq 0}$ given as follows. The space $\tilde{B}_k$ is a $\chi$-BGW-tree conditioned on having $k$ leaves, $\tilde{\rho}_k$ is its root vertex, and $\tilde{d}_k$ is the graph distance on that space. The labeling function $\tilde{\ell}_k$ is chosen to be some bijection between $\{1, \ldots, k\}$ and the leaves of $\tilde{B}_k$. 
Thus $\mu_k$ is the uniform measure on the $k$ leaves of $\tilde{B}_k$.
The measure $\tilde{\nu}_k$ is the counting measure on the non-root vertices of $\tilde{B}_k$. 

Choose $(w_n)_{n\geq 0}$ and $(\iota_n)_{n\geq 0}$ such that $\Ex{\xi}= \Ex{\chi} = 1$, and such that $\xi, \chi$ follow asymptotic power laws so that $\xi$ is in the domain of attraction of a stable law with index $\alpha_2 \in ]1,2[$ and $\chi$ lies in the domain of attraction of a stable law with index $\alpha_1 \in ]1,2[$.  By~\cite[Thm. 6.1, Rem. 5.4]{MR2946438} there is a sequence 
\begin{align*}
	b_n = \left( \frac{|\Gamma(1-\alpha_2)|}{\Pr{\xi=0}} \right)^{1/\alpha_2}\inf \{x \ge 0 \mid \Pr{\xi > x} \le 1/n \} \sim c_1 n^{1/\alpha_2}
\end{align*}
for some constant $c_1>0$ such that
\begin{align*}
	(b_n^{-1} W_{\lfloor|\tau_n|t\rfloor}(\tau_n))_{0 \le t \le 1} \convdis  X^{\text{exc}, (\alpha_2)}.
\end{align*}
The number of vertices of $\tilde{B_k}$ concentrates at $k / \Pr{\chi=0}$, and by~\cite[Thm. 5.8]{MR2946438} it follows that for some constant $c_2>0$
\begin{align*}
	\left(\tilde B_k,\tilde\rho_k, c_2 k^{-1-1/\alpha_1} \tilde d_k, \mu_k,(k/\Pr{\chi=0})^{-1}\tilde \nu_k\right) \to (\cT_{\alpha_1},\rho_{\cT_{\alpha_1}}, d_{\cT_{\alpha_1}},\nu_{\cT_{\alpha_1}},\nu_{\cT_{\alpha_1}})
\end{align*}
with $(\cT_{\alpha_1}, \rho_{\cT_{\alpha_1}},  d_{\cT_{\alpha_1}}, \nu_{\cT_{\alpha_1}})$ denoting the $\alpha_1$-stable tree with root vertex $\rho_{\cT_{\alpha_1}}$. For   $\alpha_1 > \frac{1}{2 - \alpha_2}$ we may  use the construction from Section~\ref{s:cts_construction} to form a decorated stable tree $(\cT_{\alpha_1,\alpha_2}, d_{\cT_{\alpha_1,\alpha_2}}, \nu_{\cT_{\alpha_1,\alpha_2}})$ with  distance exponent $\gamma_1 := 1 - 1 / \alpha_1$, obtained by blowing up the branchpoints of the $\alpha_2$-stable tree with rescaled independent copies of the $\alpha_1$-stable tree. The measure $\nu_{\cT_{\alpha_1,\alpha_2}}$ is taken to be the push-forward of the Lebesgue measure, corresponding to the case $\beta=1$. Theorem~\ref{th:invariance} yields the following scaling limit:

\begin {theorem}\label{thm:mainconv}
Choose $(w_n)_{n\geq 0}$ and $(\iota_n)_{n\geq 0}$ such that $\Ex{\xi}= \Ex{\chi} = 1$, and such that $\xi, \chi$ follow asymptotic power laws so that $\xi$ is in the domain of attraction of a stable law with index $\alpha_2 \in ]1,2[$ and $\chi$ lies in the domain of attraction of a stable law with index $\alpha_1 \in ]1,2[$.  Suppose that $\alpha_1 > \frac{1}{2 - \alpha_2}$. Then
\begin {equation*}
\left(\mM_n^\omega, c_1^{-1+1/\alpha_1} c_2  n^{-\frac{\alpha_1-1}{\alpha_2 \alpha_1}} d_{\mM_n^\omega}, \frac{1}{|\mM_n^\omega|}\nu_{\mM_n^\omega}\right)  \convdis (\cT_{\alpha_1,\alpha_2}, d_{\cT_{\alpha_1,\alpha_2}}, \nu_{\cT_{\alpha_1,\alpha_2}})
\end {equation*}
in the Gromov--Hausdorff--Prokhorov topology.
\end {theorem}

By Theorem~\ref{thm:compact_hausdorff} and standard arguments we also obtain the Hausdorff-dimension of the iterated stable tree:

\begin {theorem}
For every $\alpha_2 \in (1,2)$ and $\alpha_1 \in (\frac{1}{2-\alpha_2},2)$, almost surely the random tree $\mathcal{T}_{\alpha_1,\alpha_2}$ has Hausdorff dimension $\displaystyle\frac{\alpha_2 \alpha_1}{\alpha_1-1}$.
\end {theorem}

The construction may be iterated: With $\gamma_2 := \gamma_1 / \alpha_2$ we may choose any $\alpha_3 \in ]1, 1+\gamma_2[$ and build the iterated stable tree $\cT_{\alpha_1, \alpha_2, \alpha_3}$ as in Section~\ref{s:cts_construction} with distance exponent~$\gamma_2$, by blowing up the branchpoints of an $\alpha_3$-stable tree by rescaled independent copies of the iterated stable tree $\cT_{\alpha_1, \alpha_2}$. By Remark~\ref{re:iteriter} and analogous arguments as for Theorem~\ref{thm:mainconv}, the tree $\cT_{\alpha_1, \alpha_2, \alpha_3}$ arises as scaling limits of random finite marked trees with diameter having order $n^{\gamma_3}$ for $\gamma_3 := \gamma_2/ \alpha_3$. This construction may be iterated indefinitely often, by choosing $\alpha_4 \in ]1, 1 + \gamma_3[$ and setting $\gamma_4 = \gamma_3 / \alpha_4$, and so on. This yields a sequence $(\alpha_i)_{i \ge 1}$ so that $\cT_{\alpha_1, \ldots,  \alpha_k}$ is well-defined for all $k \ge 2$. We pose the following question:

\begin{question}
	Is there a non-trivial scaling limit for $\cT_{\alpha_1, \ldots,  \alpha_k}$ as $k \to \infty$?
\end{question}

Note that the associated sequence $(\gamma_i)_{i \ge 1}$ is strictly decreasing and satisfies $\gamma_{i+1} < \gamma_i/ (\gamma_i +1)$, which yields $\lim_{i \to \infty} \gamma_i= 0$ and hence $\lim_{i \to \infty} \alpha_i= 1$. The intuition for the question is that if $\alpha$ is close to $1$, then the vertex with maximal width in $\cT_\alpha$ should dominate and hence blowing up all branchpoints by some decoration should yield something close to a single rescaled version of that decoration. Hence $\cT_{\alpha_1, \ldots,  \alpha_{k+1}}$ \emph{should} be close to a constant multiple of  $\cT_{\alpha_1, \ldots,  \alpha_{k}}$ for $k$ large enough.
Since $\cT_{\alpha_1, \ldots,  \alpha_{k}}$ has Hausdorff dimension $\alpha_1 \cdots \alpha_k/(\alpha_1-1)$, we expect that $(\alpha_i)_{i \ge 1}$ needs to be chosen so that $\prod_{i=1}^\infty \alpha_i$ converges in order to get a scaling limit.

\subsection{Weighted outerplanar maps}

\label{sec:outer}

Planar maps may roughly be described as drawings of connected graphs on the $2$-sphere, such that edges are represented by arcs that may only intersect at their endpoints.  The connected components that are created when removing the map from the plane are called the faces of the map. The number of edges on the boundary of a face is its degree. In order avoid symmetries, usually a root edge is distinguished and oriented. The origin of the root-edge is called the root vertex. The face to the right of the root edge is called the outer face, the face to the left the root face. 

An outerplanar map is a planar map where all vertices lie on the boundary of the outer face.  The geometric shape of outerplanar maps has received some attention in the literature~\cite{zbMATH06673644,zbMATH06729837,zbMATH07138334,zbMATH07235577}. Throughout this section we fix  two sequences   $\iota = (\iota_k)_{k \ge 3}$ and $\kappa = (\kappa_k)_{k \ge 2}$ of non-negative weights. We are interested in random outerplanar maps that are generated according to $\kappa$-weights on their blocks and $\iota$-weights on their faces.  Our goal in this section is to describe phases in which we obtain decorated stable trees as scaling limits. Specifically, we will obtain stable trees decorated with looptrees and Brownian trees. This is motivated by the mentioned work~\cite{zbMATH07138334}, which described a phase transition of  random face-weighted outerplanar maps from a deterministic circle to the Brownian tree via loop trees. By utilizing the second weight sequence $\kappa$  we hence obtain a completely different phase diagram.

We will only consider outerplanar maps without multi edges or loop edges. Recall that a block of a graph is a connected subgraph $D$ that is maximal with the property that removing any of its vertices does not disconnect $D$.
 Blocks of outerplanar maps are precisely dissections of polygons. We define the weight of a dissection $D$ by
\begin{align}
	\gamma(D) = \kappa_{|D|} \prod_F \iota_{\mathrm{deg}(F)},
\end{align}
with $|D|$ denoting the number of vertices of $D$, the index $F$ ranging over the faces of $D$, and $\mathrm{deg}(F)$ denoting the degree of the face $F$. This includes the case where $D$ consists of two vertices joined by a single edge. The weight of an outerplanar map $O$ is then defined by
\begin{align}
	\label{eq:omegao}
	\omega(O) = \prod_{D} \gamma(D),
\end{align}
with the index $D$ ranging over the blocks of $O$. Given $n \ge 1$ such that at least one $n$-sized outerplanar map has positive $\omega$-weight, this allows us to define a random outerplanar map $O_n$ that is selected with probability proportional to its $\omega$-weight among the finitely many outerplanar graphs with $n$ vertices. We will only consider the case where infinitely many integers $n$ with this property exist. Likewise, we define the random dissection $D_n$ that is sampled with probability proportional to its weight among all dissections with $n$ vertices.

The random outerplanar map $O_n$ fits into the framework considered in the present work. We will make this formal. For each $k \ge 2$ let us set 
\begin{align}
	\label{eq:defd}
	d_k =   \sum_{D : |D| =k} \gamma(D),
\end{align}
with the sum index $D$ ranging over all dissections of $k$-gons. We will only consider the case where the power series $D(z) := \sum_{k \ge 2} d_k z^k$ has positive radius of convergence $\rho_D > 0$. For any $t>0$ with $D(t) < t$  we define the probability weight sequence 
\begin{align}
	(p_i(t))_{i \ge 0} = (1 - D(t)/t, 0, d_2t , d_3t^2, d_4 t^3, \ldots).
\end{align}
Its first moment is given by $\sum_{k \ge 2} k d_k t^{k-1}$, and we set
\begin{align}
	m_O = \lim_{t \nearrow \rho_D} \sum_{k \ge 2} k d_k t^{k-1} \in ]0, \infty].
\end{align}
If $m_O \ge 1$, then there is a unique $0 < t_O \le \rho_D$ for which $ \sum_{k \ge 2} k d_k t_O^{k-1} = 1$. If $m_O < 1$ we set $t_O = \rho_D$. Furthermore, we let $\xi$ denote a random non-negative integer with probabilities
\begin{align}
	\label{eq:defxi}
	\Pr{\xi= i} = p_i(t_O), \qquad i \ge 0.
\end{align}

Let $T_n$ denote a $\xi$-BGW tree conditioned on having $n$ leaves. Note that  $T_n$ has no vertex with outdegree $1$. Let $g_n>0$ be a positive real number. For each $k \ge 2$ let  $(\tilde B_k,\tilde d_k,\tilde\rho_k,\tilde\ell_k,\tilde \nu_{n,k})$ denote the decoration with $\tilde{B}_k = D_k$, $\tilde d_k $ the graph distance on $D_k$ multiplied by $g_n$, $\tilde{\rho}_k$ the root vertex of $D_k$, $\tilde{\ell}_k: \{1, \ldots, k\} \to D_k$ any fixed enumeration of the $k$ vertices of $D_k$, and $\tilde{\nu}_k$ the counting measure on the non-root vertices (that is, all vertices except for the origin of the root edge) of $D_k$. We let $\tilde{B}_0$ denote a trivial one point space with no mass.

\begin{figure}[t]
	\centering
	\begin{minipage}{\textwidth}
		\centering
		\includegraphics[width=0.7\linewidth]{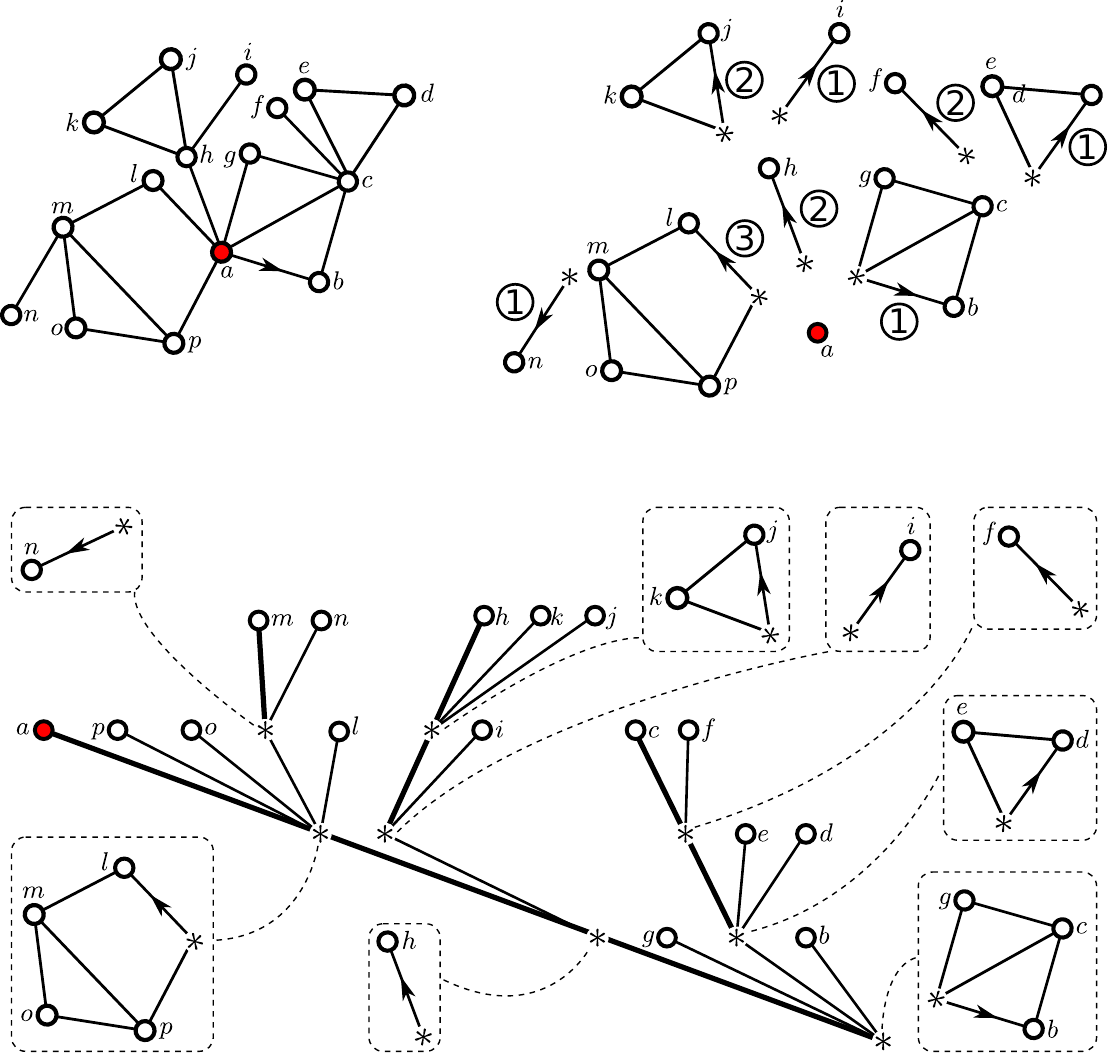}
		\caption[]%
		{Correspondence of an outerplanar map (top left corner) to a tree decorated by dissections of polygons (bottom half).\footnote{Source of image: \cite{zbMATH07138334}}}
		\label{f:outerdec}
	\end{minipage}
\end{figure}

\begin{lemma}[{\cite[Lem. 6.7, Sec. 6.1.4]{zbMATH07235577}}]
	\label{le:outerdeco}
The decorated tree $(T_n^\dec,d_n^\dec,\nu_n^\dec)$ is distributed like the random weighted outerplanar map $O_n$ equipped with the graph distance multiplied by $g_n$ and the counting measure on the non-root vertices.
\end{lemma}
\noindent See Figure~\ref{f:outerdec}  for an illustration. We refer to the cited sources for detailed justifications.

Similarly as for parameter $m_O$, we let $\rho_\iota$ denote the radius of convergence of the series $\sum_{k \ge 2} \iota_{k+1} z^{k}$ and set
\begin{align*}
	m_D := \lim_{t \nearrow \rho_\iota} \sum_{k \ge 2} k \iota_{k+1} t^{k-1} \in ]0, \infty].
\end{align*}
If $m_D\ge 1$, there is a unique constant $0 < t_D \le \rho_\iota$ such that $\sum_{k \ge 2} k \iota_{k+1} t_D^{k-1} = 1$. For $m_D < 1$ we set $t_D = \rho_\iota$. Furthermore, we let $\zeta$ denote a random non-negative integer with distribution given by the probability weight sequence
\begin{align*}
	(q_i)_{i \ge 0} := (1 - \sum_{k \ge 2} \iota_{k+1} t_D^{k-1}, 0, \iota_{3} t_D, \iota_4 t_D^2, \ldots).
\end{align*}

We list three known non-trivial scaling limits for the random face-weighted dissection~$D_n$. If $m_D>1$, then by~\cite{MR3382675} $D_n$ lies in the universality class of the Brownian tree $\mathcal{T}_2$, with a diameter of order $\sqrt{n}$. (That is, $\mathcal{T}_2$ is given by $\sqrt{2}$ times the random real tree obtained by identifying points of a Brownian excursion of duration $1$.) 
 Specifically, 
\begin{align}
	\label{eq:d1}
	\frac{\sqrt{2}}{\sqrt{ \Va{\zeta} q_0}} \frac{1}{4} \left( \Va{\zeta} + \frac{q_0\Pr{\zeta \in 2 \ndN_0}}{2\Pr{\zeta \in 2 \ndN_0} - q_0 } \right) \frac{1}{\sqrt{n}} D_n \convdis \mathcal{T}_2.
\end{align}

If $m_D = 1$ and $\Pr{\zeta \ge n} \sim c_D n^{-\eta}$ for $1 < \eta < 2$ and $c_D>0$, then by~\cite{MR3286462} $D_n$ lies in the universality class of the $\eta$-stable looptree $\mathcal{L}_\eta$:
\begin{align}
	\label{eq:d2}
 (c_D q_0 |\eta(1- \eta)|)^{1/\eta} 	\frac{1}{n^{1/\eta}} D_n \convdis \mathcal{L}_\eta
\end{align}
in the Gromov--Hausdorff--Prokhorov sense. Here and in the following we use a shortened notation, where the product of a real number and a random finite graph refers to the vertex set space with the corresponding rescaled graph metric and the uniform measure on the set of vertices. 

If $0<m_D < 1$ and $\Pr{\zeta=n} = f(n) n^{-\theta}$ for some constant $\theta > 2$ and a slowly varying function $f$, then $D_n$ lies in the universality class of a deterministic loop. That is,
\begin{align}
	\label{eq:d3}
	 \frac{q_0}{n(1 - \Ex{\zeta})} D_n \convdis S^1
\end{align}
with $S^1$ denoting the $1$-sphere. See for example~\cite{zbMATH07138334} for a detailed justification.

Knowing the asymptotic behaviour of $D_n$ allows us to describe the asymptotic shape of the random weighted outerplanar map $O_n$:

\begin{itemize}

\item If $m_O>1$, then by~\cite[Thm. 6.60]{zbMATH07235577} there is a constant $c_\omega > 0$ (defined in a complicated manner using expected distances in bi-pointed Boltzmann  dissections) such that
\begin{align}
	c_\omega \frac{1}{\sqrt{n}} O_n \convdis 	 \mathcal{T}_2.
\end{align}
In this case the asymptotics of weighted dissections only influence the constant $c_\omega$, but not the universality class.

\item If $m_O=1$ and $\Pr{\xi \ge  n} \sim c_O n^{-\alpha}$ for $1<\alpha<2$, then $T_n$ lies in the universality class of an $\alpha$-stable tree. That is, by~\cite[Thm. 6.1, Rem. 5.4]{MR2946438} there is a sequence 
\begin{align*}
	b_n = \left( \frac{|\Gamma(1-\alpha)|}{\Pr{\xi=0}} \right)^{1/\alpha}\inf \{x \ge 0 \mid \Pr{\xi > x} \le 1/n \}
\end{align*}
of order $n^{1/\alpha}$ for which
\begin{align*}
	(b_n^{-1} W_{\lfloor|T_n|t\rfloor}(T_n))_{0 \le t \le 1} \convdis  X^{\text{exc}, (\alpha)}.
\end{align*}

Now, the outcome of applying Theorem~\ref{th:invariance} depends on the decoration. Suppose that $\gamma > \alpha-1$. Then in each of the three discussed cases~\eqref{eq:d1},~\eqref{eq:d2}, and~\eqref{eq:d3} we obtain a scaling limit of the form
\begin{align}
	\frac{1}{b_n^\gamma} O_n \convdis \mathcal{T}_\alpha^\dec.
\end{align}
Specifically:
\subitem a)  If $m_D>1$ then $\gamma = 1/2$ and   $\mathcal{T}_\alpha^\dec$ is the $\alpha$-stable tree decorated according to a constant multiple of the Brownian tree $\mathcal{T}_2$ (with the constant given by the inverse of the scaling factor in~\eqref{eq:d1}). 
\subitem b) If $m_D = 1$ and $\Pr{\zeta \ge n} \sim c_D n^{-\eta}$ for $1 < \eta < 2$ and $c_D>0$, then $\gamma = 1/\eta$ and $\mathcal{T}_\alpha^\dec$ is the $\alpha$-stable tree decorated according to the stretched $\eta$-stable looptree  $(c_D q_0 |\eta(1- \eta)|)^{-1/\eta}  \mathcal{L}_\eta$.
\subitem c) If $0<m_D<1$ and $\Pr{\zeta=n} = f(n) n^{-\theta}$ for some constant $\theta > 2$ and a slowly varying function $f$, then $\gamma = 1$ and  $\mathcal{T}_\alpha^\dec$ is the $\alpha$-stable tree decorated according to the stretched circles $\frac{n(1 - \Ex{\zeta})}{q_0} S^1$. In other words, $\mathcal{T}_\alpha^\dec$ is distributed like  the stretched $\alpha$-stable loop tree $\frac{n(1 - \Ex{\zeta})}{q_0} \mathcal{L}_\alpha$. 

Here condition B1 (which is necessary for applying Theorem~\ref{th:invariance}) may be verified using Proposition~\ref{prop:small blobs don't contribute} and Remark~\ref{re:remarkleaves}. The necessary sufficiently high uniform integrability of the diameter of the decorations is trivial in case c), and follows from~\cite[Lem. 6.61, Sec. 6.1.3]{zbMATH07235577} in case a), and analogously from tail-bounds of conditioned BGW trees~\cite{kortchemski_sub_2017} in case b).

\item Finally, if $0<m_O < 1$ and $\Pr{\xi=n} = f_O(n) n^{-\theta_O}$ for some constant $\theta_O > 2$ and a slowly varying function $f$, then $O_n$ has giant $2$-connected component of size about $n (1- \Ex{\xi}) / \Pr{\xi=0}$.  Hence, if $D_n$ admits a scaling limit like in the three discussed cases, then the scaling limit for $M_n$ is, up to a constant multiplicative factor, the same as for $D_n$. See for example~\cite{zbMATH07138334} for details on such approximations. 

\end{itemize}

\subsection{Weighted planar maps with a boundary}

A planar map with a boundary refers to a planar map where the outer face typically plays a special role, and the perimeter (number of half-edges on the boundary of the outer face) serves as a size parameter. The reason for counting half-edges instead of edges is that both sides of an edge may be adjacent to the outer face, and in this case such an edge is counted twice. We say the boundary of a planar map is simple, if it is a cycle.  Here we explicitly allow the case of degenerate $2$-cycles and $1$-cycles. That is, a map consisting of two vertices joined by single edge has a simple boundary. A map consisting of a loop with additional structures on the inside has a simple boundary, and so has a map consisting of two vertices joined by two edges and additional structures on the inside.

Planar maps with a boundary fit into the framework of decorated trees by identical arguments as for outerplanar maps. That is, we may decompose a planar map into its components with a simple boundary in the same way as an outerplanar map may be decomposed into dissections of polygons. They may be bijectively encoded as trees decorated by maps with a simple boundary in the same way as illustrated in Figure~\ref{f:outerdec}. Since we allow multi-edges and loops, the leaves of the encoding tree canonically and bijectively correspond to the half-edges on the boundary of the map.

There are infinitely many planar maps with a fixed positive perimeter, hence it makes sense to assign summable weights. Say, for each planar map $D$ with a simple boundary we are given a weight $\gamma(D) \ge 0$. Like in Equation~\eqref{eq:omegao}, we then define the weight $\omega(M)$ of a planar map $M$ by 
\begin{align}
	\label{eq:omegam}
	\omega(M) = \prod_{D} \gamma(D),
\end{align}
with the index $D$ ranging over the components of $M$ with a simple boundary.
Thus, for any positive integer $n$ for which the sum of all $\omega$-weights of $n$-perimeter maps is finite and non-zero we may define the random $n$-perimeter map $M_n$ that is drawn with probability proportional to its weight. Note that this formally encompasses the random outerplanar map $O_n$, for which $\gamma(D) = 0$ whenever $D$ is not a dissection of a polygon of perimeter at least $3$.

Using the same definitions~\eqref{eq:defd}--\eqref{eq:defxi} for $m_O$ and $\xi$, it follows that the tree $T_n$ corresponding to the random map $M_n$ is distributed like the result of conditioning a $\xi$-BGW tree $T_n$ on having $n$ leaves.  Furthermore, employing analogous definitions for the decorations $(\tilde B_k,\tilde d_k,\tilde\rho_k,\tilde\ell_k,\tilde \nu_k)_{k \ge 0}$, it follows like in Lemma~\ref{le:outerdeco} that $M_n$ with distances rescaled by some scaling sequence $a_n>0$ is a discrete decorated tree:
\begin{corollary}
	The decorated tree $(T_n^\dec,d_n^\dec,\nu_n^\dec)$ is distributed like the random weighted  map $M_n$ with perimeter $n$, equipped with the graph distance multiplied by $a_n$ and the counting measure on the non-root vertices.
\end{corollary}

The boundary of $M_n$ also fits into this framework  but its scaling limits have already been determined in pioneering works~\cite{Richier:2017,zbMATH07253709}, with the parameter $m_O$ and the tail of $\xi$ ultimately determining whether its a deterministic loop, a random loop tree, or the Brownian tree.  

In order to apply Theorem~\ref{th:invariance} to $M_n$ (as opposed to its boundary) we need to be able to look inside of the components, that is, we need a description of the asymptotic geometric shape of Boltzmann planar maps with a large simple boundary. However no such results appear to be known.

What is known by~\cite{Richier:2017} is that  so-called non-generic critical face weight sequences with parameter $\alpha' = ]1, 3/2[$ lie in a stable regime with $m_O = 1$ and $T_n^\dec$ in the universality class of an $\alpha$-stable tree for $\alpha = 1 / (\alpha' - 1/2)$. The boundary of $M_n$ lies in the universality class of an $\alpha$-stable looptree by~\cite[Thm. 1.1]{Richier:2017}. It is natural to wonder whether additional knowledge of Boltzmann planar maps with a simple boundary in this regime enables the application of Theorem~\ref{th:invariance}. This  motivates the following question:

\begin{question}
	Do decorated stable trees constructed in the present work arise as scaling limits of face-weighted Boltzmann planar maps with a boundary?
\end{question}

We note that there is a connection to the topic of stable planar maps arising as scaling limits (along subsequences) of Boltzmann planar maps without a boundary~\cite{zbMATH06469338,zbMATH06932734}. Roughly speaking,  Boltzmann planar maps with a large boundary are thought to describe the asymptotic geometric behaviour of macroscopic faces of Boltzmann planar maps without a boundary.

There are results for related  models of triangulations and quadrangulations with a simple boundary with an additional weight on the vertices~\cite{zbMATH07039779, zbMATH07343343}. Scaling limits for models with a non-simple boundary have  been determined for the special case of uniform quadrangulations with a boundary~\cite{MR3335010,zbMATH07212165} that are condition on both the number of vertices and the boundary length.

For Boltzmann triangulations  with the vertex weights as in~\cite{zbMATH07343343} we would expect to be in the regime  $m_O<1$ where the shape of the map is dominated by a unique macroscopic component with a simple boundary. However, we could introduce block-weights as before  in order to force the model into a stable regime. At least in principle, Theorem~\ref{th:invariance} then yields convergence towards a decorate stable tree obtained by blowing up the branchpoints of a stable tree by rescaled Brownian discs. Checking the requirements of  Theorem~\ref{th:invariance} (such as verifying convergence of the moments of the rescaled diameter of Boltzmann triangulations with a simple boundary) does not appear involve any mayor obstacles, but we did not go through the details. To do so, we would need to recall extensive combinatorial background, and this appears to be beyond the scope of the present work.



\appendix
\section{Appendix}
This appendix contains the proofs of two technical statements that are used in the paper, Proposition~\ref{prop:small blobs don't contribute} and Lemma~\ref{lem:moment measure discrete block implies B2 or B3}, which ensure that the main assumptions under which the rest of the paper is stated are satisfied for reasonable models of decorated BGW trees. 
Section~\ref{subsec:regularly varying functions and doa} recalls some general results about regularly varying functions and domain of attraction of stable random variables. Section~\ref{subsec:estimates for BGW} presents some estimates related to critical BGW trees whose reproduction law is in the domain of attraction of a stable law that are useful later on.
In Section~\ref{subsec:spine decomposition with marks}, we introduce the notion of trees with marks on the vertices and prove a spine decomposition result for those marked trees.
Finally in Section~\ref{subsec:small blobs do not contribute} we prove Proposition~\ref{prop:small blobs don't contribute}, which is the main technical result of this Appendix, using all the results and estimates derived before.
At the end, in Section~\ref{subsec:proof of lemma about measure}, which is independent of what comes before, we prove Lemma~\ref{lem:moment measure discrete block implies B2 or B3}. 

\subsection{Regularly varying functions and domains of attraction}\label{subsec:regularly varying functions and doa}
\subsubsection{Compositional inverse of regularly varying functions}
Following standard terminology, we say that a function $f$ defined on a neighbourhood of infinity is regularly varying (at infinity) with exponent $\alpha\in \R$ if for every $\lambda\in \R\setminus \{0\}$ we have
\begin{align*}
	\frac{f(\lambda x)}{f(x)}\rightarrow \lambda^\alpha.
\end{align*} 
We consider those functions up to the equivalence relation of having a ratio tending to $1$ at infinity. When the index of regularity is positive $\alpha >0$, a regularly varying function $f$ with exponent $\alpha$ tends to infinity and we can define (at least at a neighbourhood of infinity):
\begin{align*}
	f^{[-1]}(x):=\inf\enstq{y\in\R_+}{f(y)\geq x},
\end{align*}
and the equivalence class of $f^{[-1]}$ only depends on the equivalence class of $f$. Then $f^{[-1]}$ is a regularly varying function with index $\alpha^{-1}$ and satisfies:
\begin{align*}
	f\circ f^{[-1]}(x)\underset{x \rightarrow\infty}{\sim} f^{[-1]}\circ f (x)\underset{x \rightarrow\infty}{\sim} x.
\end{align*}

\subsubsection{Asymmetric stable random variable}
For $\alpha\in\intervalleoo{0}{1}\cup \intervalleoo{1}{2}$, we let $Y_\alpha$ be a random variable with so-called asymmetric stable law of index $\alpha$, with distribution characterized by its Laplace transform, for all $\lambda>0$,
\begin{align*}
	\Ec{\exp\left(-\lambda Y_\alpha \right)}&= \exp\left(-\lambda ^\alpha \right) \quad \text{if}\quad 0<\alpha <1,\\
	&=\exp\left(\lambda ^\alpha \right) \quad \text{if}\quad 1<\alpha <2.
\end{align*}
It is ensured by \cite[Eq.(I20)]{zolotarev_one_1986} that those distributions have a density $d_\alpha$ and that this density is continuous and bounded. 

\paragraph{Domain of attraction.}
Let $X$ be a random variable such that $\Pp{\abs{X}>x}$ is regularly varying with index $-\alpha$, centred if $\alpha\in\intervalleoo{1}{2}$. 
Consider a sequence $X_1, X_2,\dots$ of i.i.d random variables with the law of $X$. Suppose that 
\begin{align*}
	\frac{\Pp{X>x}}{\Pp{\abs{X}>x}}\rightarrow 1.
\end{align*}
Then, let
\begin{align}\label{eq:regularly varying function associated to X}
	B_X(x):=\abs{\frac{1-\alpha}{\Gam{2-\alpha}}}^{-\frac{1}{\alpha}} \left(\frac{1}{\Pp{\abs{X}\geq x}}\right)^{[-1]}.
\end{align}
We also denote this function by $B_\nu$ if $\nu$ is the law of $X$.
Since $x\mapsto \frac{1}{\Pp{\abs{X}\geq x}} $ is $\alpha$-regularly varying, $B_X$ is $\alpha^{-1}$-regularly varying. Then, by \cite[Theorem~4.5.1]{whitt_stochastic_2002}, we have:
\begin{align*}
	\frac{1}{B_X(n)}\cdot \sum_{i=1}^{n}X_i \tend{n\rightarrow\infty}{(d)} Y_\alpha,
\end{align*}
where $Y_\alpha$ has the asymmetric $\alpha$-stable law, which we recall has a density $d_\alpha$.
In fact, if the random variable $X$ is integer-valued (and not supported on a non-trivial arithmetic progression), a more precise version of the above convergence known as \emph{local limit theorem}, is true, see \cite[Theorem~4.2.1]{MR0322926}. 
Let $S_n=\sum_{i=1}^n X_i$, then we have
\begin{align}\label{eq:stable local limit theorem}
	 \sup_{k\in \Z} \Big| B_X(n)\Pp{S_n=k} - d_\alpha\left(\frac{k}{B_X(n)}\right) \Big| \underset{n\rightarrow\infty}{\rightarrow} 0.
\end{align}
Using the fact that the asymmetric $\alpha$-stable density $d_\alpha$ is bounded, note that the above convergence ensures in particular that there exists a constant $C$ so that for all $n\geq 1$ and all $k\in \Z$ we have
\begin{align}\label{eq:uniform bound for the probability Sn=k}
	\Pp{S_n=k} \leq \frac{C}{B_X(n)}.
\end{align}

\subsubsection{The Potter bounds.}
From \cite[Theorem~1.5.6(iii)]{bingham_regular_1989}: for $f$ a regularly varying function of index $\rho$, then for every $A>1$ and $\epsilon$, there exists $B$ such that for all $x,y\geq B$ we have 
\begin{equation}
\frac{f(y)}{f(x)}\leq A \cdot  \max \left\lbrace\left(\frac{y}{x}\right)^{\rho-\epsilon},\left(\frac{y}{x}\right)^{\rho+\epsilon}\right\rbrace.
\end{equation}
When $f$ is defined on the whole interval $\intervalleoo{0}{\infty}$ and bounded below by a positive constant, we can increase the constant $A$ so that the last display holds for any $x,y\geq 1$. 

Let us apply this to $B_X$, the slowly varying function associated to some positive random variable $X$ in the domain of attraction of an $\theta$-stable law.
For all $\epsilon>0$ there exists a constant $C$ such that for all $n\geq 1$ for all $\frac{1}{B_X(n)}\leq \lambda \leq 1$,
\begin{align}\label{eq:potter bounds applied to B-1(lambda B_n)}
C^{-1}\cdot n\cdot \lambda^{\theta+\epsilon}\leq B_X^{[-1]}(\lambda  B_X(n)) \leq C\cdot n\cdot \lambda^{\theta-\epsilon},
\end{align}  
and so that in particular, possibly for another constant $C$,
\begin{align}
C^{-1}\cdot \frac{1}{n}\cdot \lambda^{-\theta+\epsilon}\leq \Pp{X\geq \lambda  B_X(n)} \leq C\cdot \frac{1}{n}\cdot \lambda^{-\theta-\epsilon},
\end{align}  

\subsection{Estimates for Bienaymé--Galton--Watson trees with $\alpha$-stable tails}\label{subsec:estimates for BGW}\label{subsec:the case of BGW}
Let $\mu$ be a critical reproduction law in the domain of attraction of an $\alpha$-stable distribution, with $\alpha\in \intervalleoo{1}{2}$.
Recall that $d_\alpha$ is the density function of the random variable $Y_\alpha$.
The following lemma contains all the results that we need in the subsequent sections of the appendix. 
\begin{lemma}\label{lem:application of the potter bounds}
Let $D$ be a random variable with distribution $\mu^*$ the size-biaised version of $\mu$ and $(\tau_i)_{i\geq 1}$ be independent BGW trees with reproduction law $\mu$. 
	Then, for any $\eta>0$, there exists a constant $C$ such that for all $n\geq 1$ and for $\lambda\in \intervalleff{\frac{1}{B_n}}{1}$, we have
	\begin{enumerate}[(i)]
		\item \label{it:probability of degree being large} $\displaystyle \Pp{D\geq \lambda B_n}\leq C \cdot \lambda^{1-\alpha-\eta} \cdot \frac{B_n}{n},$
		\item \label{it:probability total size of Bn GW is k} For all $k\in \mathbb Z$,\  \[\displaystyle\Pp{\sum_{i=1}^{\lambda B_n}\abs{\tau_i}=k} \leq C\cdot \frac{ \lambda^{-\alpha-\eta}}{n}.\]
		\item \label{it:expectation inverse total size of GW} $\displaystyle\Ec{\frac{1}{\sum_{i=1}^{\lambda B_n}\abs{\tau_i}}} \leq C\cdot \frac{ \lambda^{-\alpha-\eta}}{n},$
	\end{enumerate}
\end{lemma}

\begin{proof}[Proof of Lemma~\ref{lem:application of the potter bounds}]
	We first prove (\ref{it:probability of degree being large}).
	We have using \cite[Eq.(45)]{kortchemski_sub_2017}
	\begin{align*}
		\Pp{D\geq k} = \mu^*(\intervallefo{k}{\infty}) \underset{k\rightarrow\infty}{\sim} \frac{\alpha}{\alpha-1 }\cdot k \cdot \mu(\intervallefo{k}{\infty}),
	\end{align*}
	so for $C$ a constant chosen large enough we get that for all $k\geq 1$,
	\begin{align*}
		\Pp{D\geq k} \leq C\cdot k \cdot  \mu(\intervallefo{k}{\infty}).
	\end{align*}
	Applying this to $k=\lceil\lambda B_n\rceil$ and using the Potter bounds yields, for a value of $C$ that may change from line to line,
	\begin{align*}
		\Pp{D\geq \lambda B_n} \leq C\cdot \lceil\lambda B_n\rceil\cdot  \mu(\intervallefo{\lambda B_n}{\infty})
		&\leq C\cdot \lceil\lambda B_n\rceil\cdot  \frac{1}{n}\cdot \lambda^{-\alpha-\eta}\\
		&\leq C\cdot\frac{B_n}{n}\cdot \lambda^{1-\alpha-\eta}.
	\end{align*}
	So (\ref{it:probability of degree being large}) is proved. Now let us turn to (\ref{it:probability total size of Bn GW is k}).
	Using, for example, \cite[Eq.(26)]{kortchemski_sub_2017}, we have
	\begin{align}\label{eq:probability that a BGW has n vertices}
	\Pp{\abs{\tau}=n}\underset{n\rightarrow\infty}{\sim} \frac{d_\alpha(0)}{n B_\mu(n)}
\quad \text{  and  so } \quad
	\Pp{\abs{\tau}\geq n}\underset{n\rightarrow\infty}{\sim} \frac{\alpha d_\alpha(0)}{B_\mu(n)}.
	\end{align}
	This ensures that the random variable $\abs{\tau}$ is in the domain of attraction of the asymmetric $1/\alpha$-stable random variable.
	We can also check from the last display and the definition \eqref{eq:regularly varying function associated to X} that 
	\begin{align}\label{eq:relation Babstau to Bmu}
	B_{\abs{\tau}}(n)\sim C\cdot  B_\mu^{[-1]}(n),
	\end{align}
	for some constant $C$. Using \eqref{eq:uniform bound for the probability Sn=k} we then get 
	\begin{align*}
		\Pp{\sum_{i=1}^{\lambda B_n}\abs{\tau_i}=k} \leq \frac{C}{B_\mu^{[-1]}(\lambda B_n)} \leq C\cdot \frac{ \lambda^{-\alpha-\eta}}{n},
	\end{align*}
	where the last inequality follows from \ref{eq:potter bounds applied to B-1(lambda B_n)}. This finishes the proof of \ref{it:probability total size of Bn GW is k}.
The last point (\ref{it:expectation inverse total size of GW}) follows from an application of Lemma~\ref{lem:expectation of inverse of sum of asymmetric heavy tailed rv} stated below to the distribution of $\abs{\tau}$ using \eqref{eq:relation Babstau to Bmu} and an application of the Potter bounds to conclude. 
\end{proof}

\begin{lemma}\label{lem:expectation of inverse of sum of asymmetric heavy tailed rv}
	Let $X$ be a distribution on $\intervallefo{1}{\infty}$, in the domain of attraction of an asymmetric $\theta$-stable law, with $\theta\in\intervalleoo{0}{1}$, and $X_1, X_2,\dots$ i.i.d random variables with the law of $X$. Then, there exists a constant $C$ such that for all $N\geq 1$,
	\begin{align*}
		\Ec{\frac{1}{\sum_{i=1}^{N}X_i}} \leq \frac{C}{B_X(N)}.
	\end{align*}
\end{lemma}
\begin{proof}
	Write
	\begin{align*}
		\Ec{\frac{1}{\sum_{i=1}^{N}X_i}} = \frac{1}{B_X(N)} \cdot \Ec{\frac{B_X(N)}{\sum_{i=1}^{N}X_i}}\leq \frac{1}{B_X(N)} \cdot \Ec{\frac{B_X(N)}{\max_{1\leq i \leq N}X_i}}.
	\end{align*}
	Then 
	\begin{align*}
		\Ec{\frac{B_X(N)}{\max_{1\leq i \leq N}X_i}} &= \int_{0}^{\infty}\Pp{\frac{B_X(N)}{\max_{1\leq i \leq N}X_i}>x} \mathrm{d} x=\int_{0}^{\infty}\Pp{\max_{1\leq i \leq N}X_i<x^{-1}\cdot B_X(N)} \mathrm{d} x.
	\end{align*}
	Since the values of $X_i$ are at least $1$, the integrand of the last display becomes $0$ when $x>B_X(N)$.
	But when $x\leq B_X(N)$ we have
	\begin{align*}
		\Pp{\max_{1\leq i \leq N}X_i<x^{-1}\cdot B_X(N)} &\leq \Pp{X<x^{-1}\cdot B_X(N)}^N\\
		&\leq \left(1 - \Pp{X\geq x^{-1}B_X(N)}\right)^N\\
		&\leq \exp\left(-N\cdot\Pp{X\geq x^{-1}B_X(N)}\right)\\
		&\leq \exp\left(-N \cdot C\cdot \frac{x^{\theta/2}}{N}\right)\leq \exp\left(-C\cdot x^{\theta/2}\right),
	\end{align*}
	which is integrable over $\R$. In the last line, we use that
	\begin{equation*}
		\Pp{X\geq x^{-1}B_X(N)}\geq C \cdot \frac{1}{N}\cdot  (x^{-1})^{-\frac{\theta}{2}}
	\end{equation*}
	for some constant $C$ that does not depend on $N$, thanks to the Potter bounds with $\epsilon=\theta/2$, which apply here because $x^{-1}\geq \frac{1}{B_X(N)}.$
\end{proof}

\subsection{Spine decomposition for marked Bienaymé--Galton--Watson trees}\label{subsec:spine decomposition with marks}

\paragraph{Definitions.}
We first recall a few definitions from Section~\ref{subsec:plane trees}. 
We have $\UHT=\bigcup_{n\geq0} \N ^n$ the Ulam tree, and we denote $\bT$ the set of plane trees.
For any $\bt\in\bT$, and $v\in \bt $, we define $\theta_v(\bt)=\enstq{u\in\bU}{vu\in \bt}$. 
For $w\in \bU$, we let $w\bt = \enstq{wv}{v\in\bt}$.
If $\bt\in\bT$ and $u\in \bt$, then we define $\out{\bt}(u)$ to be the number of children of $u$ in $\bt$. 
We drop the index if it is clear from the context. 
If $v=v_1\dots v_n$ we say that $\abs{v}=n$ is the height of $v$, and for any $k\leq n$, we set $\left[v\right]_k=v_1\dots v_k$. 
If $v\neq\emptyset$ we also let $\hat{v}=\left[v\right]_{n-1}$ be the father of $v$.
We define $\bT^*=\enstq{(\bt,v)}{\bt\in\bT, \ v\in \bt}$ the set of plane trees with a marked vertex. 

Let $\mu$ be a probability measure on the integers, such that $\sum_{k\geq 0}k\mu(k)=1$. 
The associated probability measure $\BGW_\mu$ on plane trees is the measure such that for $T$ a random tree taken under this measure, we have for every $\bt\in\bT$, finite:
\begin{equation}
	\Pp{T=\bt}=\BGW_\mu(\{\bt\})=\prod_{u\in \bt}\mu(\out\bt(u)).
\end{equation}
\paragraph{Marks on a tree.}
Let $E$ be any measurable space. 
A \emph{finite tree with marks} is an ordered pair $\tilde{\bt}=(\bt,(x_u)_{u\in \bt})$ where $\bt\in\bT$ and $(x_u)_{u\in\bt}$ are elements of $E$ (corresponding to marks of the vertices). 
We denote by $\widetilde{\bT}$ the space of marked trees, and similarly, we let $\widetilde{\bT}^*$ the set of marked trees with a distinguished vertex. 

Let $(\pi_k)_{k\geq 0}$ be a sequence of probability measures on $E$. 
A natural way to put random marks on a tree $\bt$ is to mark it with $(X_u)_{u\in\bt}$ that is a family of independent random variables with respective distribution $X_u\sim \pi_{\out{\bt}(u)}$. 

\paragraph{Cutting a tree at a vertex.}
If $\bt\in\bT$, and $v\in \bt $, we define:
\begin{align}\label{eq:definition Cut}
	\Cut{\bt}{v}&:=\bt\setminus\enstq{vu}{u\in\bU,\ u\neq\emptyset}
	=\bt\setminus (v\theta_v \bt) \cup \{v\}.
\end{align}
\paragraph{The Kesten tree.}
We let 
\begin{itemize}
	\item $\left(D_i\right)_{i\geq 0}$ be a sequence of i.i.d. random variables with distribution $\mu^*$, where $\mu^*(k)=k\mu(k)$, for any $k\geq 0$,
	\item $(K_i)_{i\geq 1}$, which conditionally on the sequence $\left(D_i\right)$ are uniform on $\intervalleentier{1}{D_i}$,
	\item $\mathbf{U}_n^*=K_1K_2\dots K_n$.
	\item $\left(\tau_u\right)_{u\in\bU}$ are $\BGW$ trees with reproduction law $\mu$.
	\item $\cH=\enstq{\mathbf{U}_n^* i}{n\geq 0,\ i\in \intervalleentier{1}{D_i}\setminus K_i}$.
\end{itemize} 
Then the Kesten tree is defined as 
\begin{align*}
	\KT^*=\enstq{\mathbf{U}_n^*}{n\geq 0}\cup \bigcup_{u\in\cH}u\tau_u.
\end{align*}
\paragraph{Cut Kesten tree.}
Let $\tau$ be a BGW tree with offspring distribution $\mu$, independent of everything else. We let:
\begin{align*}
	\KT^*_k:=\Cut{\KT^*}{\mathbf{U}_k^*}\cup (\mathbf{U}_k^*\tau),
\end{align*}
which is almost surely a finite tree. 

As before, we endow the finite random tree $\KT^*_k$ with some marks $(X_u)_{u\in \KT^*_k}$ such that conditionally on $\KT^*_k$, the marks are independent with respective conditional distribution $X_u\sim \pi_{\out{}(u)}$. We denote by $\widetilde{\KT}^*_k$ the obtained tree with marks.
\paragraph{The ancestral line of a vertex.}
We let 
 \begin{align*}
	\bA:=\bigcup_{n\geq1}(\N_0\times E)^n.
\end{align*}
If $\tilde{\bt}=(\bt,(x_u)_{u\in\bt})\in\bT$, and $v=v_1v_2\dots v_k\in \bt$  we define:
\begin{align*}
	A(\tilde{\bt},v):=\left((\out{\bt}(\left[v\right]_0),x_{\left[v\right]_0}),(\out{\bt}(\left[v\right]_1),x_{\left[v\right]_1}),\dots,(\out{\bt}(\left[v\right]_{k}),x_{\left[v\right]_{k}})\right)\in \bA,
\end{align*}
the marked ancestral line of vertex $v$ in the marked tree $\tilde{\bt}$. 

\paragraph{BGW tree with a uniform vertex satisfying some property.}
Let $\mathscr{P}\subset\bA$ be some property. Consider $\widetilde{T}=(T,(X_u)_{u\in T})$ a marked BGW tree with reproduction distribution $\mu$ and mark distributions $(\pi_k)_{k\geq 0}$, and on the event $\enstq{\exists v \in T}{A(\widetilde{T},v)\in \mathscr{P}}$ where at least one vertex of $\widetilde{T}$ has the property $\mathscr{P}$ (which depends only on its ancestral line and the marks along it), we let $U$ be a uniform vertex such that this is the case.
\begin{proposition}\label{prop:spine decomposition with marks} 
	For any non-negative function $F:\widetilde{\bT}^*\rightarrow\R$, we have
	\begin{align}\label{eq:spine decomposition with marks}
		\Ec{F(\widetilde{T},U)\ind{\exists v \in T,\ A(\widetilde{T},v)\in \mathscr{P}}}=\sum_{k\geq 0}\Ec{\frac{F(\widetilde{\KT}^*_k,\mathbf{U}_k^*)\ind{A(\widetilde{\KT}^*_k,\mathbf{U}_k^*)\in \mathscr{P}}}{\#\enstq{v\in\KT^*_k}{A(\widetilde{\KT}^*_k,v)\in \mathscr{P}}}}.
	\end{align}
\end{proposition}

\begin{proof}
	For any $\bt$ and $u\in \bt$ of height $k$, we have 
	\begin{align}\label{eq:equality proba GW - biaised GW with vertex}
		&\Pp{(\KT^*_k,\mathbf{U}_k^*)=(\bt,u)}\notag\\
		&=\prod_{1\leq i \leq k}\Pp{D_i=\out{\bt}(\left[u\right]_{i-1})}\cdot \Ppsq{K_i=u_i}{D_i=\out{\bt}(\left[u\right]_{i-1})} \cdot \prod_{v\in \bt, v\nprec u}\mu(\out{\bt}(v))\notag\\
		&=\left(\prod_{1\leq i \leq k}\out{\bt}(\left[u\right]_{i-1})\mu(\out{\bt}(\left[u\right]_{i-1})  \frac{1}{\out{\bt}(\left[u\right]_{i-1})}\right) \cdot \prod_{v\in \bt, v\nprec u}\mu(\out{\bt}(v))\notag\\
		&=\prod_{v\in \bt}\mu(\out{\bt}(v))\notag\\
		&=\Pp{T=\bt}.
	\end{align}
	
	We now expand the expectation as a sum
	\begin{align*}
		\Ec{F(\widetilde{T},U)\ind{\exists v \in T,\ A(\widetilde{T},v)\in \mathscr{P}}}&=\sum_{\bt\in \bT, u\in \bt}\Ec{F(\widetilde{T},U)\ind{\exists v \in T,\ A(\widetilde{T},v)\in \mathscr{P}} \cdot \ind{T=\bt, U=u}}.
	\end{align*}
	For every $(\bt,u)$ with $\abs{u}=k$, we have
	\begin{align*}
		&\Ec{F(\widetilde{T},U)\ind{\exists v \in T,\ A(\widetilde{T},v)\in \mathscr{P}} \cdot \ind{T=\bt, U=u}}\\
		&= \Pp{T=\bt} \cdot \Ecsq{F(\widetilde{T},u) \cdot \ind{A(\widetilde{T},u)\in \mathscr{P}} \cdot \Ppsq{U=u}{\widetilde{T}}}{T=\bt}\\
		&= \Pp{T=\bt} \cdot \Ecsq{\frac{F(\widetilde{T},u) \cdot \ind{A(\widetilde{T},u)\in \mathscr{P}}}{\#\enstq{v\in \bt}{A(\widetilde{T},v)\in \mathscr{P}}}}{T=\bt}\\
		&=\Pp{(\KT^*_k,\mathbf{U}_k^*)=(\bt,u)}\cdot \Ecsq{\frac{F(\widetilde{\KT}^*_k,u) \cdot \ind{A(\widetilde{\KT}^*_k,u)\in \mathscr{P}}}{\#\enstq{v\in \bt}{A(\widetilde{\KT}^*_k,v)\in \mathscr{P}}}}{\KT^*_k=\bt}\\
		&=\Pp{(\KT^*_k,\mathbf{U}_k^*)=(\bt,u)}\cdot \Ecsq{\frac{F(\widetilde{\KT}^*_k,u) \cdot \ind{A(\widetilde{\KT}^*_k,u)\in \mathscr{P}}}{\#\enstq{v\in \bt}{A(\widetilde{\KT}^*_k,v)\in \mathscr{P}}}}{\KT^*_k=\bt,\mathbf{U}_k^*=u}\\
		&=\Ec{\frac{F(\widetilde{\KT}^*_k,\mathbf{U}_k^*) \cdot \ind{A(\widetilde{\KT}^*_k,\mathbf{U}_k^*)\in \mathscr{P}} \cdot \ind{(\KT^*_k,\mathbf{U}_k^*)=(\bt,u)}}{\#\enstq{v\in \bt}{A(\widetilde{\KT}^*_k,v)\in \mathscr{P}}}}.\\
	\end{align*}
	The third equality uses \eqref{eq:equality proba GW - biaised GW with vertex} and the fact that conditionally on the tree $T=\bt$ or $\KT^*_k=\bt$, the distribution of the marks on the tree is the same. 
	The fourth equality uses the fact that the quantity in the conditional expectation does not depend on $\mathbf{U}_k^*$.
	The result is then obtained by summing over all heights $k\geq 0$ and all finite trees $\bt$ with a distinguished vertex $u$ at height $k$. 
\end{proof}

\begin{remark}
	For our purposes, we will just use $E=\R_+$ but we want to emphasize that this result still holds true for other types of marks: the key point for the proof is that the law of the marks conditionally on the tree only depends on the degree of the corresponding vertex.
\end{remark}

%
\subsection{The contribution of vertices of small degree is small}\label{subsec:small blobs do not contribute}
The goal of this section is to prove Proposition~\ref{prop:small blobs don't contribute}. For that, we prove Proposition~\ref{prop:small marks don't contribute} which implies the former directly. 

Let $\alpha\in\intervalleoo{1}{2}$ and $\gamma>\alpha-1$. 
Let $\mu$ be a critical reproduction distribution in the domain of attraction of an $\alpha$-stable law. 
We let $b_n=B_\mu(n)$  be the $\frac{1}{\alpha}$-regularly varying function associated to $\mu$ by \eqref{eq:regularly varying function associated to X}. 
For all $n\geq 1$ for which the conditioning is non-degenerate, we let $T_n$ be a BGW tree with reproduction law $\mu$, conditioned to have exactly $n$ vertices.

This random tree is endowed with marks $(X_u)_{u\in T_n}$ such that conditionally on $T_n$, the marks are independent with distribution that only depends on the degree of the corresponding vertex $X_u\sim \pi_{\out{T_n}(u)}$, where the sequence $(\pi_k)_{k\geq 0}$ is a sequence of probability measures on $\R_+$.
We further require that for some real number $m> \frac{4\alpha}{ ( 2\gamma + 1 - \alpha)}$, we have
\begin{align*}
	\sup_{k\geq 0} \Ec{\left(\frac{X^{(k)}}{k^\gamma}\right)^m}<\infty,
\end{align*}
where for every $k\geq 0$, the random variable $X^{(k)}$ has distribution $\pi_k$. 
We prove the following.
\begin{proposition}\label{prop:small marks don't contribute}
In the setting decribed above, for any $\epsilon>0$ we have 
	\begin{align*}
		\Pp{\sup_{v\in T_n} \left\lbrace\sum_{u\preceq v} X_u \ind{\out{T_n}(u)\leq \delta b_n}\right\rbrace>\epsilon b_n^\gamma} \underset{\delta\rightarrow 0}{\rightarrow}0,
	\end{align*}
	uniformly in $n\geq 1$ such that $T_n$ is well-defined.
\end{proposition}
Remark that this proposition directly implies Proposition~\ref{prop:small blobs don't contribute} when applying it to marks distribution $(\pi_k)_{k\geq 0}$ that are respectively the laws of $(\diam(B_k))_{k\geq 0}$.
In order to prove this result, we first prove an intermediate lemma.
\begin{lemma}\label{lem:cutting slices}
	For any small enough $\eta>0$, $\epsilon>0$ and all $\delta\in \intervalleoo{0}{\epsilon^{\frac{1}{\eta}}}$, we have simultaneously for all $n\geq 1$, 
	\begin{align*}
		\Pp{\sup_{v\in T_n} \left\lbrace\sum_{u\preceq v} X_u \ind{\frac{\delta}{2}b_n<\out{T_n}(u)\leq \delta b_n}\right\rbrace>\epsilon b_n^\gamma} \leq \delta^{\beta},
	\end{align*}
	where $\beta=(\gamma +\frac{1-\alpha}{2}-5\eta)m-2\alpha-4\eta$, which is positive if $\eta$ is chosen small enough.
\end{lemma}
Let us show how the result we want follows from this lemma.
\begin{proof}[Proof of Proposition~\ref{prop:small marks don't contribute}]
	Let us take $\eta>0$ small enough such that Lemma~\ref{lem:cutting slices} holds. For any $\epsilon>0$ small enough, and $0<\delta<\epsilon^{\frac{1}{\eta}}$ we have
	\begin{align*}
		\sup_{v\in T_n} \left(\sum_{u\preceq v} X_u \ind{\out{T_n}(u)\leq \delta b_n}\right) &\leq  \sum_{i=0}^{\infty} \sup_{v\in T_n} \left\lbrace\sum_{u\preceq v} X_u \ind{\delta 2^{-i-1}b_n<\out{T_n}(u)\leq \delta 2^{-i} b_n}\right\rbrace.
	\end{align*}
	Write $\frac{\epsilon}{1-2^{-\eta}}=\sum_{i=0}^{\infty} \epsilon \cdot 2^{-i\eta}$ and use a union bound and the result of Lemma~\ref{lem:cutting slices} for all pairs $(\epsilon',\delta')\in \{(2^{-i\eta}\epsilon,2^{-i}\delta),\ i\geq 0\}$ which thanks to our assumption, still satisfy $\delta'<(\epsilon')^{\frac{1}{\eta}}$. This yields
	\begin{align*}	
		&\Pp{\sup_{v\in T_n} \left(\sum_{u\preceq v} X_u \ind{\out{T_n}(u)\leq \delta b_n}\right)>\frac{\epsilon}{1-2^{-\eta}}b_n^\gamma} \\
		&\leq  \sum_{i=0}^{\infty} \Pp{\sup_{v\in T_n} \left\lbrace\sum_{u\preceq v} X_u \ind{\delta 2^{-i-1}b_n<\out{T_n}(u)\leq \delta 2^{-i} b_n}\right\rbrace>2^{-i\eta}\cdot \epsilon \cdot b_n^\gamma}\\
		&\leq  \sum_{i=0}^{\infty} (2^{-i}\delta)^{\beta} \underset{\delta\rightarrow 0}{\longrightarrow} 0
	\end{align*}
	which is what we wanted to prove.
\end{proof}
The rest of the section is then devoted to the proof of Lemma~\ref{lem:cutting slices}.
\subsubsection{Proof of Lemma~\ref{lem:cutting slices}}
Now, let us turn to the proof of Lemma~\ref{lem:cutting slices}. For this one, we are going to use  Proposition~\ref{prop:spine decomposition with marks} with a certain property $\tilde{P}(\epsilon,\delta,n)$, which is defined as the subset of $\bA$ such that for any $\tilde\bt=(\bt,(x_u)_{u\in\bt})$ and $v\in\bt$,
\begin{align*}
	A(\tilde\bt,v)\in \tilde P(\epsilon,\delta,n) \quad \iff \quad  \sum_{u\preceq v} x_u \ind{\frac{\delta}{2}b_n<\out{\bt}(u)\leq \delta b_n}>\epsilon b_n^\gamma.
\end{align*}
Let $T_n\sim \BGW_\mu^n$. For a small $\eta>0$ we can write using a union bound
\begin{align}\label{eq:exponential decay of height of conditioned BGW}
	&\Pp{\sup_{v\in T_n} \left\lbrace\sum_{u\preceq v} X_u \ind{\frac{\delta}{2}b_n<\out{T_n}(u)\leq \delta b_n}\right\rbrace>\epsilon b_n^\gamma}\notag \\
	&\leq \Pp{\sup_{v\in T_n} \left\lbrace\sum_{u\preceq v} X_u \ind{\frac{\delta}{2}b_n<\out{T_n}(u)\leq \delta b_n}\right\rbrace>\epsilon b_n^\gamma \quad \text{and} \quad \haut(T_n)\leq \delta^{-\eta}\frac{n}{b_n}}+\Pp{\haut(T_n)>\delta^{-\eta}\frac{n}{b_n}}.
\end{align}
Thanks to \cite[Theorem~2]{kortchemski_sub_2017}, the second term is already smaller than some $C_1\exp(-C_2\delta^{-\eta})$, uniformly in $n$. 
We then just have to take care of the first term, for which we write
\begin{align}\label{eq:conditioned BGW to unconditioned}
	&\Pp{\sup_{v\in T_n} \left\lbrace\sum_{u\preceq v} X_u \ind{\frac{\delta}{2}b_n<\out{T_n}(u)\leq \delta b_n}\right\rbrace>\epsilon b_n^\gamma \quad \text{and} \quad \haut(T_n)\leq \delta^{-\eta}\frac{n}{b_n}}\notag\\
	&=\frac{1}{\Pp{\abs{T}=n}}\cdot \Pp{\abs{T}=n \quad \text{and} \quad \exists v\in T, \quad  \sum_{u\preceq v} X_u \ind{\frac{\delta}{2}b_n<\out{T_n}(u)\leq \delta b_n}>\epsilon b_n^\gamma \quad \text{and} \quad \haut(T)\leq \delta^{-\eta}\frac{n}{b_n}},
\end{align}
where here $T$ is an unconditioned BGW tree.
We already know from \eqref{eq:probability that a BGW has n vertices} that 
\begin{align*}
	\Pp{\abs{T}=n}\underset{n \rightarrow\infty}{\sim} \frac{d_\alpha(0)}{nb_n}.
\end{align*}
Using the spine decomposition Proposition~\ref{prop:spine decomposition with marks}, we can write the second factor as
\begin{align}\label{eq:spine decomposition applied to delta slice}
	&\Pp{\abs{T}=n, \text{ and } \exists v\in T,\ \sum_{u\preceq v} X_u \ind{\frac{\delta}{2}b_n<\out{T_n}(u)\leq \delta b_n}> \epsilon b_n^\gamma \text{  and  }\haut(v)\leq \delta^{-\eta}\frac{n}{b_n} } \notag \\
	&=\Ec{\sum_{k=0}^{\delta^{-\eta}\frac{n}{b_n}} \frac{\ind{\abs{\KT_k^*}=n}\cdot \ind{A(\widetilde\KT_k^*,\mathbf{U}_k^*)\in \tilde{P}(\epsilon,\delta,n)}}{\#\enstq{v\in \KT_k^*}{A(\widetilde\KT_k^*,v)\in \tilde{P}(\epsilon,\delta,n)}}}. 
\end{align}
Then we are going to upper-bound the last display by using the fact that from the definition of $\tilde{P}(\epsilon,\delta,n)$, for any $v\in \KT_k^*$ such that $A(\widetilde\KT_k^*,v)\in \tilde{P}(\epsilon,\delta,n)$, any descendant $u$ of $v$ satisfies $A(\widetilde\KT_k^*,u)\in \tilde{P}(\epsilon,\delta,n)$. We first introduce
\begin{align*}
	N(\epsilon,\delta,n):=\inf \enstq{i\geq 0}{\sum_{u\preceq \mathbf{U}_i^*} X_u \ind{\frac{\delta}{2}b_n<\out{\KT_k^*}(u)\leq \delta b_n}> \epsilon b_n^\gamma},
\end{align*}
which is the height of the first vertex on the spine that satisfies $\tilde{P}(\epsilon,\delta,n)$. 
This, in particular, entails that $\frac{\delta}{2}b_n \leq \out{\KT_k^*}(\mathbf{U}_{N(\epsilon,\delta,n)}^*) \leq \delta b_n$. 
For technical reasons that will be explained later, let us split the offspring of $u=\mathbf{U}_{N(\epsilon,\delta,n)}^*$ into two subsets whose sizes are of the same order: those written as $ui$ for $i\leq \frac{\delta}{4} b_n$ and those written as $ui$ for $i> \frac{\delta}{4} b_n$. 
We only consider the progeny of the former for now and use that to lower-bound the total number $\#\enstq{v\in \KT_k^*}{A(\widetilde\KT_k^*,v)\in \tilde{P}(\epsilon,\delta,n)}$. We then get
\begin{multline}\label{eq:upper bound using only a quarter of the progeny of the tip}
	\Ec{\sum_{k=0}^{\delta^{-\eta}\frac{n}{b_n}} \frac{\ind{\abs{\KT_k^*}=n}\cdot \ind{A(\widetilde\KT_k^*,\mathbf{U}_k^*)\in \tilde{P}(\epsilon,\delta,n)}}{\#\enstq{v\in \KT_k^*}{A(\widetilde\KT_k^*,v)\in \tilde{P}(\epsilon,\delta,n)}}}\\
	\leq \sum_{k=0}^{\delta^{-\eta}\frac{n}{b_n}} \Ec{\frac{\ind{\abs{\KT_k^*}=n}\cdot \ind{A(\KT_k^*,\mathbf{U}_k^*)\in \tilde{P}(\epsilon,\delta,n)}}{\#\enstq{v\in \KT_k^*}{v\succeq \mathbf{U}_{N(\epsilon,\delta,n)}^*i\text{ for some } i\leq \frac{\delta}{4}b_n}}}.
\end{multline}

Now let us state a lemma and conclude from there. We prove the lemma at the end of the section.
\begin{lemma}\label{lem:bounding every term in the sum of delta slices}
	For all $n$ and $k\leq \delta^{-\eta}\frac{n}{b_n}$, and $\delta<\epsilon^\frac{1}{\eta}$ small enough, there exists a constant $C$ such that
	\begin{enumerate}[(i)]
		\item\label{it:proba that size is n conditionally on rest of tree}
		\begin{align*}
			\Ppsq{\abs{\KT_k^*}=n}{\frac{\ind{A(\KT_k^*,\mathbf{U}_k^*)\in \tilde{P}(\epsilon,\delta,n)}}{\#\enstq{v\in \KT_k^*}{v\succeq \mathbf{U}_{N(\epsilon,\delta,n)}^*i\text{ for some } i\leq \frac{\delta}{4}b_n}}} \leq \frac{C\cdot \delta^{-\alpha-\eta}}{n},
		\end{align*}
		\item\label{it:proba that Uk satisfies property}
		\begin{align*}
			\Pp{A(\widetilde{\KT}_k^*,\mathbf{U}_k^*)\in \tilde{P}(\epsilon,\delta,n)} \leq \delta^{(\gamma +\frac{1-\alpha}{2}-5\eta)m},
		\end{align*}
		\item \label{it:expectation of the inverse of number of vertices above UN}
		\begin{align*}
			\Ecsq{\frac{1}{\#\enstq{v\in \KT_k^*}{v\succeq \mathbf{U}_{N(\epsilon,\delta,n)}^*i\text{ for } i\leq \frac{\delta}{4}b_n}}}{A(\KT_k^*,\mathbf{U}_k^*)\in \tilde{P}(\epsilon,\delta ,n)}\leq\frac{C \delta^{-\alpha-\eta}}{n}.
		\end{align*}
	\end{enumerate}
\end{lemma}

Using the above lemma to take care of every term in the sum appearing on the right-hand-side of \eqref{eq:upper bound using only a quarter of the progeny of the tip} and sum over $k$ to get the following
\begin{align*}
	\sum_{k=0}^{\delta^{-\eta}\frac{n}{b_n}}\Ec{ \frac{\ind{\abs{\KT_k^*}=n}\cdot \ind{A(\KT_k^*,\mathbf{U}_k^*)\in \tilde{P}(\epsilon,\delta ,n)}}{\#\enstq{v\in \KT_k^*}{v\succeq \mathbf{U}_{N(\epsilon,\delta,n)}^*}}}&\leq  \delta^{-\eta}\frac{n}{b_n} \cdot C \frac{\delta^{-\alpha-\eta}}{n} \cdot C \frac{\delta^{-\alpha-\eta}}{n}\cdot \delta^{(\gamma +\frac{1-\alpha}{2}-5\eta)m}\\
	&\leq C \cdot \frac{\delta^{(\gamma +\frac{1-\alpha}{2}-5\eta)m-2\alpha-3\eta}}{n b_n}.
\end{align*}
Plugging the last display into \eqref{eq:conditioned BGW to unconditioned} using the equality \eqref{eq:spine decomposition applied to delta slice} and using
 \eqref{eq:exponential decay of height of conditioned BGW} and \eqref{eq:probability that a BGW has n vertices}, we then get that for small enough $\delta$, for $\delta<\epsilon^{\frac{1}{\eta}}$ we have
\begin{align*}
	\Pp{\sup_{v\in T_n} \left\lbrace\sum_{u\preceq v} X_u \ind{\frac{\delta}{2}b_n<\out{T_n}(u)\leq \delta b_n}\right\rbrace>\epsilon b_n^\gamma} \leq \delta^{(\gamma +\frac{1-\alpha}{2}-5\eta)m-2\alpha-4\eta},
\end{align*}
which is what we wanted to prove.

\subsubsection{Proof of Lemma~\ref{lem:bounding every term in the sum of delta slices}}

Let us successively prove the three points of the lemma.
First, let us remark that if $\delta\leq \frac{8}{b_n}$, then (\ref{it:proba that size is n conditionally on rest of tree}) and (\ref{it:expectation of the inverse of number of vertices above UN}) are trivial.
Indeed, in that case the quantities on the left-hand-side are smaller than $1$.
On the other side, we can prove using the Potter bounds, that if we choose the constant large enough, the right-hand-side of that inequality is always greater than $1$ for $\delta$ in that range. 
Hence when proving (\ref{it:proba that size is n conditionally on rest of tree}) and (\ref{it:expectation of the inverse of number of vertices above UN}), we can always assume that $\delta\geq \frac{8}{b_n}$ so that $\frac{\delta}{4}b_n\geq 2$.

\subparagraph{Proof of (\ref{it:proba that size is n conditionally on rest of tree}).}
Let us work on the event $\left\lbrace A(\KT_k^*,\mathbf{U}_k^*)\in \tilde{P}(\epsilon,\delta ,n) \right \rbrace$. 
We denote by $v_1,\dots v_J$ the vertices of the form $\mathbf{U}_{N(\epsilon,\delta,n)}^*i$ for $i> \frac{\delta}{4}b_n$ that are not $\mathbf{U}_{N(\epsilon,\delta,n)+1}$. There is some number $J$ of them where $1\leq \frac{\delta}{4} b_n-1 \leq J \leq \delta b_n$.
For this proof, we define $\Cut{\KT_k^*}{v_1,v_2,\dots v_J}$ similarly as in \eqref{eq:definition Cut} except this time we remove all the vertices that are strictly above all every one of the vertices $v_1,v_2,\dots v_J$.

Now, note that the knowledge of the tree $\Cut{\KT_k^*}{v_1,v_2,\dots v_J}$ is enough to compute the quantity $\#\enstq{v\in \KT_k^*}{v\succeq \mathbf{U}_{N(\epsilon,\delta,n)}^*i \text{ for } i \leq \frac{\delta}{4}b_n}$, and that conditionally on $\Cut{\KT_k^*}{v_1,v_2,\dots v_J}$, the subtrees $\tau_1,\dots ,\tau_J$ respectively above $v_1,\dots v_J$ are i.i.d.\ $\BGW_\mu$-distributed random trees. 
Then the total size of the whole tree $\KT_k^*$ is exactly $n$ if and only if the total volume of those trees $\tau_1,\dots ,\tau_J$ is exactly $n$ minus the number of vertices in the rest of $\KT_k^*$. 
This (conditional) probability is bounded above as follows
\begin{align*}
&\Ppsq{\abs{\KT_k^*}=n}{\frac{\ind{A(\KT_k^*,\mathbf{U}_k^*)\in \tilde{P}(\epsilon,\delta,n)}}{\#\enstq{v\in \KT_k^*}{v\succeq \mathbf{U}_{N(\epsilon,\delta,n)}^*i\text{ for some } i\leq \frac{\delta}{4}b_n}}}\\
	&= \Ecsq{\Ppsq{\abs{\KT^*_k}=n}{\Cut{\KT_k^*}{v_1,v_2,\dots v_J}}}{\frac{\ind{A(\KT_k^*,\mathbf{U}_k^*)\in \tilde{P}(\epsilon,\delta,n)}}{\#\enstq{v\in \KT_k^*}{v\succeq \mathbf{U}_{N(\epsilon,\delta,n)}^*i\text{ for some } i\leq \frac{\delta}{4}b_n}}}\\
	&= \Ecsq{\Ppsq{\sum_{j=1}^J \abs{\tau_j}=n+J-\abs{\KT^*_k}}{\Cut{\KT_k^*}{v_1,v_2,\dots v_J}}}{\frac{\ind{A(\KT_k^*,\mathbf{U}_k^*)\in \tilde{P}(\epsilon,\delta,n)}}{\#\enstq{v\in \KT_k^*}{v\succeq \mathbf{U}_{N(\epsilon,\delta,n)}^*i\text{ for some } i\leq \frac{\delta}{4}b_n}}}\\
	&\leq \frac{C\cdot \delta^{-\alpha-\eta}}{n},
\end{align*}
using Lemma~\ref{lem:application of the potter bounds}(\ref{it:probability total size of Bn GW is k}) and the fact that by our assumptions we have $J\geq \frac{\delta}{4}b_n-1\geq 1$.

\subparagraph{Proof of (\ref{it:proba that Uk satisfies property}).}
For any $k\geq 1$, let us write $(X_i)_{0\leq i\leq k}=(X_{\mathbf{U}_i^*})$ to simplify the notation. We can write, using that $\epsilon\geq \delta^{\eta}$,
\begin{align}\label{eq:proba that Uk has the property is split in two term}
	&\Pp{A(\widetilde{\KT}_k^*,\mathbf{U}_k^*)\in \tilde{P}(\epsilon,\delta,n)}\notag\\
	&\leq \Pp{\sum_{i=1}^{k}X_i \ind{\frac{\delta}{2}b_n< D_i\leq \delta b_n}>\epsilon b_n^\gamma}\notag\\
	&\leq \Ppsq{\sum_{i=1}^{k}\frac{X_i}{(\delta b_n)^\gamma}\ind{\frac{\delta }{2}b_n< D_i\leq \delta  b_n}>\delta^{-\gamma+\eta}}{\sum_{i\leq k}\ind{\frac{\delta }{2}b_n< D_i\leq \delta  b_n}\leq \delta^{1-\alpha -3\eta}}+\Pp{\sum_{i\leq k}\ind{\frac{\delta }{2}b_n< D_i\leq \delta  b_n}> \delta^{1-\alpha -3\eta}}.
\end{align}
The second term of the last display is always smaller than what we would get if we take the maximum value for $k$, i.e. $\delta^{-\eta}\frac{n}{b_n}$. Using Lemma~\ref{lem:application of the potter bounds}(\ref{it:probability of degree being large}) we have  for all $n\geq 1$,
\[\Pp{\frac{\delta}{2}b_n\leq D_i \leq \delta  b_n)}\leq C\cdot\delta ^{1-\alpha-\eta}\cdot \frac{b_n}{n}.\]
Hence
\begin{align*}
	\Pp{\sum_{i\leq k}\ind{\frac{\delta }{2}b_n< D_i\leq \delta  b_n}> \delta^{1-\alpha -3\eta}}
	&\leq \Pp{\mathrm{Bin}\left(\delta^{-\eta} \frac{n}{b_n},1\wedge (C\cdot \delta^{1-\alpha-\eta}\cdot \frac{b_n}{n}) \right)>\delta^{1-\alpha -3\eta}},
\end{align*}
which decays exponentially in a negative power of $\delta$. 
For the first term of \eqref{eq:proba that Uk has the property is split in two term}, we use the fact that conditionally on the event $\{\sum_{i\leq k}\ind{\frac{\delta }{2}b_n< D_i\leq \delta  b_n}\leq \delta^{1-\alpha -3\eta}\}$, all the random variables $\frac{X_i}{(\delta b_n)^\gamma}$ are independent with $m$-th moment bounded above uniformly by the same constant $C$,
\begin{align*}
	\sup_{0\leq i\leq k} \Ec{\left(\frac{X_i}{(\delta b_n)^\gamma}\right)^m}<C.
\end{align*}
Note that we have (for another constant $C$)
\begin{align}\label{eq:upper-bound sum expectation Xi}
	\sum_{i=1}^{k}\frac{\Ec{X_i}}{(\delta b_n)^\gamma}\ind{\frac{\delta }{2}b_n< D_i\leq \delta  b_n}\leq C \cdot \sum_{i\leq k}\ind{\frac{\delta }{2}b_n< D_i\leq \delta  b_n}.
\end{align}
We use Markov's inequality and a result of Petrov \cite[Chapter~III. Result 5.16]{MR0388499}, which applies thanks to the fact that $m\geq 2$.
In the third line, we will use that $\gamma>\alpha-1$ so that for $\delta$ small enough we have $\delta^{-\gamma+\eta}-C\cdot \delta^{1-\alpha -3\eta}\geq \delta^{-\gamma+2\eta}$. 
\begin{align*}
	&\Ppsq{\sum_{i=1}^{k}\frac{X_i}{(\delta b_n)^\gamma}\ind{\frac{\delta }{2}b_n< D_i\leq \delta  b_n}>\delta^{-\gamma+\eta}}{\sum_{i\leq k}\ind{\frac{\delta }{2}b_n< D_i\leq \delta  b_n}\leq \delta^{1-\alpha -3\eta}}\\
	&\underset{\eqref{eq:upper-bound sum expectation Xi}}{\leq} \Ppsq{\sum_{i=1}^{k}\frac{X_i-\Ec{X_i}}{(\delta b_n)^\gamma}\ind{\frac{\delta }{2}b_n< D_i\leq \delta  b_n}>\delta^{-\gamma+\eta}-C\delta^{1-\alpha -3\eta}}{\sum_{i\leq k}\ind{\frac{\delta }{2}b_n< D_i\leq \delta  b_n}\leq \delta^{1-\alpha -3\eta}}\\
	&\underset{\text{Markov}}{\leq} \frac{\Ecsq{\left(\sum_{i=1}^{k}\frac{X_i-\Ec{X_i}}{(\delta b_n)^\gamma}\ind{\frac{\delta }{2}b_n< D_i\leq \delta  b_n}\right)^m}{\sum_{i\leq k}\ind{\frac{\delta }{2}b_n< D_i\leq \delta  b_n}\leq \delta^{1-\alpha -3\eta}}}{\left(\delta^{-\gamma+2\eta}\right)^m}\\
	&\underset{\text{Petrov }}{\leq} C\cdot \frac{(\delta^{1-\alpha -3\eta})^\frac{m}{2} }{\left(\delta^{-\gamma+2\eta}\right)^m}\\
	&\leq \delta^{(\gamma +\frac{1-\alpha}{2}-4\eta)m}.
\end{align*}
Taking the sum of the two terms in \eqref{eq:proba that Uk has the property is split in two term} ensures that for $\epsilon$ small enough and $\delta<\epsilon^{1/\eta}$,
\begin{align*}
	\Pp{A(\widetilde{\KT}_k^*,\mathbf{U}_k^*)\in \tilde{P}(\epsilon,\delta,n)} \leq \delta^{(\gamma +\frac{1-\alpha}{2}-5\eta)m}.
\end{align*}

\subparagraph{Proof of (\ref{it:expectation of the inverse of number of vertices above UN}).}
Now let us reason conditionally on the event $\{A(\widetilde{\KT}_k^*,\mathbf{U}_k^*)\in \tilde{P}(\epsilon,\delta,n)\}$. On that event the quantity $\#\enstq{v\in \KT_k^*}{v\succeq \mathbf{U}_{N(\epsilon,\delta,n)}^*i \text{ for } i\leq \frac{\delta}{4}b_n}$ is at least the sum of the total size of $\frac{\delta}{4}b_n-1$ independent $\BGW$ tree branching off of the spine.
Hence 
\begin{align*}
	&\Ecsq{\frac{1}{\#\enstq{v\in \KT_k^*}{v\succeq \mathbf{U}_{N(\epsilon,\delta,n)}^*i \text{ for } i\leq \frac{\delta}{4}b_n}}}{A(\KT_k^*,\mathbf{U}_k^*)\in \tilde{P}(\epsilon,\delta ,n)}\\
	&\leq \Ec{\frac{1}{\sum_{i=1}^{\frac{\delta}{4}b_n-1}\abs{\tau_i}}} \leq\frac{C \delta^{-\alpha-\eta}}{n},
\end{align*}
where the $\tau_i$'s are i.i.d.\ under distribution $\BGW_\mu$. The last inequality is obtained using Lemma~\ref{lem:application of the potter bounds}(\ref{it:expectation inverse total size of GW}).

\subsection{Proof of Lemma~\ref{lem:moment measure discrete block implies B2 or B3}} \label{subsec:proof of lemma about measure}

Before going into the proof of Lemma~\ref{lem:moment measure discrete block implies B2 or B3}, let us recall some general arguments concerning the \L ukasiewicz path of BGW trees conditional on their total size. 
For $T_n$ a BGW tree conditioned on having total size $n$, the law of the vector $(\out{T_n}(v_1)-1,\dots,\out{T_n}(v_n)-1; \nu_{v_1}(B_{v_1}), \dots ,\nu_{v_n}(B_{v_n}))$ can be described as
\begin{align*}
	\mathrm{Law}\left((X_1,\dots, X_n; Z_1, \dots ,Z_n) \ | \ \sum_{i=1}^nX_i=-1, \ \forall k\leq n-1, \sum_{i=1}^k X_i\geq 0\right),
\end{align*}
where the $(X_i,Z_i)_{1\leq i \leq n}$ are i.i.d. with the distribution of $(D-1, \nu_D(B_D))$ where $D$ follows the reproduction distribution. 

Now using the so-called Vervaat transform, the above law can also be expressed as
\begin{align*}
\mathrm{Law}\left((X_{U+1},\dots, X_{U+n}; Z_{U+1}, \dots ,Z_{U+n}) \ \Big| \ \sum_{i=1}^nX_i=-1\right),
\end{align*}
where $U$ is defined as $U:=\min\enstq{1\leq k \leq n}{\sum_{i=1}^k X_i = \min_{1\leq k \leq n} \sum_{i=1}^k X_i}$, and the indices in the last display are interpreted modulo $n$. 

In what follows, we will also use an argument of absolute continuity. For any bounded function $F$ we can write
\begin{align*}
	&\Ecsq{F((X_1,\dots, X_{\lfloor \frac{3n}{4}\rfloor }; Z_1, \dots ,Z_{\lfloor \frac{3n}{4}\rfloor }))}{\sum_{i=1}^n X_i=-1}\\
	&= \Ec{F((X_1,\dots, X_{\lfloor \frac{3n}{4}\rfloor }; Z_1, \dots ,Z_{\lfloor \frac{3n}{4}\rfloor })) \frac{\Ppsq{\sum_{i=\lfloor \frac{3n}{4}\rfloor+1}^n X_i=-1 -\sum_{i=1}^{\lfloor \frac{3n}{4}\rfloor} X_i }{\sum_{i=1}^{\lfloor \frac{3n}{4}\rfloor} X_i}}{\Pp{\sum_{i=1}^n X_i=-1}}}
\end{align*}
Using the local limit theorem \cite[Theorem~4.2.1]{MR0322926} and the fact that the $\alpha$-stable density is bounded \cite[Eq.(I20)]{zolotarev_one_1986}, the term
\begin{align*}
	\frac{\Ppsq{\sum_{i=\lfloor \frac{3n}{4}\rfloor+1}^n X_i=-1 -\sum_{i=1}^{\lfloor \frac{3n}{4}\rfloor} X_i }{\sum_{i=1}^{\lfloor \frac{3n}{4}\rfloor} X_i}}{\Pp{\sum_{i=1}^n X_i=-1}}
\end{align*} 
appearing in the integral is bounded uniformly in $n\geq 4$. 
Using indicator functions for $F$, this ensures that any event that occurs with probability tending to $0$ or $1$ as $n\rightarrow \infty$ for the unconditioned model, also tend to $0$ or $1$ for the conditioned model. 
Note that we can apply the same argument for the vector $(X_{\lfloor \frac{n}{4}\rfloor },\dots, X_{n}; Z_{\lfloor \frac{n}{4}\rfloor}, \dots ,Z_{n}))$. 

\begin{proof}[Proof of Lemma~\ref{lem:moment measure discrete block implies B2 or B3}]
We consider the case $\beta \leq \alpha$ and the case $\beta > \alpha$ in turn. 
	
$\bullet$ Case $\beta \leq \alpha$. In this case, we want to show that there exists a sequence $(a_n)_{n\geq 1}$ so that 
$\frac{1}{a_n}\cdot (M_{\lfloor nt \rfloor})_{0\leq t\leq 1} \underset{n \rightarrow \infty}{\rightarrow} (t)_{0\leq t \leq 1}$ in probability for the uniform topology. 
First, let us show that it is enough to check this if we replace $(M_k)_{1\leq k \leq n}$ with $(S_k)_{1\leq k \leq n}$ where $S_k:=\sum_{i=1}^kZ_i$.
Indeed, if 
\begin{align}\label{eq:Sk converges in probability}
\frac{1}{a_n}\cdot (S_{\lfloor nt \rfloor})_{0\leq t\leq 1} \underset{n \rightarrow \infty}{\rightarrow} (t)_{0\leq t \leq 1}
\end{align}
in probability, 
then we have 
\begin{align*}
	\frac{1}{a_n}\cdot (S_{\lfloor nt \rfloor})_{0\leq t\leq \frac{3}{4}} \underset{n \rightarrow \infty}{\rightarrow} (t)_{0\leq t \leq \frac{3}{4}} \quad  \text{and} \quad \frac{1}{a_n}\cdot (S_{\lfloor nt \rfloor} - S_{\lfloor \frac{3}{4}n \rfloor})_{\frac14\leq t\leq 1} \underset{n \rightarrow \infty}{\rightarrow} (t-\frac14)_{\frac14\leq t \leq 1}
\end{align*}
also in probability. 
By the absolute continuity argument, both those convergences also hold conditional on the event $\{\sum_{i=1}^nX_i=-1\}$, which ensures that $\frac{1}{a_n}\cdot (S_{\lfloor nt \rfloor})_{0\leq t\leq 1} \underset{n \rightarrow \infty}{\rightarrow} (t)_{0\leq t \leq 1}$ under $\Ppsq{\cdot }{\sum_{i=1}^{n}X_i=-1}$. 
Then, this is enough to get that $\frac{1}{a_n}\cdot (S_{U+\lfloor nt \rfloor})_{0\leq t\leq 1} \underset{n \rightarrow \infty}{\rightarrow} (t)_{0\leq t \leq 1}$ as well. 
Since $(M_k)_{1\leq k \leq n}$ is distributed as $S_{U+\lfloor nt \rfloor}$ under $\Ppsq{\cdot }{\sum_{i=1}^{n}X_i=-1}$, we get $\frac{1}{a_n}\cdot (M_{\lfloor nt \rfloor})_{0\leq t\leq 1} \underset{n \rightarrow \infty}{\rightarrow} (t)_{0\leq t \leq 1}$ in probability as required.

In the end, we just have to find a sequence $a_n$ such that \eqref{eq:Sk converges in probability} holds. 
In fact, by monotonicity, we just need to check that the convergence holds for $t=1$.
If $\Ec{Z}=\Ec{\nu_D(B_D)}$ is finite then it is easy to check that we can take $a_n:= \Ec{\nu_D(B_D)} \cdot n$. This is always the case if $\beta <\alpha$ since, using the assumption of the lemma
\begin{align*}
	\Ec{\nu_D(B_D)}&= \sum_{k=0}^{\infty}\Pp{D=k}\Ec{\nu_k(B_k)}\\
	&\underset{\text{assum.}}{\leq} C \cdot \sum_{k=0}^{\infty}\Pp{D=k} k^\beta \leq C\cdot \Ec{D^\beta} < \infty,
\end{align*}
so we can apply the law of large numbers.

If $\Ec{Z}$ is infinite, which in our case can only happen if $\beta=\alpha$, we need to understand the tail behaviour of $Z$. 
We will show in that case that 
\begin{align}\label{eq:tail Z from tail D}
\P(Z\geq x) \sim \Pp{D \geq x^{\frac{1}{\alpha}}} \cdot \Ec{\nu(\cB)} \qquad \text{as}~x\to \infty,
\end{align}
which in particular entails that $Z$ has a regularly varying tail of index $-1$ as $x\to \infty$. 
Using then general results (e.g. \cite[Theorem~3.7.2]{MR2722836}) for sums of random variables with regularly varying tail yields the result.  

First, for a given integer $N$ we introduce $Y_N^+$ and $Y_N^-$ whose distribution are defined in such a way that for all $x\geq 0$,
\begin{align*}
 \Pp{Y_N^+\geq x}= \sup_{k\geq N} \Pp{\frac{\nu_k(B_k)}{k^\alpha} \geq x} \quad \text{  and  }  \quad \Pp{Y_N^-\geq x}= \inf_{k\geq N} \Pp{\frac{\nu_k(B_k)}{k^\alpha} \geq x}.
\end{align*}
Note that for both of them, we have the following bound thanks to our assumption and Markov's inequality
\begin{align*}
 \Pp{Y_N^-\geq x}\leq \Pp{Y_N^+\geq x} \leq \sup_{k\geq 1}\Pp{\frac{\nu_k(B_k)}{k^\alpha} \geq x} \leq \frac{\sup_{k\geq 1} \Ec{\left(\frac{\nu_k(B_k)}{k^\alpha}\right)^{1+\eta}}}{x^{1+\eta}}.
\end{align*}
Since the quantity in the last display is integrable in $x$, and since $\Pp{Y_N^\pm\geq x}\rightarrow \Pp{\nu(\cB)\geq x}$ as $N\rightarrow\infty$ thanks to assumption D\ref{c:GHPlimit}, we have $\Ec{Y_N^\pm}\rightarrow \Ec{\nu(\cB)}$ by dominated convergence.
Also note that we have $\Ec{\left(Y_N^\pm\right)^{1+\eta/2}}<\infty$.
In order to prove \eqref{eq:tail Z from tail D}, we then just need to show that for every $N$, we have
\begin{align*}
	\P(Z\geq x) \geq \Pp{D \geq x^{\frac{1}{\alpha}}} \cdot \Ec{Y_N^-} (1+o(1)) \quad \text{  and  } \quad \P(Z\geq x) \leq \Pp{D \geq x^{\frac{1}{\alpha}}} \cdot \Ec{Y_N^+} (1+o(1)).
\end{align*}

Let us fix $N$ and write
\begin{align*}
	\Pp{Z\geq x}= \Pp{\nu_D(B_D) \geq x} &= \Pp{\nu_D(B_D)\geq x, \ D\geq N} + \Pp{\nu_D(B_D) \geq x, \ D<N}\\
	&= \Pp{\nu_D(B_D)\geq x, \ D\geq N} + O(x^{-1-\eta}).
\end{align*}
%
Now with a random variable $Y_N^+$ that is independent of $D$, we can write
\begin{align*}
	\Pp{\nu_D(B_D)\geq x, \ D\geq N}&= \Ec{\Ppsq{ \frac{\nu_D(B_D)}{D^\alpha}\geq \frac{x}{D^\alpha}}{D} \ind{D\geq N}}\\
	&\leq \Ec{\Pp{Y_N^+\geq \frac{x}{D^\alpha}} \ind{D\geq N}}\\
	&= \Pp{Y_N^+D^\alpha \geq x, \ D>N}\\
	&=\Pp{Y_N^+ D^\alpha \geq x} - \Pp{Y_N^+ D^\alpha \geq x, \ D\leq N}\\
	&=\Pp{Y_N^+ D^\alpha \geq x} - O(x^{-1-\eta}). 
\end{align*}
Now using \cite[Proposition~1.1]{kasahara_note_2018} this last display is equivalent to $\Pp{D\geq x^{\frac{1}{\alpha}}} \cdot \Ec{Y_N^+}$. 
We also get the other inequality $\Pp{\nu_D(B_D) \geq x} \geq \Pp{D\geq x^{\frac{1}{\alpha}}} \cdot \Ec{Y_N^-}\cdot (1+o(1))$, which finishes the proof.

$\bullet$ Case $\beta > \alpha$. 
Since the arguments used here are quite similar to what we did above, let us just sketch the proof. 
Let $\epsilon >0$ and consider
\begin{align*}
	R(n,\delta):= b_n^{-\beta} \cdot \sum_{i=1}^{n} \nu_{v_i}(B_{v_i})\ind{\out{}(v_i)\leq \delta b_n}.
\end{align*}
Using the same type of arguments as above, we can first show that 
\begin{align*}
	 b_n^{-\beta} \cdot \sum_{i=1}^{\lfloor\frac{3n}{4}\rfloor} Z_i \ind{X_i+1\leq \delta b_n} \qquad \text{  and  } \qquad b_n^{-\beta} \cdot \sum_{i=\lfloor\frac{n}{4}\rfloor}^{n} Z_i \ind{X_i+1\leq \delta b_n}
\end{align*}
tend to $0$ in probability under the unconditioned measure, uniformly in $n\geq 4$. Then using the absolute continuity argument together with the equality in distribution we get that  
\begin{align}
	\limsup_{n\rightarrow \infty}\Pp{R(n,\delta)\geq \epsilon} \underset{\delta \rightarrow 0}{\rightarrow} 0,
\end{align}
This ensures that condition B\ref{cond:small decorations dont contribute to mass} is satisfied.
\end{proof}

\bibliographystyle{siam}
\bibliography{dtree}

\end{document}